\documentclass[12pt]{amsart}
\pdfoutput=1
\usepackage{microtype}
\usepackage{etex} 
\usepackage[active]{srcltx}
\usepackage{calc,amssymb,amsthm,amsmath,amscd, eucal,ulem}
\usepackage{alltt}
\usepackage{mathtools}
\usepackage{verbatim}
\usepackage{bbding}
\RequirePackage[dvipsnames,usenames]{color}

\input{mabliautoref.sty}
\input{xy}
\xyoption{all}
\input{kmacros3.sty}
\usepackage{tikz}
\usepackage{amsfonts, mathrsfs}
\usepackage[cal=boondox, calscaled=1.05]{mathalfa}
\usepackage{calligra}
\usepackage{amssymb}
\usepackage{stmaryrd} 
\usepackage{dcpic, pictexwd}
\usepackage[left=1in,top=1in,right=1in,bottom=1in]{geometry}
\usepackage{bm}
\usepackage{verbatim}
\usepackage{upgreek}
\usepackage{systeme}

\usepackage{marvosym}
\usepackage{hyperref}

\numberwithin{equation}{theorem}

\newcommand{\kay}{\mathcal{k}}
\newcommand{\el}{\mathcal{l}}

\renewcommand{\:}{\colon}
\newcommand{\eg}{{\itshape e.g.} }

\renewcommand{\m}{\mathfrak{m}}
\renewcommand{\n}{\mathfrak{n}}
\newcommand{\p}{\mathfrak{p}}
\newcommand{\q}{\mathfrak{q}}

\DeclareMathOperator{\Norm}{N}

\DeclareMathOperator{\val}{val}

\DeclareMathOperator{\colim}{colim}
\DeclareMathOperator{\eps}{\varepsilon}

\newcommand{\FEt}{\textsf{F\'{E}t}} 

\newcommand{\Gal}{\textnormal{Gal}}

\DeclareMathOperator{\DIV}{Div}
\DeclareMathOperator{\Cl}{Cl}

\usepackage{setspace}
\usepackage{hyperref}
\usepackage{enumerate}
\usepackage{graphicx}
\usepackage[all,cmtip]{xy}

\usepackage{verbatim}

\renewcommand{\sA}{\mathcal{A}}

\renewcommand{\sC}{\mathcal{C}}
\renewcommand{\sD}{\mathcal{D}}
\renewcommand{\sE}{\mathcal{E}}
\renewcommand{\sF}{\mathcal{F}}

\renewcommand{\sH}{\mathcal{H}}

\renewcommand{\sL}{\mathcal{L}}
\renewcommand{\sM}{\mathcal{M}}

\renewcommand{\sO}{\mathcal{O}}

\renewcommand{\sY}{\mathcal{Y}}

\renewcommand{\sp}{\mathcal{p}}

\begin{document}
\title{Tame fundamental groups of pure pairs and Abhyankar's lemma} 
\author[J.~Carvajal-Rojas]{Javier Carvajal-Rojas}
\address{\'Ecole Polytechnique F\'ed\'erale de Lausanne\\ SB MATH CAG\\MA C3 615 (B\^atiment MA)\\ Station 8 \\CH-1015 Lausanne\\Switzerland \newline\indent
Universidad de Costa Rica\\ Escuela de Matem\'atica\\ San Jos\'e 11501\\ Costa Rica}
\email{\href{mailto:javier.carvajalrojas@epfl.ch}{javier.carvajalrojas@epfl.ch}}
\author[A.~St\"abler]{Axel St\"abler}
\address{Universit\"at Leipzig\\ Mathematisches Institut\\ Augustusplatz 10\\
04109 Leipzig\\Germany} 
\email{\href{mailto:staebler@math.uni-leipzig.de}{staebler@math.uni-leipzig.de}}


\thanks{The first named author was supported in part by the NSF CAREER Grant DMS \#1252860/1501102 and by the ERC-STG \#804334. The second named author was supported in part by SFB-Transregio 45 Bonn-Essen-Mainz financed by Deutsche Forschungsgemeinschaft.}



\begin{abstract}
Let $(R,\fram, \kay)$ be a strictly local normal $\kay$-domain of positive characteristic and $P$ be a prime divisor on $X=\Spec R$. We study the Galois category of finite covers over $X$ that are at worst tamely ramified over $P$ in the sense of Grothendieck--Murre. Assuming that $(X,P)$ is a purely $F$-regular pair, our main result is that every Galois cover $f \: Y \to X$ in that Galois category satisfies that $\bigl(f^{-1}(P)\bigr)_{\mathrm{red}}$ is a prime divisor. We shall explain why this should be thought as a (partial) generalization of a classical theorem due to S.S.~Abhyankar regarding the \'etale-local structure of tamely ramified covers between normal schemes with respect to a divisor with normal crossings. Additionally, we investigate the formal consequences this result has on the structure of the fundamental group representing the Galois category. We also obtain a characteristic zero analog by reduction to positive characteristics following Bhatt--Gabber--Olsson's methods.
\end{abstract}
\maketitle

\tableofcontents
\section{Introduction} \label{sec.Intro}
In their former preprint \cite{CarvajalStablerFsignaturefinitemorphisms}, the authors studied the behavior of \emph{pure $F$-regularity} under finite covers; see \cite[\S4, \S5]{CarvajalStablerFsignaturefinitemorphisms}. In the present work, we shall deepen into the consequences of \cite[Theorems 4.8, 5.12]{CarvajalStablerFsignaturefinitemorphisms}, which explain the behavior of splitting primes, splitting ratios, and test ideals along closed subvarieties under finite covers. In the spirit of \cite{CarvajalSchwedeTuckerEtaleFundFsignature,CarvajalFiniteTorsors,JeffriesSmirnovTransformationRule}, we shall do so by studying the conditions they impose on the structure of covers over purely $F$-regular singularities that are at worst tamely ramified over the minimal center of $F$-purity divisor. To this end, consider the following setup.

\begin{setup} \label{setup2}
Let $(R, \fram, \kay, K)$ be a strictly local normal $\kay$-domain of (equi-)characteristic $p\geq 0$ and dimension $\geq 2$.\footnote{Let us recall that a strictly local ring is a henselian local ring with separably closed residue field.} Set $X \coloneqq \Spec R$, let $Z \subset X$ be a closed subscheme of codimension $\geq 2$, and set $X^{\circ} = X \smallsetminus Z$. Let $P$ be a prime divisor on $X$, whose restriction to $X^{\circ}$ we denote by $P$ as well. Consider the Galois category $\mathsf{Rev}^{P}(X^{\circ})$ of finite covers over $X^{\circ}$ that are at worst tamely ramified over $P$ and denote by $\pi_1^{\mathrm{t},P}(X^{\circ})$ the corresponding fundamental group; see \autoref{sec.preliminaresTameStuff}. 
\end{setup}
\begin{terminology}[local pure log pairs]
With notation as in \autoref{setup2}, we say that $(X,P)$ is a \emph{pure pair} if either $p > 0$ and $(X,P)$ is purely $F$-regular, or $p =0 $ and $(X,P+\Delta)$ is a purely log terminal pair for some (auxiliary) effective divisor $\Delta$ on $X$ with coefficients strictly less than $1$. See \autoref{section.preliminaries} for more details on these definitions. 
\end{terminology}

In positive characteristic, our main result is the following.

\begin{theoremA*}[\autoref{prop.OnlyOnePrime}, \autoref{pro.(a)HoldPLT}]
Work in \autoref{setup2}. If $(X,P)$ is a pure pair, then every connected cover $f\: Y^{\circ} \to X^{\circ}$ in $\mathsf{Rev}^{P}(X^{\circ})$ satisfies that $Q \coloneqq \bigl(f^{-1}(P)\bigr)_{\mathrm{red}}$ is a prime divisor on $Y^{\circ}$ and $(Y,Q)$ is a pure pair. 
\end{theoremA*}

The proof of this result in positive characteristic is inspired by our previous work \cite{CarvajalStablerFsignaturefinitemorphisms}. The analog in characteristic zero is well-known to experts; see \autoref{sec.CharZeroBusiness}. In a nutshell, we use \cite[Theorem 4.8]{CarvajalStablerFsignaturefinitemorphisms} and the symmetry induced by the Galois action to prove that there is only one point of $Y$ lying over the generic point of $P$. Indeed, any such a point must correspond to the splitting prime of the pair $(Y,Q)$. Then, one may use \cite[Theorem 5.12]{CarvajalStablerFsignaturefinitemorphisms} to prove that $(Y,Q)$ is a pure pair. In fact, one may do this quantitatively by means of the transformation rule for splitting ratios in \cite[Theorem 4.8]{CarvajalStablerFsignaturefinitemorphisms}. Theorem A in combination with finiteness of local fundamental groups \cite{CarvajalSchwedeTuckerEtaleFundFsignature,XuFinitenessOfFundGroups} have very strong consequences on the structure of $\pi_1^{\mathrm{t},P}(X^{\circ})$. Concretely:

\begin{theoremB*}[\autoref{thm.TameApplication}] Work in \autoref{setup2}. Suppose that $(X,P)$ is a pure pair of characteristic $p>0$. Then, there exists an exact sequence of topological groups
\[
 \hat{\Z}^{(p)} \to \pi_1^{\textnormal{t},P}(X^{\circ}) \to \pi^{P}_{\mathrm{1,\acute{e}t}}(X^{\circ}) \to 1,
\]
where $\pi^{P}_{\mathrm{1,\acute{e}t}}(X^{\circ})$ is the fundamental group corresponding to the Galois subcategory of covers that are \'etale over $P$. The group $\pi^{P}_{\mathrm{1,\acute{e}t}}(X^{\circ})$ is finite with order prime-to-$p$ and no more than  $\mathrm{min}\bigl\{1\big/r\big(R,P\big),1/s(R)\bigr\}$, where $r\big(R,P\big)$ is the splitting ratio of $(R,P)$ and $s(R)$ is the $F$-signature of $R$. Furthermore:
\begin{enumerate}
    \item The homomorphism $\hat{\Z}^{(p)} \to \pi_1^{\textnormal{t},P}(X^{\circ})$ is injective if the divisor class of $P$ is a prime-to-$p$ torsion element of $\Cl R $. In this case, the short exact sequence
    \[
     1 \to \hat{\Z}^{(p)} \to \pi_1^{\textnormal{t},P}(X^{\circ}) \to \pi^{P}_{\mathrm{1,\acute{e}t}}(X^{\circ}) \to 1
    \]
    splits (as topological groups) if and only if the divisor class of $P$ is trivial.
    
     \item If $P$ is a non-torsion element of $\Cl X$, we have a short exact sequence
     \[ 0 \to \varprojlim_{n \in N^P(X^{\circ})} \bZ/n\bZ \to \pi_1^{\textnormal{t},P}(X^{\circ}) \to \pi^{P}_{\mathrm{1,\acute{e}t}}(X^{\circ}) \to 1,\]
     where $N^P(X^\circ) \subset \bN$ is the set of prime-to-$p$ positive integers $n \in \bN$ for which there is a divisor $D$ on $X$ such that $P-n \cdot D \in \Cl X$ has prime-to-$p$ torsion and $D|_U$ is Cartier, where $U \coloneqq X^\circ \smallsetminus Z$. The sequence is split if and only if for every $n \in N^P(X^\circ)$ there are divisors $D_n$ with $D_n\vert_U$ Cartier such that $P = nD_n \in \Cl X$ which are compatible in the sense that $mD_{nm} = D_n$ in $\Cl X$ for all $n,m \in N^P(X^\circ)$
\end{enumerate}
\end{theoremB*}

\begin{remark}
By \cite[Corollary 1.2]{TaylorInversionOfAdjuntionFSignature}, we expect that $\mathrm{min}\bigl\{1\big/r\big(R,P\big),1/s(R)\bigr\} = 1/s(R)$ in Theorem B. Indeed, Taylor's result establishes that this is the case when $P$ has a prime-to-$p$ torsion divisor class. 
\end{remark}

Over the complex numbers, we obtain the following analog.
\begin{theoremC*}
[\autoref{thm.TameApplicationCharZero}]
Work in \autoref{setup2}. Suppose that $(X,P)$ is a pure pair over $\mathbb{C}$. Then, there is an exact sequence of topological groups
\[
\hat{\mathbb{Z}} \to \pi_1^{\mathrm{t}, P}(X^{\circ}) \to \pi^{P}_{\mathrm{1,\acute{e}t}}(X^{\circ}) \to 1, \]  where $\pi^{P}_{\mathrm{1,\acute{e}t}}(X^{\circ})$ is finite; it is the fundamental group corresponding to the Galois subcategory of covers which are \'etale over $P$.
Additionally:
\begin{enumerate}
    \item The homomorphism $\hat{\Z} \to \pi_1^{\textnormal{t},P}(X^{\circ})$ is injective if the divisor class of $P$ is a torsion element of $\Cl R$.  In this case, the short exact sequence
   \[
    1 \to \hat{\mathbb{Z}} \to \pi_1^{\mathrm{t}, P}(X^{\circ}) \to \pi^{P}_{\mathrm{1,\acute{e}t}}(X^{\circ}) \to 1, 
    \] 
    splits (as topological groups) if and only if the divisor class of $P$ is trivial.
    
     \item If $P$ is a non-torsion element of $\Cl X$, then we have a short exact sequence
     \[ 0 \to \bZ/n\bZ \to \pi_1^{\textnormal{t},P}(X^{\circ}) \to \pi^{P}_{\mathrm{1,\acute{e}t}}(X^{\circ}) \to 1.\]
     The sequence is split if and only if there is a divisor $D$ with $D\vert_U$ Cartier such that $P = nD \in \Cl X$
\end{enumerate}
\end{theoremC*}

We shall prove Theorem B and Theorem C as formal consequences of the following two statements. See \autoref{sect.MainFormal} and especially \autoref{thm.MainFormal} for further details.
\begin{itemize}
     \item Every connected cover $f\: Y^{\circ} \to X^{\circ}$ in $\mathsf{Rev}^{P}(X^{\circ})$ satisfies that $\bigl(f^{-1}(P)\bigr)_{\mathrm{red}}$ is a prime divisor on $Y^{\circ}$. 
    \item There exists a universal \'etale-over-$P$ cover $\tilde{X}^{\circ} \to X^{\circ}$.
\end{itemize}
In positive characteristic, we give direct proofs of these statements in \autoref{sect.TameFunpurelyFregularpair}.

\subsection{Abhyankar's lemma}
We briefly mention here why the results in this work should be thought of as partial generalizations to Abhyankar's lemma; see \autoref{sec.preliminaresTameStuff} for further details.

Abhyankar's lemma \cite[Expos\'e XIII, \S 5]{GrothendieckSGA} is a theorem on the local structure, from the point of view of the \'etale topology, of finite covers between normal integral schemes that are tamely ramified with respect to a divisor with normal crossings (on the base). It establishes that, locally in the \'etale topology, any such cover is a quotient of a (generalized) Kummer cover; see \cite{GrothendieckMurreTameFundamentalGroup}, \cite[\href{https://stacks.math.columbia.edu/tag/0EYG}{Tag 0EYG}]{stacks-project}. In a sense, Abhyankar's lemma is a \emph{purity} theorem for Kummer covers. Indeed, by definition and \autoref{theo.serretameiskummer}, a tamely ramified cover with respect to a divisor is a cover that is Kummer at the codimension-$1$ \'etale-germs. Assuming the divisor has normal crossings; which is a regularity condition, Abhyankar's lemma establishes that such a cover is Kummer at all \'etale germs.

Let us understand Abhyankar's lemma with a simple but already fundamental example. With notation as in \autoref{setup2}, assume that $R$ is regular (or just pure in the sense of \cite{CutkoskyPurity}) and $P=\Div f$. A finite cover $R \subset S$ with $S$ a normal local domain is tamely ramified with respect to $P$ if $R_f \subset S_f$ is \'etale and the generic field extension $K(S)/K(R)$ is tamely ramified with respect to the DVR $R_{(f)}$. An example of such an extension is a Kummer cover: $S=R[T]/(T^n-f)$ with $n$ prime to the characteristic. However, there may exist several non-Kummer tamely ramified covers; see \autoref{ex.TheCusp}. In general, the connected components of the pullback of a tamely ramified cover $R \subset S$ to $R_{(f)}^{\mathrm{sh}}$ must be Kummer covers and the converse holds provided that $R_f \subset S_f$ is \'etale; see \autoref{theo.serretameiskummer} and \autoref{lem.Lemma228GM71}. 

In the above setup, Abhyankar's lemma establishes that if $R/f$ is regular, then any Galois tamely ramified cover of $R$ with respect to the prime divisor $\Div f$ is necessarily Kummer. One may then wonder for what singularities of $R/f$ Abhyankar's lemma hold. We shall see that in the situation of Theorems B and C, if $R/f$ is either KLT in characteristic zero or strongly $F$-regular in positive characteristic, then the statement of Abhyankar's lemma hold. A simpler version of our partial generalization of Abhyankar's lemma is the following. For the more general statement see \autoref{lem.THELEMMAFORABH} (keeping in mind \autoref{ex.Smoothambient} and \autoref{exSommothAmbientCharZero}).

\begin{theoremD*}[{\autoref{lem.THELEMMAFORABH}, \autoref{cor.AbhyankarLemmaPureFRegPairs}, \autoref{cor.abyankarcharzero}}]
With notation as in \autoref{setup2}, suppose that $X$ is regular and $P = \Div f$. If $(X,P)$ is a pure pair, then every Galois tamely ramified cover over $X$ with respect to $P$ is of the form $\Spec R[T]/(T^n-f) \to X$ for $n$ prime to the characteristic.
\end{theoremD*}

\begin{convention}
If a scheme $X$ or ring $R$ is defined over $\mathbb{F}_p$, then we denote the $e$-th iterate of the Frobenius endomorphism by $F^e \: X \to X$, or by $R \to F^e_* R$. We use the shorthand notation $q \coloneqq p^e$ to denote the $e$-th power of the prime $p$, for instance $F^e\: r \mapsto r^q$. We assume all our schemes and rings to be locally noetherian. In positive characteristic we also assume that they are $F$-finite and hence excellent. 
\end{convention}

\subsection*{Acknowledgements} 
 We would like to thank Bhargav Bhatt, Manuel Blickle, Alessio Caminata, Elo\'isa~Grifo, Zsolt Patakfalvi, Karl Schwede, Anurag Singh, Ilya Smirnov, Roberto Svaldi, and Maciej Zdanowicz for very useful discussions and help throughout the preparation of this preprint. The authors are grateful to Karl~Schwede for very valuable comments on a draft of this preprint and for suggesting us the use of adjoint ideals to study the perseverance of pure $F$-regularity. The first named author commenced working on this project while in his last year of Graduate School at the Department of Mathematics of the University of Utah. He is greatly thankful for their hospitality and support. He is particularly grateful to his advisor Karl~Schwede for his guidance and generous support. We are also indebted to a diligent referee who pointed out several mistakes in an earlier version.
 
 \section{Preliminaries on pure log pairs} \label{section.preliminaries}
 
In this preliminary section, we review the definitions and main aspects of \emph{pure log pairs}. By this, we mean log pairs $(X,\Delta)$ that are purely $F$-regular if defined over a positive characteristic field, or purely log terminal if defined over a characteristic zero field.

\subsection{Pure $F$-regularity}
Consider $X = \Spec R$ where $R$ is an $F$-finite normal $\kay$-domain of positive characteristic $p$ and let $\sC$ be a Cartier algebra acting on $R$. See \cite[Section 2]{CarvajalStablerFsignaturefinitemorphisms} for the relevant notions of Cartier algebras and modules in the way we employ them here. Following \cite{SchwedeCentersOfFPurity}, a \emph{center of $F$-purity (or $F$-pure center) for $(R, \sC)$} is an integral closed subscheme $P = V(\mathfrak{p}) \subset X$ such that $\mathfrak{p}$ is a $\sC$-submodule of $R$. We say that $P$ is a \emph{minimal center of $F$-purity for $(R,\sC)$} if $\mathfrak{p}$ is a maximal proper $\sC$-submodule. Given a closed point $x \in \Spec R$, we call $P$ a \emph{minimal center of $F$-purity through $x$} if $x \in P$.

Following \cite[\S 3.1]{Smolkinphdthesis} and \cite[\S4]{SmolkinSubadditivity}, one defines $\uptau_{\p}(R,\sC)$ to be the smallest Cartier $\sC$-submodule of $R$ not contained in $\p$, which exists provided that $\sC_e(R) \not\subset \p$ for some $e>0$ (this condition is referred to as \emph{nondegeneracy}). See \cite{TakagiPLTAdjoint,TakagiHigherDimensionalAdjoint}, \cf \cite[\S 5.2]{CarvajalStablerFsignaturefinitemorphisms}. By \cite[Proposition 3.1.14]{Smolkinphdthesis}, we see that $P$ is a minimal center of $F$-purity for $(R,\sC)$ if and only if $\uptau_{\p}(R,\sC) + \p = R$. When $\uptau_{\p}(R,\sC) = R$, one says that \emph{$(R,\sC)$ is purely $F$-regular along $P$}. For the generalization to the case $\p$ is radical, see \cite[Lemma 5.11]{CarvajalStablerFsignaturefinitemorphisms}.

In the local case, minimal centers of $F$-purity exist, are unique, and admit a simpler description. Indeed, if $(R,\fram)$ is local then \emph{the} minimal $F$-pure center of $(R,\sC)$ is given by the closed subscheme cut out by the splitting prime $\sp(R,\sC)$; see \cite[Remark 4.4]{SchwedeCentersOfFPurity}. Further, we see that $\uptau_{\p}(R,\sC) = R$ if $\p=\sp(R,\sC)$. In other words, in the local case, $(R,\sC)$ is always purely $F$-regular along its (unique) minimal $F$-pure center. We are implicitly assuming that $\sp(R,\sC)$ is a proper ideal of $R$ (i.e. $(R,\sC)$ is $F$-pure).

Still assuming $R$ is local, let $P=V(\p) \subset X$ be the closed subscheme cut out by a prime ideal $\p \subset R$. Let $\sC_R^{[P]} \subset \sC_R$ be the Cartier subalgebra consisting of $P$-compatible $p^{-e}$-linear maps. In other words, $\varphi \in \sC_{e,R}$ belongs to $\sC_{e,R}^{[P]}$ if and only if $\varphi(F^e_* \p) \subset \p$. Since the splitting prime $\sp\big(R,\sC_R^{[P]}\big)$ is the unique largest prime ideal compatible with all the $p^{-e}$-linear maps in $\sC_R^{[P]}$, we have an inclusion 
\[
\p \subset \sp\Big(R,\sC_R^{[P]}\Big).
\]
This inclusion is an equality exactly when $P$ is the minimal $F$-pure center of $\sC_R^{[P]}$. In particular, we may say that \emph{$P$ is a minimal $F$-pure center of $X$} (with no explicit reference to a Cartier algebra) to say that it corresponds to the splitting prime of some Cartier algebra---necessarily $\sC_R^{[P]}$. In that case, $(R,\sC_R^{[P]})$ is purely $F$-regular along $P$.

In this paper, we are interested in minimal $F$-pure center divisors. In this case, we have:

\begin{proposition}
\label{pro.[P]=PforDivisors}
Let $R$ be a local normal $\kay$-domain with $P=V(\p)$ a prime divisor on $\Spec R$. Then, we have $\sC_R^{[P]} = \sC_R^P$, where $\sC_R^{P}$ is the Cartier algebra corresponding to the divisor $\Delta = P$; see \cite{SchwedeFAdjunction}.\footnote{It also coincides with $\sC_R^{\p}$ as in \cite[\S 3.3]{BlickleTestIdealsViaAlgebras}.} 
\end{proposition}
\begin{proof}
Observe that membership in these Cartier algebras can be checked (by localizing) at $\mathfrak{p}$, where these Cartier algebras are obviously the same.
\end{proof}

\begin{notation}
With notation as in \autoref{pro.[P]=PforDivisors}, we write $\sp(X,P)$ and $r(X,P)$ (or with $X$ replaced by $R$) to denote the splitting prime and splitting ratio of the pair $\bigl(R,\sC_R^{[P]}\bigr)$. Moreover, we shall write $\sC_R^P$ instead of $\sC_R^{[P]}$.
\end{notation}

\begin{definition}[Purely $F$-regular local pair] \label{def.PurelyFRegularPair}
With notation as in \autoref{pro.[P]=PforDivisors}, we say that the pair $(X,P)$ (or with $R$ in place of $X$) is \emph{purely $F$-regular} if $P$ is a minimal $F$-pure center prime divisor on $X$.\footnote{Note that this is called \emph{divisorially F-regular} in \cite{HaraWatanabeFRegFPure}. However, we use the \emph{purely F-regular} terminology to emphasize the connections with purely log terminal (PLT) singularities and avoid confusions with \emph{divisorially log terminal} singularities (DLT).}
\end{definition}

\begin{remark} \label{rem.PureDRegularPairAdjointIdeal}
With notation as in \autoref{pro.[P]=PforDivisors}, notice that $(X,P)$ is a purely $F$-regular pair if and only $\uptau_{\p}(R,P) = R$, i.e. if $(X,P)$ is purely $F$-regular along $P$.
\end{remark}

We observe that $X$ must have ``mild'' singularities to admit a purely $F$-regular divisor.

\begin{proposition}
\label{pro.XisSFR}
Let $(X,P)$ be a purely $F$-regular local pair, then $R$ (or $X$) is strongly $F$-regular (with respect to its full Cartier algebra $\sC_R$). More generally, if $A$ is a local domain with an action by some Cartier algebra $\sA \subset \sC_A$, and $C = V(\mathfrak{c})$ a minimal $F$-pure center prime divisor for $\sA^C\coloneqq \sA \cap \sC_A^{C}$ and $(A, \sA)$ is $F$-regular at the generic point of $C$, then $(A, \sA)$ is $F$-regular.
\end{proposition} 
\begin{proof}
Since $\sA^{C} \subset \sA$, we have that $\sp(A,\sA) \subset \sp\bigl(A,\sA^C\bigr) = \mathfrak{c}$,
where the equality follows from $\mathfrak{c}$ being a prime maximal center of $F$-purity. Since $\mathfrak{c}$ has height $1$, $\sp(\sA)$ is either $0$ or $\mathfrak{c}$. If $\sp(\sA)=0$, we are done. If $\sp(A, \sA)=\mathfrak{c}$, then $(A,\sA)$ is not $F$-regular at the generic point of $C$, contradicting our hypothesis. To see the first statement follows from the last one, notice that, since $R$ is normal, $(R,\sC_R)$ is $F$-regular at the generic point of $P$.
\end{proof}

\begin{remark} \label{rem.SFRAmbient}
In \autoref{pro.XisSFR}, the normality hypothesis on $R$ is necessary. Indeed, we may consider the Whitney's umbrella singularity as a counterexample; see \cite[\S4.3.2]{BlickleSchwedeTuckerFSigPairs1}.
\end{remark}

Finally, we recall the global-to-local passage for $F$-pure centers.

\begin{proposition}
\label{pro.globaltolocalFpurecenters}
Let $X$ be the spectrum of a normal $\kay$-domain and let $\sC$ be a Cartier algebra on $X$. Let $P=V(\p)$ be a minimal center of $F$-purity passing through a geometric point $\bar{x} \to X$, then $\p \cdot \sO_{X,\bar{x}}^{\mathrm{sh}}$ is the splitting prime of the Cartier $\sO_{X,\bar{x}}^{\mathrm{sh}}$-algebra $\sO_{X,\bar{x}}^{\mathrm{sh}} \otimes \sC$.
\end{proposition}
\begin{proof}
See \autoref{pro.splittingprimepreservedsh} in the appendix.
\end{proof}

\subsubsection{Some examples of purely $F$-regular pairs}
\label{sect.examplesofsplittingratios}
We provide next some examples of purely $F$-regular pairs. Our method to prove that a given pair is purely $F$-regular is the following. 

\begin{lemma} \label{pro.OurMethodtoProvePureFRegularity}Let $R$ be a normal local domain, $\mathfrak{p}\subset R$ be a prime ideal (not necessarily of height $1$), and set $P=V(\p)$. Then, $\mathfrak{p}$ is the splitting prime of $\sC_R^{[P]}$ if and only if $R/\mathfrak{p}$ is $F$-regular with respect to the induced action of $\sC^{[P]}_R$. In that case, the splitting ratio of $(R,P)$ is the $F$-signature of $R/\mathfrak{p}$ with respect to $\sC^{[P]}_R$.
\end{lemma}

\begin{proof}
Note that $\mathfrak{p}$ is the splitting prime of $\sC_R^{[P]}$ if and only if $R/\mathfrak{p}$ viewed as an $\sC_R^{[P]}$-module is simple. However, we may equivalently view $R/\mathfrak{p}$ as an $\overline{\sC^{[P]}_R}$-module (\cf \cite[discussion before Lemma 2.20]{BlickleTestIdealsViaAlgebras}). See \cite[Lemma 2.13]{BlickleSchwedeTuckerFSigPairs1}.
\end{proof}

We shall also need the following observation.

\begin{remark}
\label{remark.FrobeniustraceFinitealgebra}
Consider the category of finite type $\kay$-algebras for some $F$-finite field $\kay$. Fix an isomorphism $\lambda\: \kay \to F^! \kay$ with adjoint $\kappa\: F_\ast \kay \to \kay$. If we have two Cartier linear maps $\Phi, \Psi\: F_\ast^e R \to R$ for some finite type $\kay$-algebra $R$, then, by choosing a presentation $S = \kay[x_1, \ldots, x_n] \to R$ and via \cite{FedderFPureRat}, we reduce the problem of whether $\Phi = \Psi$ to a computation in the polynomial ring $S$. Indeed, note that $\lambda$ induces an isomorphism $f^! \kay = \omega_S \to F^! \omega_S$, where $f\: \Spec S \to \Spec \kay$ is the structural map. Identifying $\omega_S$ with $S$, we obtain an isomorphism $\Sigma\: S \to F^! S$. By \cite[Lemma 4.1]{staeblerunitftestmodules}, the adjoint of $\Sigma$ is given by 
\[
\xi x_1^{i_1} \cdots x_n^{i_n} \mapsto \kappa(\xi) x_1^{(i_1 +1)/q} \cdots x_n^{(i_n +1)/q}, 
\]
with the usual convention that $x_i^{a/b}$ is zero whenever the exponent is not an integer. Now, by adjunction $\Hom_R(F_\ast^e R, R) = \Hom_R(R, {F^e}^! R)$ and by our choice of isomorphism $\Sigma$, we have that $\Hom_R(R, {F^e}^! R) \cong \Hom_R(R, R)$, and, by making this identification, $\Sigma$ induces the identity so that the adjoint of $\Sigma$ is a generator of $\Hom_R(F^e_\ast R,R)$.

In this way, if we want to check that two Cartier linear maps of a finite type $\kay$-algebra $R$ coincide, we may reduce, via a choice of presentation and Fedder's crtierion, to a comparison of two Cartier linear maps in a polynomial ring. For those to coincide in turn, we choose any basis $B$ of $F_\ast \kay$ as a $\kay$-module and then just need to check that they agree on $b \cdot x_1^{i_1} \cdots x_n^{i_n}$, where $b \in B$ and $0 \leq i_j \leq q -1$.

This line of reasoning is also preserved if we pass to completions. Indeed, by \cite[\href{https://stacks.math.columbia.edu/tag/0394}{Lemma 0394}]{stacks-project}, we may identify ${(F_\ast R)}^\wedge$ with $F_\ast R^\wedge$. Since both are finite free modules, the claim is clear.
\end{remark}

\begin{example}[Purely $F$-regular pairs on a regular ambient] \label{ex.Smoothambient}
Let $R$  be a \emph{regular} local ring. Recall that regular local rings are UFD; see \cite[\href{https://stacks.math.columbia.edu/tag/0AG0}{Lemma 0AG0}]{stacks-project}. In particular, any prime divisor on $\Spec R$ is principal \cite[\S 19, Theorem 47]{MatsumuraCommutativeAlgebra}. Let $\p=(f)$ be a height-$1$ prime ideal of $R$ with corresponding prime divisor $P$. As an immediate application of \autoref{pro.OurMethodtoProvePureFRegularity} and Fedder's criterion \cite{FedderFPureRat}, we see that $(R,P)$ is purely $F$-regular if and only if $R/f$ is a strongly $F$-regular ring. Moreover, in this case, one has $r(R,P)=s(R/f)$.\end{example}

\begin{example}[Graded hypersurfaces]
\label{exa.weightedhypersurface}
Let $R = \kay \llbracket z, x_0,x_1,\ldots,x_d \rrbracket/(z^n-x_0 h)$ be a normal hypersurface over a perfect field $\kay$, where $h$ is an irreducible weighted polynomial in the variables $x_1,\ldots,x_d$; see \cite{SinghSpiroffDivisorClassGroups}. 
Then, $\Cl R \cong \bZ/n\bZ$ and the divisor class of $(z,h)$ is a generator for $\Cl R$; see \cite[Corollary 3.4]{SinghSpiroffDivisorClassGroups}. Letting $P$ be the prime divisor corresponding to $(z,h)$, we claim the following.  
\end{example}
\begin{claim}
The pair $(R,P)$ is purely $F$-regular if $A \coloneqq \kay\llbracket x_1,\ldots,x_d \rrbracket/h$ is strongly $F$-regular. In that case, $r(R,P) \geq s(A)/n$, and $r(R,P)=1/n$ if $A$ is regular.
\end{claim}
\begin{proof}[Proof of claim]
Let $S \coloneqq \kay \llbracket z,x_0,x_1,\ldots,x_d \rrbracket$ and $f \coloneqq z^n-x_0h$. By Fedder's criterion \cite{FedderFPureRat}, we have that $\sC_{e,R}$ is generated by the reduction of $\Phi^e \cdot f^{q-1} \in \sC_{e,S}$ to $R$, where $\Phi$ denotes the Frobenius trace on $S$. In other words, $\sC_{R} = \sC^{\phi}_R$ where $\phi \coloneqq \overline{\Phi \cdot f^{p-1}}$. Having \autoref{pro.[P]=PforDivisors} in mind, we recall that $\sC^{P}_R$ is given, in degree $e$, by all maps $\phi^e \cdot g$ such that $\val_{\p} g \geq q-1$. Note that $z$ is a uniformizer for $R_{\p}$ and $\val_{\p} h = n$. In particular, $\sC^{P}_R$ contains the maps $\phi^e \cdot z^ih^j$ where $i+nj=q-1$. Therefore, the reduction of $\sC_R^{P}$ to $R/\p \cong \kay \llbracket x_0, x_1,\ldots,x_d \rrbracket/h$ contains the maps $\overline{\phi^e \cdot z^ih^j}$ with $i+nj=q-1$. However, these maps are the reductions of $\Phi^e \cdot z^ih^j f^{q-1} $ to $R/\p=S/(z,h)$, and we have that
\[
z^i h^j f^{q-1} \equiv (-1)^{q-1-j} \binom{q-1}{j} z^{q-1} x_0^{q-1-j} h^{q-1}  \bmod{\bigl( z^q , h^q \bigr)} 
\]
for all $i+nj=q-1$. In other words, the reduction of $\sC_R^{P}$ to $R/\p \cong \kay \llbracket x_0, x_1,\ldots,x_d \rrbracket/h \cong S/(z,h)$ contains, in degree $e$, the reductions of $\Phi^e \cdot z^{q-1} x_0^{q-1-j}h^{q-1}$ for all $j \leq (q-1)/n$. Alternatively, if $\Psi$ is the Frobenius trace for $\kay \llbracket x_0, x_1,\ldots,x_d \rrbracket$, we have that the reduction of $\sC_R^{P}$ to $\kay \llbracket x_0, x_1,\ldots,x_d \rrbracket / h$ contains the reductions of $\Psi^e \cdot  x_0^{q-1-j}h^{q-1}$. Therefore, we have:
\[
s\Big(R/\mathfrak{p}, \overline{\sC_R^{P}} \Big) \geq \frac{1}{n} \cdot s(A).\footnote{To see this, note that we may work in the polynomial case as completions have no bearing on the value of $F$-signatures; see \cite{YaoObservationsAboutTheFSignature} or \cite[\S 3]{CarvajalSchwedeTuckerBertiniFSignature}. Then, the result follows from the behavior of $F$-signatures with respect to tensor products; see \cite[Proposition 5.5]{CarvajalSmolkinUSTPandDFR} for instance.}
\]
Consequently, by applying \autoref{pro.OurMethodtoProvePureFRegularity}, we conclude that  $\mathfrak{p} = \sp(R,\sC_R^P)$, and furthermore
\[
r(R,P)=s\Big(R/\mathfrak{p}, \overline{\sC_R^{P}} \Big) \geq \frac{1}{n} \cdot s(A).
\]
To see this is an equality if $A$ is regular, we may use the transformation rule for splitting ratios \cite[Theorem 4.8]{CarvajalStablerFsignaturefinitemorphisms}. Indeed, suppose for sake of contradiction that the inequality is strict, and let $\tilde{R}$ be the Veronese-type cyclic cover given by $P$. That is, $\tilde{R}=\bigoplus_{i=0}^{n-1} \p^{(i)}$. It is not difficult to see that $\tilde{R}=R\bigl[x_0^{1/n},h^{1/n}\bigr]$ and the only prime in $\tilde{R}$ lying over $\p$ is $\tilde{\p}=\bigl(h^{1/n}\bigr)$; whose corresponding prime divisor we denote by $\tilde{P}$. Therefore, $\tilde{\p}$ must be the splitting prime of the pullback of $\sC^{P}_R$ along the cover $R \subset \tilde{R}$. Hence, the transformation rule for splitting ratios yields $r\bigl(\tilde{R},\tilde{P}\bigr) = n \cdot r (R,P) > s(A) = 1$,
which is a contradiction.
\end{proof}

\begin{example} \label{ex.P1crossP1}
Let $R = \kay \llbracket x, y, z, w \rrbracket /(xy-zw)$. Recall that the divisor class group of $R$ is free of rank $1$; see \cite[II, Exercise 6.5]{Hartshorne}. Moreover, the divisor class of the height-1 prime ideal $\mathfrak{p}=(x,z)$ is a generator of $\Cl R$.  We claim that $P=V(\mathfrak{p})$ is a minimal $F$-pure center.
\begin{claim} \label{clai.ClaimP1crosP1}
The pair $(R,P)$ is purely $F$-regular and $r(R, P) \geq 1/2$.
\end{claim}
\begin{proof}[Proof of claim]
Let $S = \kay \llbracket x, y, z, w \rrbracket$ and $f=xy-zw$. We use Fedder's criterion \cite{FedderFPureRat} to conclude that $\sC_{e,R}$ is generated by the reduction of $\Phi^e \cdot f^{q-1} \in \sC_{e,S}$ to $R$; where $\Phi$ denotes the Frobenius trace on $S$. That is, $\sC_R = \sC_R^{\phi}$ where $\phi \coloneqq \overline{\Phi \cdot f^{p-1}}$. With \autoref{pro.[P]=PforDivisors} in mind, recall that $\sC^{P}_R$ is given, in degree $e$, by all maps $\phi^e \cdot g$ such that $\val_{\p} g \geq q-1$. In particular, we have that $\sC_{e,R}^{P}$ contains all the maps $\phi^e \cdot x^i z^j$ such that $i+j=q-1$.  Thus, the reduction of $\sC_R^{P}$ to $R/\mathfrak{p} = \kay \llbracket y, w \rrbracket$ contains, in degree $e$, the maps $\overline{\phi}^e \cdot x^i y^j$ such that $i+j=q-1$. Notice that these maps are, respectively, the reductions of the map
$\Phi \cdot x^iy^jf^{q-1}$. Nonetheless, one readily sees that
\[
x^iy^jf^{p-1} \equiv (-1)^{i} \binom{p-1}{j} (xz)^{q-1}y^jw^i  \bmod{\bigl(x^p,z^p \bigr)}.
\]
Therefore, $\overline{\phi}^e \cdot x^i y^j$, $i+j=q-1$, is, up to pre-multiplication by units in $\kay$, the dual map of $F^e_* y^{i}w^{j}$ with respect to the free basis of $F^e_* R/\mathfrak{p}$ over $R/\mathfrak{p}$ given by $\{F^e_* y^k w^l\mid 0 \leq k,l \leq q-1\}$. That is, $\overline{\phi}^e \cdot x^i z^j = \Psi^e \cdot y^jw^i$ where $\Psi$ denotes the Frobenius trace of $R/\mathfrak{p} = \kay \llbracket y, w \rrbracket$. Hence,
\[
s\Bigl(\kay \llbracket y,w \rrbracket, \overline{\sC_R^{P}}\Bigr) \geq \textnormal{area} \bigl( [0,1]^{\times 2} \cap \{(y,w) \in \mathbb{R}^2 \mid y+w \geq 1\} \bigr) = 1/2.
\]
This proves the claim by \autoref{pro.OurMethodtoProvePureFRegularity}.
\end{proof}
\end{example}

\begin{example} \label{ex.P2crossP2}
Let $A \coloneqq \kay \llbracket u,v,w,x,y,z \rrbracket$ and $I \coloneqq(\Delta_1, \Delta_2, \Delta_3)$ where $\Delta_1 \coloneqq vz-wy$, $\Delta_2 \coloneqq wx-uz$, and $\Delta_3 \coloneqq uy - vx$. Let $R = A/I$. We claim that the prime divisor $P$ defined by the height-$1$ prime ideal $\p \coloneqq (u,v,w)$ is a minimal $F$-pure center, and moreover $r(R,P)\geq1/6$. We use \autoref{pro.OurMethodtoProvePureFRegularity}. To this end, we recall that $\sC_{e,R}$ was explicitly computed in \cite[Proposition 5.1]{KatzmanSchwedeSinghZhangRingsFrobOperators}.  Indeed, for non-negative integers $s,t$ such that $s+t\leq q-1$, one writes
\[
y^sz^t (\Delta_2 \Delta_3)^{q-1} \equiv x^{s+t} f_{s,t} \bmod{I^{[q]}},
\]
for some $f_{s,t}$, which is well-defined mod ${I^{[q]}}$. Then, 
\[I^{[q]}:I = I^{[q]} + (f_{s,t} \mid s,t\geq 0,\, s+t \leq q-1 ).
\]
Thus, by Fedder's criterion \cite{FedderFPureRat}, $\sC_{e,R}$ is generated by $\Phi^e \cdot f_{s,t}$, where $\Phi$ is a Frobenius trace associated to $A$. We choose $f_{0,0}$ to be $(\Delta_2 \Delta_3)^{q-1}$. In fact, we have that $I^{2(q-1)} \subset I^{[q]}:I$.
In particular, we have the following relations
\begin{equation}
\label{eq.SegreFedderCriterion}
y^sz^t f_{0,0} \equiv x^{s+t} f_{s,t} \bmod I^{[q]}.
\end{equation}
Let $\phi^e_{s,t}$ be the map in $\sC_{e,R}$ induced by $\Phi^e \cdot f_{s,t}$ for $s + t \leq q-1$.
\begin{claim}
$\sC^{P}_{R}$ contains the maps 
$ 
\{\phi_{s,t}^e \cdot u^l v^m w^n \, \vert \, l+m+n = q-1, s + t \leq q-1\}.
$
\end{claim}
\begin{proof}[Proof of claim]
Observe that, up to pre-multiplications by units, all the maps $\phi^e_{s,t}$ induce the same map after we localize at $\mathfrak{p}=(u,v,w)$ by \autoref{eq.SegreFedderCriterion}. Note that $R_\p$ is a DVR so that $\sC_{e, R_\p}$ is principally generated. As the $\phi^e_{s,t}$ generate $\sC_{e,R}$, any $\phi^e_{s,t}$ is a generator of $\sC_{e,R_\p}$. Now, any element $u$, $v$, $w$ is a uniformizer in $R_\p$. To verify the claim, we may localize at $\p$, but then $\phi_{s,t} u^l v^m w^n$ is of the form $\kappa \cdot t^{l+m+n}$ where $\kappa$ is a generator of $\Hom(F_\ast^e R_\p, R_\p)$ and $t$ is a uniformizer. This map is $\p$-compatible if and only if $m+l+n \geq q-1$.
\end{proof}

Next, observe that $R/\p \cong \kay \llbracket x, y ,z \rrbracket$, with Frobenius trace denoted by $\Psi$. By \autoref{remark.FrobeniustraceFinitealgebra}, we may choose $\Psi$ in such a way that $\varphi^e_{s,t}$ and $\Psi^e$ are induced by the same map $\kappa:F_\ast \kay \to \kay$. Thus, for all $s+t \leq q-1$ and all $l+n+m=q-1$, we have that $\phi^e_{s,t} \cdot u^lv^mw^n$ restricts to a map in $\sC_{e,R/\p}$; say $\varphi^e_{s,t} \cdot u^lv^mw^n$. Hence, we have an equality $\varphi^e_{s,t} \cdot u^lv^mw^n = \Psi^e \cdot a_{s,t;l,m,n}$ for a uniquely determined $a_{s,t;l,m,n} \in \kay \llbracket x, y ,z \rrbracket$, which are explictly described as follows:
\begin{claim}
\label{claim.whichmaps}
Let $l,m,n;s,t$ be non-negative integers such that $l+m+n = q -1, s+t \leq q-1$. Let us set $q - 1-s -t \eqqcolon r \geq 0$, so that $r+s + t = q-1$. Then, we have that $a_{s,t;l,m,n} = 0$ unless one of the following four triples $(l+r,m+s,n+t), (l+r-q,m+s,n+t), (l+r,m+s-q,n+t), (l+r,m+s,n+t-q)$ belongs to $\{0,\ldots,q-1\}^{\times 3}$, in which case
\[
a_{s,t;l,m,n} = \xi \cdot x^{l+r}y^{m+s}z^{n+t} 
\]
for some unit $\xi \in \mathbb{F}_p^\times \subset \kay^{\times}$.
\end{claim}
\begin{proof}[Proof of claim]
First of all, note that:
\begin{align*}
f_{0,0}=\Delta_2^{q-1}\Delta_3^{q-1} &= \left(\sum_{a+b = q-1} (-1)^b \binom{q-1}{a} {(wx)^a(uz)^b} \right) \left( \sum_{c+d = q-1} (-1)^d \binom{q-1}{c} {(uy)^c(vx)^d}\right) \\
&=\sum_{\substack{a+b=q-1 \\ \nonumber 
c+d=q-1}}{(-1)^{b+d}\binom{q-1}{a}\binom{q-1}{c} u^{b+c} v^d w^a x^{a+d} y^c z^b}.
\end{align*}
Therefore,
\begin{equation} \label{eqn.f00}
u^lv^m w^nf_{0,0} \equiv -\binom{q-1}{m}\binom{q-1}{n} u^{q-1}v^{q-1} w^{q-1}x^{l+q-1}y^mz^n \bmod \p^{[q]}.
\end{equation}
Indeed, after multiplying by $u^lv^mw^n$, every summand vanishes modulo $\p^{[q]}$ except for the summands where simultaneously $l+b+c\leq q-1$, $m+d \leq q-1$, and $n+a \leq q-1$. However, given the constraints $a+b=q-1$ and $c+d =q-1$, we have that
\[
(l+b+c)+(m+d)+(n+a) = 3(q-1).
\]
Hence, $l+b+c, m+d, n+a = q-1$, and also $a+d=l+q-1$. In particular, $m= c$ and $n=b$.  Set $\xi \coloneqq -\binom{q-1}{m}\binom{q-1}{n} \in \kay^{\times}$.

On the other hand, for $0 \leq i,j,l \leq q-1$ we have that
\begin{align*}
u^lv^mw^nx^iy^jz^kf_{s,t}&=\frac{1}{x^q} u^l v^m w^n x^{i+q-s-t}y^{j}z^{k}x^{s+t}f_{s,t}\\ &= \frac{1}{x^q} u^l v^m w^n x^{i+r+1}y^{j+s}z^{k+t}f_{0,0}\\
&\equiv  \frac{\xi}{x^q} u^{q-1}v^{q-1}w^{q-1}x^{q+i+r+l}y^{j+s+m}z^{k+t+n}  \bmod \p^{[q]}  \\
&\equiv \xi u^{q-1}v^{q-1}w^{q-1}x^{i+r+l}y^{j+s+m}z^{k+t+n} \bmod \p^{[q]}.
\end{align*}
Therefore,
\[
\Phi^e\bigl( u^lv^mw^nx^iy^jz^kf_{s,t} \bigr) \equiv \xi \Phi^e\bigl(u^{q-1}v^{q-1}w^{q-1}x^{i+r+l}y^{j+s+m}z^{k+t+n} \bigr) \bmod \p
\]
Next, we observe that this element is $0 \bmod \p$ unless
\[
i+r+l, j+s+m, k+t+n \equiv q-1 \bmod q.
\]
Since all these three sums are at most $3(q-1)$, we then have
\[
\begin{cases}
i+r+l=q-1+\alpha q \\ j+s+m = q-1 + \beta q \\ k+t+n = q-1 +\gamma q
\end{cases}
\]
for some $\alpha, \beta, \gamma \in \{0,1\}$. However, if we add these equations together, we obtain:
\[
i+j+k + 2(q-1) = 3(q-1) + (\alpha+\beta+\gamma)q.
\]
Equivalently,
\[
i+j+k=q-1 + (\alpha + \beta + \gamma)q.
\]
Being $i+j+k$ at most $3(q-1)$, this forces $\alpha + \beta + \gamma \in \{0,1\}$. Hence, $\alpha,\beta,\gamma$ are either all $0$ or one of them is $1$ while the other two are $0$. In the first case, we then have:
\[
i+r+l, j+s+m, k+t+n = q-1.
\]
Therefore, in this case, we have that $\Phi^e\bigl( u^lv^mw^nx^iy^jz^kf_{s,t} \bigr) \equiv 0 \bmod \p$ unless
\[
i=q-1-(r+l), j=q-1-(s+m), k=q-1 -(t+n) \geq 0. 
\]
In that case, $\Phi^e\bigl( u^lv^mw^nx^iy^jz^kf_{s,t} \bigr) \equiv \xi \bmod \p$, and so $a_{s,t;l,m,n} = \xi x^{r+l}y^{s+m}z^{t+n}$ (whenever $r+l,s+m, t+n \geq q-1$).

Let us consider now the remaining three cases, i.e. $(\alpha,\beta, \gamma) \in \bigl\{(1,0,0),(0,1,0),(0,0,1)\bigr\}$. By symmetry, it suffices to consider $(\alpha,\beta,\gamma) = (1,0,0)$. In this case, we have that the element $\Phi^e\bigl( u^lv^mw^nx^iy^jz^kf_{s,t} \bigr)$ vanishes modulo $\p$ unless
\[
q- 1 \geq i=q-1+q-(r+l) \geq 0; \text{ and } j=q-1-(s+m), \, k = q-1-(t+n) \geq 0, 
\]
equivalently
\[
0 \leq (r+l) - q \leq q-1; \text{ and } j=q-1-(s+m), \, k = q-1-(t+n) \geq 0, 
\]
which implies $\Phi^e\bigl( u^lv^mw^nx^iy^jz^kf_{s,t} \bigr) \equiv \xi x  \bmod \p$. In this case, $a_{s,t;l,m,n} = \xi x^{r+l} y^{s+m} z^{t+n}$.
\end{proof}
Let us analyze which maps the first case $(l +r, m+s, n+t) \in \{0,\ldots, q-1\}^{\times 3}$ of \autoref{claim.whichmaps} yields. Note that the map from the set
\[
\bigl\{(l,m,n;r,s,t) \in \{0,\ldots,q-1\}^{\times 6} \bigm| l+m+n, r+s+t= q-1 \text{ and } l+r, m+s, n+t \leq q-1 \bigr\} 
\]
to the set
\[
\bigl\{ (i,j,k) \in \{0,\ldots,q-1\}^{\times 3} \bigm| i+j+k = 2(q-1) \bigr\}
\]
defined by
\[
(l,m,n;r,s,t) \mapsto (l+r,m+s,n+t)
\]
is surjective. Indeed, taking $ l = s =0$ and given $0 \leq r,m \leq q-1$, we obtain $2(q-1) = r + t + m +n$ or, put differently, $2(q-1) - r- m = t+n$. Thus, we see that this case yields the maps $\Psi^e \cdot x^i y^j z^k$ with $i+j+k = 2(q-1)$. In other words, we obtain the Cartier algebra given by the pair $(\kay\llbracket x,y,z\rrbracket, (x,y,z)^2)$.

For the remaining three cases of \autoref{claim.whichmaps}, we obtain the maps \[
x \cdot \Psi^e \cdot x^{r+l-q}y^{s+m}z^{t+n}, \quad y \cdot \Psi^e \cdot x^{r+l}y^{s+m-q}z^{t+n}, \quad z \cdot \Psi^e \cdot x^{r+l}y^{s+m}z^{t+n-q}
\]
where, respectively, $(r+l-q,s+m,t+n) \in \{0,\ldots,q-1\}$, $(r+l,s+m-q,t+n) \in \{0,\ldots,q-1\}$,  $(r+l,s+m,t+n-q) \in \{0,\ldots,q-1\}$). However, these are all nonsurjective maps. 

In conclusion, we obtain
\[
s \Bigl(R/\p, \overline{\sC_{R}^{P}}\Bigr) \geq \textnormal{volume} \Bigl( [0,1]^{\times 3} \cap \{(x,y,z) \in \mathbb{R}^3 \mid x+y+z \geq 2\} \bigr) = 1/6>0,
\]
where we use \cite[Theorem 4.20]{BlickleSchwedeTuckerFSigPairs1} for the inequality. Hence, $r(R,P) \geq 1/6$ and $(R,P)$ is purely $F$-regular.
\end{example}

\begin{remark}
In \autoref{ex.P2crossP2}, it would be interesting to fully compute $\sC_R^{P}$ to check whether or not $r(R, P) = 1/6$. The issue is that we cannot apply Fedder's criterion for $R$ since $R$ is not regular. One may apply Fedder's criterion to $I + \mathfrak{p}$ in $A$ to work around this.
\end{remark}

\begin{question} \label{que.DeterminantalRings}
Let $C_{r,s}$ be the cone singularity given by the Segre embedding of $\mathbb{P}_{\kay}^r \times \mathbb{P}_{\kay}^s$. The $F$-signatures of these toric rings were compute in \cite{SinghFSignatureOfAffineSemigroup} and it is well-known that $\Cl C_{r,s} \cong \bZ$. In fact, $C_{r,s}$ is a determinantal ring. Let $S=\kay\llbracket x_{i,j}\mid 1\leq i \leq m, 1\leq j \leq n \rrbracket$ be the power series ring in the $m\times n$ matrix of variables $(x_{i,j})$, and let $I_t$ be the ideal generated by the $t \times t$ minors of $(x_{i,j})$ ($2 \leq t \leq \mathrm{min}\{m,n\}$). The quotient ring $R=R(m,n,t)=S/I_t$ is called a determinantal ring. We observe that $C_{r,s}$ is none other than $R(r+1,s+1,2)$. Moreover, if $P$ is the prime divisor on $\Spec C_{r,s}$ given by $\p=(x_{1,1},\ldots,x_{1,s+1})$,\footnote{In fact, any ideal generated by either a fixed column or row of variables.} then the divisor class of $P$ is a free generator of $\Cl C_{r,s}$. Based on the previous examples, it is natural to ask whether the pair $(C_{r,s},P)$ is purely $F$-regular and if so what its splitting ratio is. More generally, if $R$ is an arbitrary determinantal ring, we have that $\Cl R$ is freely generated by $P$ the divisor class of the height-$1$ prime ideal $\p$ generated by the $t-1$ size minors of any set of $t-1$ rows (or columns); see \cite[Corollary 8.4]{BrunsVetterDeterminantalRings}. We ask the same question as before for the pair $(R,P)$. Note that in order to answer to these questions along the same ideas we had for $C_{1,1}$ and $C_{1,2}$, a good understanding of the colon ideal $I_t^{[q]}:I_t$ is needed. Nonetheless, to the best of the authors' knowledge, very little is known about this. The authors believe a different approach is required. 
\end{question}

\subsection{Purely log terminal pairs} \label{sec.PLTPairs} We refer the reader to \cite{KollarMori} for a detailed exposition on log canonical singularities and to \cite{AmbroThesis} for the notion of (minimal) log canonical centers. We will, however, briefly review these notions here. Let $(X,\Delta)$ be a log pair defined over an algebraically closed field of characteristic zero. Fix a log resolution $\pi: Y \to (X, \Delta)$ and write $\Delta_Y = \pi^\ast(K_X + \Delta) -K_Y$. The pair $(X, \Delta)$ is \emph{log canonical} (LC) if the coefficients of $\Delta_Y$ are $\leq 1$. The pair $(X, \Delta)$ is called \emph{purely log terminal} (PLT) if it is LC and the exceptional components of $\Delta_Y$ have coefficients $< 1$. We say that $(X, \Delta)$ is \emph{Kawamata log terminal} (KLT) if all coefficients of $\Delta_Y$ are $< 1$. A prime divisor $P$ on $X$ is called an \emph{LC center} if the coefficient $a_P$ of the strict transform of $P$ in $\Delta_Y$ is $\geq 1$. Since the multiplier ideal $\mathcal{J}(X, \Delta)$ is given by $\pi_\ast \mathcal{O}_Y(K_Y  - \lfloor \pi^\ast(K_X + \Delta \rfloor)$, this is equivalent to $\mathcal{O}_X(-P) \supset \mathcal{J}(X, \Delta)$.

In analogy to \autoref{pro.globaltolocalFpurecenters}, we recall the global-to-local passage is for LC centers. Let $(X, \Delta)$ be a log canonical pair, $\dim X \geq 2$. Let $P$ be an LC center going through a closed point $x \in X$.
In studying $\sO_{X,x}^\mathrm{sh}$, we are free to replace $X$ by any open neighborhood $U$ of $x$ and $\Delta$ by $\Delta_U$. In particular, we may assume that $(X, \Delta)$ is purely log terminal. Indeed, we may write $\Delta_Y = E_1 + \cdots + E_n + \sum_{E} a_E E$, for some $n$ and such that $a_E < 1$. Note that one of the $E_i$, say $E_1$ is $P$. By the assumption that $P$ is a divisor and the minimal LC center through $x$, the other divisors $E_i$ do not contain $x$. Hence, replacing $X$ by a suitable neighborhood $U$ of $x$, we may assume that $(X, \Delta)$ is PLT and moreover $\lfloor \Delta \rfloor = P$ is a prime divisor going through $x$.\footnote{That is, the generic point of $P$ is the only codimension-$1$ point in the non-KLT locus of $(X,\Delta)$.} Thus, we may work in the following setup.

\begin{setup} \label{sec.LastSectionSetup2}
Let $(X,\Delta)$ be a PLT log pair of dimension at least $2$, such that $\lfloor \Delta \rfloor = P$ is a prime divisor going through a closed point $x\in X$. We set  $X_{\bar{x}}^{\circ} = \Spec \mathcal{O}^{\mathrm{sh}}_{X,x} \smallsetminus Z$, where $Z$ is some closed subset of codimension $\geq 2$. We denote by $P$ the pullback of $P$ to $U$.
\end{setup}

The following is analogous to \autoref{pro.XisSFR} and well-known to experts (cf.\ \cite[Proposition 2.43]{KollarMori})

\begin{proposition}
\label{le.plttokltperturbation}
Let $(X, \Delta = \sum a_i \Delta_i)$ be a PLT pair with $X$ quasi-projective and $0 \leq a_i \leq 1$. Then there is a $\mathbb{Q}$-Cartier $\mathbb{Q}$-divisor $\Delta'$ such that the pair $(X, \Delta + \eps \Delta')$ is KLT for all rational $0 < \eps \ll 1$
\end{proposition}
\begin{proof}
Let $m > 0$ be so that $m \Delta$ is integral. Since $X$ is quasi-projective, there is an ample divisor $H$. Choose $n \gg 0$ so that $\mathcal{O}_X(nH + m \Delta)$ is globally generated. As the base locus of the linear system $\vert nH + m\Delta\vert$ is empty, we find an element $D$ of this linear system having no component in common with $\Delta$. Set $\Delta' = \frac{1}{m} D - \Delta$. Then, $K_X + \Delta + \eps \Delta'$ is $\mathbb{Q}$-Cartier since $\Delta'$ is so: $m \cdot \Delta' = D - m \Delta \sim nH$. Note that $\Delta + \eps \Delta' \geq 0$ for all rational $0 \leq \eps \ll 1$. Since $D$ and $\Delta$ share no components, $\lfloor \Delta + \eps \Delta' \rfloor =0$. As  $a(E,X, \Delta + \eps \Delta') \to a(E,X,\Delta)$ for $\eps \to 0$, there is $\eps$ so that $(X, \Delta+ \eps \Delta')$ is KLT.
\end{proof}

We make precise the connection between purely $F$-regular pairs and PLT pairs. 

\begin{theorem}[{\cite[Corollay 5.4]{TakagiPLTAdjoint}}]
\label{theo.PLTspreadout}
Let $(X, \Delta)$ be a log pair. Spread $(X, \Delta)$ out over some finitely generated $\mathbb{Z}$-algebra $A$. Then, $(X,\Delta)$ is PLT if and only if there is a dense open $U \subset \Spec A$ such that the reduction $(X_a, \Delta_a)$ is purely $F$-regular for all $a \in U$.
\end{theorem}

\begin{theorem}
\label{theo.centerspreadout}
Let $(X, \Delta)$ be an affine PLT pair. Assume that $\lfloor \Delta \rfloor = P$ is an minimal LC center for some closed point $x \in P$. Spread $(X, \Delta)$, $P$, and $x$ out over some finitely generated $\mathbb{Z}$-algebra $A$. Then, for all $a \in U$, where $U$ is a dense open subset of $\Spec A$, the divisor $P_a$ is the minimal $F$-pure center through $x_a$. In this situation, a minimal LC center is normal. Conversely, if $P$ is not the minimal LC center through $x$, then $P_a$ is not the minimal $F$-pure center for $x_a$ for all closed points in a dense open set.
\end{theorem}
\begin{proof}
See \cite[Theorem 6.8]{SchwedeCentersOfFPurity}. Schwede's argument immediately also gives the converse statement: If there is some smaller LC center $Q$ passing through $x$, after reduction, we obtain an $F$-compatible ideal $\mathfrak{q}_a$ strictly containing $\mathfrak{p}_a$. Thus, $P_a$ cannot be the minimal $F$-pure center through $x_a$. For normality of the minimal LC center see \cite[Theorem 7.2]{FujinoGongyoCanonicalBundleFormulasAndSubadjunction}.
\end{proof}

By \autoref{theo.PLTspreadout} and \autoref{theo.centerspreadout}, examples in \autoref{sect.examplesofsplittingratios} are examples of PLT pairs when we let $\kay$ have characteristic zero. However, we need to sharpen our hypothesis for the analog of \autoref{ex.Smoothambient}.

\begin{example} \label{exSommothAmbientCharZero}
Let $R$ be regular, local and essentially of finite type over an algebraically closed field of characteristic zero, and let $(f) \subset R$ be a prime ideal. Then, the pair $(R,\Div f)$ is a PLT pair if and only if $R/f$ is a (Gorenstein) KLT singularity.
\end{example}

\section{Digression on local tame fundamental groups} \label{sec.preliminaresTameStuff}

The objective in this section is twofold. First, we overview all the necessary material regarding tame fundamental groups that we need to establish our results. Second, we prove the theorem establishing that Theorems B and C in \autoref{sec.Intro} are formal consequences of structural properties of the Galois category being studied. We start off with our first goal.

\subsection{Tame ramification, cohomological tameness, and Abhyankar's lemma} We commence by recalling some standard definitions in \cite{GrothendieckMurreTameFundamentalGroup}.
\begin{definition} [Tamely ramified field extensions with respect to a DVR] \label{def.TameRamification}
Let $K$ be a field with a discrete valuation ring (DVR) $(A, (u), \kay)$. One says that a finite separable field extension $L/K$ is \emph{tamely ramified with respect to $A$} if: for all (the finitely many) discrete valuation rings $(B, (v), \el)$ of $L$ lying over $A$, we have that $\kay \subset \el$ is separable and $\Char \kay = p$ does not divide the ramification index\footnote{The ramification index $e$ is characterized by the equality $u = b\cdot v^e$ with $b$ a unit in $B$.} of the extension $A \subset B$. If the extensions $A \subset B$ are \'etale, we say \emph{$L/K$ is \'etale with respect to $A$}.
\end{definition}

\begin{definition}[Tamely ramified covers with respect to a divisor] \label{def.TameCover}Let $X$ be a connected normal scheme and let $D = \sum_i P_i$ be a reduced effective divisor on $X$ with prime components $P_i$. One says that a finite cover $Y \to X$ is \emph{tamely ramified with respect to $D$ (or simply over $D$)} if $Y$ is normal and every connected component $Y' \to Y \to X$ of $Y$ is a finite cover $X$ that is \'etale away from $D$, and $K(Y')/K(X)$ is tamely ramified with respect to the DVRs $\sO_{X, \eta_i}$, where $\eta_i$ is the generic point of $P_i$. \end{definition}

The following lemma will be important in our forthcoming discussions.

\begin{lemma}[{\cite[\S2, Lemma 2.2.8]{GrothendieckMurreTameFundamentalGroup}}] \label{lem.Lemma228GM71}
Let $f \: Y \to X$ be a finite cover between connected normal schemes and let $D = \sum_i P_i$ be a reduced divisor on $X$ with prime components $P_i$. Suppose that $f: Y \to X$ is \'etale over the complement of $D$. The following statements are equivalent:
\begin{enumerate}
    \item $f$ is a tamely ramified cover with respect to $D$,
    \item for all $x\in D$, the pullback of $f$ along $g\: \Spec \O_{X,x} \to X$ is a tamely ramified cover with respect to $g^*D$.
    \item for all $x\in D$, the pullback of $f$ along $g\: \Spec \O_{X,x}^{\mathrm{sh}} \to X$ is a tamely ramified cover with respect to $g^*D$.
    \item  for all $x\in D$ of codimension $1$ (in $X$), the pullback of $f$ along $g\: \Spec \O_{X,x} \to X$ is a tamely ramified cover with respect to $g^*D$.
    \item for all $x\in D$ of codimension $1$ (in $X$), the pullback of $f$ along $g\: \Spec \O_{X,x}^{\mathrm{sh}} \to X$ is a tamely ramified cover with respect to $g^*D$.
\end{enumerate}
\end{lemma}

\begin{defprop}[{Kummer-type cyclic covers, \cf \cite[Example 2.2.4]{GrothendieckMurreTameFundamentalGroup}}] \label{pro.KummerCoversAreTame}
Let $(X,D)$ be as in \autoref{def.TameCover} and defined over $\bZ[1/n][\zeta]$ where $\zeta$ is a primitive $n$-th root of unity, which means that $n \in \Gamma(X,\sO_X)$ is invertible and $\Gamma(X,\sO_X)$ contains a primitive $n$-th root of unity (e.g., $X$ may be defined over a separably closed field of characteristic prime to $n$). Suppose that $D=n\cdot E$ in $\Cl X$ and $E$ is Cartier away from $D$. Write $\Div_X \kappa + n\cdot E = D$ for some $\kappa \in K(X)^{\times}$. Then, the finite cover $f\: Y \to X$ determined by the $\sO_X$-algebra 
\[
\sO_X \xrightarrow{\subset} \bigoplus_{i=0}^{n-1}\sO_X(-i\cdot E), \quad \cdot \kappa \: \sO_X(-n \cdot E) \to \sO_X(-D)
\]
is a connected tamely ramified cover over $D$ that is generically cyclic of degree $n$. We refer to these covers as \emph{Kummer-type cyclic covers} or simply as \emph{Kummer covers} when $E$ and so $D$ are principal divisors. We allow $n=0$ to include the trivial cover.
\end{defprop}
\begin{proof}
Note that, over $U \coloneqq X \smallsetminus D$, the cover $Y \to X$ is the element of $H^1(U, \bm{\mu}_n)$ corresponding to $\Div_U \kappa + n \cdot E|_U = 0$ as $E|_U \in \Pic U$. In particular, $Y|_U \to U$ is a $\bZ/n\bZ$-torsor (as $\zeta \in \Gamma(X,\sO_X)$) and in particular \'etale. See \cite[III, \S4, p. 125-126]{MilneEtaleCohomology}. Next, we explain why $Y$ is normal. Note that, since $f$ is affine, $\sO_Y$ satisfies the $(\mathbf{S}_2)$ condition as so does the $\sO_X$-module $f_* \sO_Y = \bigoplus_{i=0}^{n-1}\sO_X(-i\cdot E)$. To see why $Y$ satisfies $(\mathbf{R}_1)$, it suffices to look at those codimension-$1$ points not lying over (the generic point of) the $P_i$'s as $f$ is \'etale away from $D$. That is, it suffices to check that the $\sO_{X,P_i}$-algebras $\sO_{X,P_i} \otimes_{\sO_X} f_* \sO_Y$ are regular for all $i$. Observe that $\sO_{X,P_i} \otimes_{\sO_X} f_* \sO_Y \cong \sO_{X,P_i}[T]/(T^n-t)$ where $t$ is a uniformizer of $\sO_{X,P_i}$, and further $\sO_{X,P_i}[T]/(T^n-t)$ is local with maximal ideal $(t) \oplus \bigoplus_{i=1}^{n-1} \sO_{X,P_i} \cdot T = (T)$ and so regular. This computation further shows that $\sO_{X,P_i} \otimes_{\sO_X} f_* \sO_Y$ is an extension of DVRs with ramification index $n$, which is prime to all residual characteristics of $X$ as $1/n \in \Gamma(X,\sO_X)$. This proves that $Y$ is normal.

It remains to prove that $Y$ is connected/integral, for which it suffices to show that $K(X) \otimes_{\sO_X} f_* \sO_Y$ is a field. Notice that, $K(X) \otimes_{\sO_X} f_* \sO_Y \cong K(X)[T]/(T^n-\kappa)$. Suppose, for the sake of contradiction, that $K(X)[T]/(T^n-\kappa)$ is not a field. Then, there is $\kappa' \in K(X)^{\times}$ such that $\Div \kappa = \Div \kappa'^m$ for some $1< m \mid n$ (using \cite[VI, \S6, Theorem 9.1]{LangAlgebra}); \cf \cite[Corollary 1.9]{TomariWatanabeNormalZrGradedRings}. Then, $m \cdot (\Div_X \kappa' + (n/m) D ) = P$, which violates the reducedness of $P$.
\end{proof}

Given the equivalence between (a) and (e) in \autoref{lem.Lemma228GM71}, it is of fundamental importance to understand the tamely ramified covers over a strictly local DVR with respect to its uniformizer. In this regard, the following result together with \autoref{lem.Lemma228GM71} imply that tamely ramified covers are Kummer over the \'etale-germs at the generic points of the divisor $D$.

\begin{theorem}[\cite{SerreLocalFields}]
\label{theo.serretameiskummer}
Let $K$ be a field with a \emph{strictly local} DVR $(A,(u),\kay)$ with $\Char \kay = p \geq 0$. Then, every Galois field extension $L/K$ that is tamely ramified with respect to $A$ is Kummer, \ie $L=K\big(u^{1/n}\big)$ for some $n$ prime to $p$, and in particular cyclic. In other words, every Galois tamely ramified cover over $X = \Spec A$ with respect to $\Div u$ is Kummer.
\end{theorem}
\begin{proof}
See \cite[Ch. IV, \S 2, Proposition 8]{SerreLocalFields} for the case $p=0$. For characteristic $p>0$, note that, by tameness and $\kay$ being separably closed, we have $p \nmid [L:K]$. One simply replaces the use of \cite[Corollary 2]{SerreLocalFields} with Corollary IV, \S 2, 3 in ibid.
\end{proof}

\begin{remark}
The intuition behind \autoref{theo.serretameiskummer} is the following; see \cite[I, Example 5.2 (e)]{MilneEtaleCohomology}. We think of $\Spec K \cong \Spec A \smallsetminus \{(u)\} $ as an algebraic analog of the punctured disc in the plane, then this result says that $\pi_1^{\textnormal{\'et}}(\Spec K)$ is isomorphic to $\hat{\Z}$---the profinite completion of $\Z$---at least if the residual characteristic is $0$ else what we can say is $\pi_1^{\textnormal{t}}(\Spec K) \cong \hat{\Z}^{(p)}$. 
\end{remark}

As mentioned before, \autoref{theo.serretameiskummer} tells us that tamely ramified covers over a reduced effective divisor are  of a very special type \'etale-locally around the generic points of the divisor. In case the divisor $D$ in \autoref{def.TameCover} has \emph{normal crossings} \cite[\S1.8]{GrothendieckMurreTameFundamentalGroup}, \emph{Abhyankar's lemma} establishes that the same hold at all \emph{special} points in the support of the divisor; see \cite[\S 2.3]{GrothendieckMurreTameFundamentalGroup}, \cite[Expos\'e XIII, \S5]{GrothendieckSGA}. More precisely:

\begin{theorem}[Abhyankar's lemma]
With notation as in \autoref{def.TameCover}, suppose additionally that $(X,D)$ has normal crossings in the sense of \cite[\S1.8]{GrothendieckMurreTameFundamentalGroup}. Then, the connected components of the pullback of $Y \to X$ along $ \Spec \sO_{X,\bar{x}}^{\mathrm{sh}} \to X$ are (quotients) of Kummer covers for all geometric points $\bar{x}\to X$.
\end{theorem}
We are interested in studying tame cover with respect to divisors that may not have normal crossings. Fortunately, our efforts will lead to a generalization of this result when the divisor $D$ is irreducible yet singular; see \autoref{lem.THELEMMAFORABH}. Following \cite{KerzSchmidtOnDifferentNotionsOfTameness,ChinburgErezPappasTaylorTameActions}, we have a stronger notion of tameness.

\begin{definition}[Cohomological tameness]
Let $U$ be a normal connected scheme equipped with a dense open embedding $U \to X$ with $X$ normal and connected.\footnote{Unlike \cite{KerzSchmidtOnDifferentNotionsOfTameness}, we do not require $X$ to be proper over some field.} We say that a finite Galois cover $V \to U$ is \emph{cohomologically tamely ramified with respect to $X$} if its integral closure $f\: Y \to X$ is so that the trace map $\Tr_{Y/X} \colon f_* \O_Y \to \O_X$ is surjective. A finite \'etale cover $V \to U$ is cohomologically tamely ramified if it can be dominated by a Galois one.
\end{definition}

\subsection{Tame Galois categories and their fundamental groups}
Consider the setup:

\begin{setup} \label{setupTameGroups}
Let $(R,\fram,\kay, K)$ be a strictly local normal domain of dimension at least $2$. Let $Z$ be a closed subscheme of $X \coloneqq \Spec R$ of codimension at least $2$. We consider a prime Weil divisor $P$ on $X^{\circ} \coloneqq X \smallsetminus Z$, which extends to a unique prime divisor on $X$ that we also denote by $P$, then $P$ corresponds to a unique height-$1$ prime ideal $\p \subset R$. We set $U \coloneqq X^{\circ} \smallsetminus P$. We further set $p \coloneqq \Char \kay \geq 0$ and assume that $ p \in \p \subset \fram$, so that $\Char K(P) = p = \Char \kay$ (here, $K(P)$ is the function field of $P$, i.e., the residue field of $R_{\p}$). In particular, after choosing an embedding $\bar{\bF}_p \subset \kay$, we have that $(R,\fram,\kay, K)$ is a local algebra over $\bZ_p^{\mathrm{sh}}$---the maximal unramified extension of $\bZ_p$, which is given by adjoining all prime-to-$p$ roots of unity to $\bZ_p$. Of course, this is just a fancy way to say that $R$ contains all the $n$-th roots of unity if $p  \nmid  n$. In particular, we allow $\bZ_p^{\mathrm{sh}} \to R$ to be injective, i.e., $R$ may be of mixed characteristic in this section. However, we are assuming that $R_{\p}$ has the same (mixed or not) characteristic as $R$.
\end{setup}

We study two types of tame Galois categories in this paper which we introduce next. We invite the reader to consult \cite{MurreLecturesFundamentalGroups, CadoretGaloisCategories} for a thorough exposition on Galois categories and fundamental groups, or the classic, original reference \cite[Expos\'e V]{GrothendieckSGA}.

\subsubsection{The cohomologically tame Galois category}
Working in \autoref{setupTameGroups}, the first tame fundamental group of interest is the the fundamental group $\pi_1^{\textnormal{t}}(X^{\circ})$ classifying the Galois category $\mathsf{FEt}^{\mathrm{t},X}(X^{\circ})$ of covers over $X^{\circ}$ that are cohomologically tamely ramified with respect to $X$. The minimal (or connected) objects of this Galois category are the local finite extensions $(R,\fram, \kay, K) \subset (S,\fran, \el, L)$ such that $S$ is a normal domain, $\Tr_{S/R} \: S \to R$ is surjective, and $R \subset S$ is \'etale over $X^{\circ}$. Thus, $\pi_1^{\textnormal{t},X}(X^{\circ}) = \varprojlim \textnormal{Gal}(L/K)$ where the limit runs over all Galois extensions $L/K$ inside a fixed separable closure of $K$ such that the integral closure of $R$ in $L$; say $R^L/R$, is \'etale over $X^{\circ}$ and $\Tr \: R^L \to R$ is surjective. See \cite[\S 2.4]{CarvajalSchwedeTuckerEtaleFundFsignature}.

\subsubsection{The tame Galois category of a prime divisor}
In this section, the perspective is quite different from the one above.
Working in \autoref{setupTameGroups}, consider the Galois category $\mathsf{Rev}^{P}(X^{\circ})$ of finite covers over $X^{\circ}$ that are tamely ramified with respect to $P$. The corresponding fundamental group is denoted by $\pi_1^{\mathrm{t},P}(X^{\circ})$ (we choose a geometric generic point as our base  point, which is suppressed from the notation). As before, we may restrict ourselves to a local algebra setup as the following remark explains.
\begin{remark}[Reduction to local algebra] \label{rem.ReductionLocalAlgebra}
Since $R$ is a strictly local normal domain, the Galois objects of the category $\mathsf{Rev}^{P}(X^{\circ})$ are the (generically) Galois local finite extensions of normal domains $(R,\fram, \kay, K) \subset (S, \fran, \el, L)$ that are \'etale over $U$ but tamely ramified over $P$ (i.e. $L/K$ is tamely ramified with respect to $R_{\p}$). In this way,
\[
\pi_1^{\textnormal{t},P}(X^{\circ}) = \varprojlim \textnormal{Gal}(L/K)
\]
where the limit runs over all finite Galois extensions $L/K$ inside some fixed separable closure of $K$ such that the integral closure of $R$ in $L$ is tamely ramified over $X^{\circ}$ with respect to $P$.
When we refer to a cover $Y^{\circ} \to X^{\circ }$ in $\mathsf{Rev}^{P}(X^{\circ})$, we mean that $Y=\Spec S$ with $S$ as above, and $Y^{\circ}=Y\smallsetminus f^{-1}(Z)$, where $f\: \Spec S \to \Spec R$ is the corresponding morphism.
\end{remark}

\begin{example}[Kummer-type cyclic covers] \label{ex.KummerTypeCyclicCovers}
Suppose that there is a divisor $D$ on $X$ such that $P = n \cdot D$ in $\Cl X$ and so that $D|_U$ is Cartier. Writing, $\Div \kappa + nD = P$, let $S= \Gamma(Y,\sO_Y)=\bigoplus_{i=0}^{n-1} R(-iD)$ with $f\: Y \to X$ as in \autoref{pro.KummerCoversAreTame}. Let $\fran \coloneqq \fram \oplus \bigoplus_{i=1}^{n-1} R(-iD)$ and  $\q \coloneqq \p \oplus \bigoplus_{i=1}^{n-1} R(-iD)$. One readily sees that these two are ideals of $S$. In fact, $S/\fran = R/\fram= \kay$ and $S/\q = R/\p$, thereby $\fran$ is maximal and $\q$ is prime. Moreover: $\fran \cap R = \fram $, $\q \cap R = \p$, and $\fran$, $\q$ are; respectively, the only primes of $S$ with such property. Further, we have that $\height \q = 1$. Indeed, $\height \q \leq \height (\q \cap R) = \height \p = 1$ from integrality of $S/R$ (Going-up theorem) and $S$ being a domain rules out the possibility $\height \q = 0$ (as clearly $\q \neq 0$). Thus, $Q=V(\q)$ is a prime divisor on $Y$. Putting everything together, $(R, \fram, \kay, K, \p) \subset (S, \fran, \kay, K(\kappa^{1/n}), \q)$ defines a cyclic cover in $\mathsf{Rev}^{P}(X^{\circ})$ whose pullback to $R_{\p}$ is Kummer; see the proof of \autoref{pro.KummerCoversAreTame}.

If $\p = (r)$ is principal, then $\Div \kappa + nD = \Div r$ and so $D$ is torsion, say of index $m \mid n$. In particular, we may use this to define a connected quasi-\'etale cover $g\: W \to X$ of degree $m$ trivializing $D$. Then, it follows that the base change $f_W \: Y_W \to W$ is a Kummer cover of the form $\Spec \sO_W[T]/(T^n-r) \to W$. Indeed, after trivializing $D$, say $D=\Div s$, the equality $\Div \kappa + n \Div s = \Div r$ says that $r=us^n\kappa$ for some unit $u \in \Gamma(W,\sO_W)$. Then, since $\Gamma(W,\sO_W)$ is also strictly henselian, we may say that $u=1$ by replacing $s$ by $u^{-1/n}s$. Thus, $L(r^{1/n})=L(\kappa^{1/n})$, where $L$ is the function field of $W$. 
\end{example}

\subsection{Some examples of tamely ramified covers} \label{sec.ExamplesTameCovers}

In this section, we provide some examples illustrating what may go wrong in Abhyankar's lemma if the divisor in question is too singular. Additionally, we consider instructive to have some examples at hand that we may use across the forthcoming sections to highlight particular features of our results. We will employ the following useful fact throughout.

\begin{proposition}[{\cite[\href{https://stacks.math.columbia.edu/tag/09EB}{Lemma 09EB}]{stacks-project}}] \label{pro.nefFormular}
Let $R$ be a normal domain with fraction field $K$. Let $L/K$ be a finite Galois extension of degree $d$, and let $S$ be the integral closure of $R$ in $L$. Fix a height-$1$ prime ideal $\p \subset R$, and let $\q_1,\ldots, \q_n \subset S$ be the list of distinct prime ideals of $S$ lying over $\p$. Then, all the DVR extensions $R_{\p} \to S_{\q_i}$ share the same ramification index $e$ and residual degree $i$. Moreover, the formula $d=n\cdot e \cdot i$ holds.
\end{proposition}

\begin{terminology}
We shall often refer to $i$ in \autoref{pro.nefFormular} as the \emph{inertia degree}.
\end{terminology}

\begin{example}[The cusp] \label{ex.TheCusp}
Let $(R,\fram,\kay,K)$ be a regular local ring with regular system of parameters $\fram=(x,y)$. We assume $\Char \kay \neq 2,3$. Let $L$ be the splitting field of $T^3+xT+y \in K[T]$. This polynomial is irreducible.\footnote{Indeed, if it were reducible, it would admit a root in $K$ and further in $R$ by normality of $R$. In that case, $y=t(t^2+x)$ for some $t\in R$. Since $R$ is a UFD and $y$ is an irreducible element, this implies that either $t$ or $t^2+x$ is a unit, and \emph{a fortiori} both are units implying further that $y$ is a unit, which is a contradiction.} Let $t_1, t_2, t_3 \in L$ be the distinct roots of $T^3+xT+y$. Setting,
\[
\delta \coloneqq (t_1-t_2)(t_2-t_3)(t_1-t_3)
\]
we have that $\delta^2=-4x^3-27y^2 \eqqcolon \Delta$. In particular, $\delta \notin K$ for $\Delta$ is irreducible in $R$. Therefore, $L/K$ is a Galois extension of degree $6$ with $\Gal(L/K) \cong S_3$---the symmetric group. See \cite[\S7.5]{RomanFieldTheory} or \cite[VI, \S2]{LangAlgebra}. In fact, $K(\delta)$ is the fixed field of the (cyclic) alternating group $A_3 \subset S_3$. Thus, $L=K(\delta,t_1)$ and
\[
L=K(\delta)[T]\big/\big(T^3+xT+y\big).
\]
In fact, a direct computation shows that if $t$ is one of the roots then the remaining roots are: 
\[
\frac{-t}{2} \pm \frac{\delta}{2\big(3t^2+x\big)},
\]
where $3t^2 + x \neq 0$ as the minimal polynomial of $t$ over $K$ has degree $3$.\footnote{In case the reader wants to corroborate this assertion by hand, notice that $T^3+xT+y=(T-t)\big(T^2+tT+t^2+x\big)$. Hence, it suffices to verify that these are roots of $T^2+tT+t^2+x = (T+t/2)^2+\big(3t^2+4x\big)/4$, which in turn boils down to checking  $\delta^2 +\big(3t^2+4x\big)\big(3t^2+x\big)^2=0$, which is a straightforward computation.}

We set $t=t_1$, and set $t_2$ to be the root with the positive sign in the above expression. Let $S$ be the integral closure of $R$ in $L$. Of course, $S \ni \delta, t_1,t_2, t_3$. Then, we have:
\begin{claim} \label{cla.TameCoverCusp}
$R \subset S$ is a tamely ramified extension with respect to the prime divisor $D= \Div \Delta$. Moreover, there are exactly three prime divisors of $S$ lying over $(\Delta)$, with ramification index $e=2$ and inertia degree $i=1$.
\end{claim}
\begin{proof}[Proof of claim]
By \autoref{pro.KummerCoversAreTame}, the integral closure of $R$ in  $K(\delta)$ is $R[\delta]$, so that $S$ is the integral closure of $R[\delta]$ in $L$. On the other hand, we may consider the flat extension of degree $3$
\[
R[\delta] \subset R[\delta,t] \cong R[\delta][T]\big/\big(T^3+xT+y\big).
\]
Notice that the discriminant ideal of this extension is $(\Delta)$ whereas the different ideal is $(3t^2+x)$. Therefore, $R \subset R[\delta,t]$ is \'etale away from $D$ and so $R[\delta,t]_{\delta}=R[\delta, \delta^{-1},t]$ is normal. In particular, the extension $R[\delta, t] \subset S$ is an equality after localizing at $\delta$ (or well at $3t^2+x$). Thus, the extension $R_{\Delta} \subset S_{\Delta}$ is \'etale. By \autoref{lem.Lemma228GM71}, we are left with showing $R_{(\Delta)} \to S_{(\Delta)}$ is a tamely ramified extension. To this end, observe that
\[
\Delta =\delta^2 = (t_1-t_2)^2(t_2-t_3)^2(t_1-t_3)^2.
\]
Now, let $\q\subset S$ be a prime ideal lying over $(\Delta)$. It must then contain at least one of the elements $t_1-t_2, t_2-t_3, t_1-t_3$. We argue next it can contain only one of them. Indeed, if it contains two of them it must contain the third one and thus all of them.\footnote{For instance, $t_1-t_3=(t_1-t_2)+(t_2-t_3)$.} In particular, the ramification index of $R_{(\Delta)} \to S_{\q}$ is at least $6$ and by applying \autoref{pro.nefFormular} we conclude that $n=1$, $e=6$, and $i=1$ (with notation as in \autoref{pro.nefFormular}). In particular, $\q$ is generated by either of these elements. On the other hand, we have that
\begin{equation} \label{eqn.DifferencesOft_i's}
t_1-t_2=\frac{3t\big(3t^2+x\big)-\delta}{2\big(3t^2+x\big)}, \quad t_1-t_3=\frac{3t\big(3t^2+x\big)+\delta}{2\big(3t^2+x\big)}, \quad t_2-t_3 = \frac{2\delta}{2\big(3t^2+x\big)}.
\end{equation}
From this, we conclude that all the displayed numerators belong to $\q$ and so does $6t\big(3t^2+x\big)$. Nonetheless, $\q \not\ni t$ as otherwise $y=-t\big(t^2+x\big) \in \q \cap R = (\Delta)$, which is not the case. In this way, our conclusion must be that $3t^2+x \in \q \cap R[\delta,t] = \sqrt{(\delta)}$.\footnote{To see the equality, consider $\mathfrak{r}\subset R[\delta,t]$ to be a minimal prime of $(\delta)$. Since $R[t,\delta] \subset S $ is integral, there is at least one prime ideal of $S$ lying over $\mathfrak{r}$. However, any such a prime must lie over $(\Delta) \subset R$ and so must do $\q$.} This, however, is a contradiction. Indeed, we have that $R[\delta,t]$ is a rank-$6$ free $R$ module and moreover
\[
R[\delta, t] = R \cdot 1 \oplus R\cdot t \oplus R \cdot t^2 \oplus R \cdot \delta \oplus R \cdot \delta t \oplus R \cdot \delta t^2 = \big\langle 1,t,t^2 \big\rangle_R \oplus \big\langle \delta, \delta t, \delta t^2 \big\rangle_R,   
\]
whence one sees that any power of $3t^2+x$ is going to be belong to the direct summand $\big\langle 1,t,t^2 \big\rangle_R$ whereas 
\[
(\delta) \subset \Big((\Delta) \cdot \big\langle 1,t,t^2 \big\rangle_R \Big) \oplus \big\langle \delta, \delta t, \delta t^2 \big\rangle_R.
\]
Additionally, an inductive argument readily shows that the constant coefficient of $\big(3t^2+x\big)^n$ is $x^n$ for every exponent $n$. Putting everything together, we see that $3t^2 + x \in \sqrt{(\delta)}$ yields that $x^n \in (\Delta)$ for some $n$ and so $x\in (\Delta)$, which is the sought contradiction.

In conclusion, the principal ideals $(t_1-t_2), (t_2-t_3), (t_1-t_3) \subset S$ share no minimal prime. By using \autoref{pro.nefFormular}, we conclude that these are (the) prime ideals of $S$ lying over $(\Delta) \subset R$, with ramification index $e=2$ and inertia degree $i=1$. This proves the claim.
\end{proof}
\end{example}

\begin{example}[Whitney's umbrella] \label{ex.WhitneyUmbrella}
Let $(R,\fram,\kay,K)$ be a regular local ring with $\kay$ of \emph{odd characteristic}, and let $f\coloneqq x^2-y^2z$ where $\fram=(x,y,z)$ is a regular system of parameters. The polynomial expression $x^2-y^2z$ plays a fundamental role in the description of degree-$4$ Galois extensions; see \cite[VI, \S Ex. 4]{LangAlgebra}. Thus, we start off by considering the degree-$2$ Galois extension $E=K\big(\sqrt{f}\big)$. Next, we consider the tower of degree-$2$ Galois extensions
\[
E\subset E\big(\sqrt{z}\big) \subset E\big(\sqrt{z}\big) \left(\sqrt{x+y\sqrt{z}}\right) . 
\]
Set $\alpha = x+y\sqrt{z}$, $\alpha'=x-y\sqrt{z}$, and $\beta = \sqrt{\alpha}$. The above tower is $E \subset E(\alpha) \subset E(\beta)$. By \cite[loc. cit.]{LangAlgebra}, $E(\beta)/E$ is a non-cyclic degree-$4$ Galois extension as $\alpha \alpha' = f$ is a square in $E$. In fact, $E(\beta)/E$ is the splitting field of $T^4-2xT^2+f \in E[T]$:
\[
T^4-2xT^2+f = \big(T^2-\alpha\big)\big(T^2-\alpha'\big)=(T-\beta)(T+\beta)\big(T-\sqrt{f}/\beta\big)\big(T+\sqrt{f}/\beta\big).
\]
Moreover, $E(\beta)/K$ is a degree-$8$ non-cyclic Galois extension, for it is the splitting field of $T^4-2xT^2+f \in K[T]$. In fact, setting $\beta'\coloneqq \sqrt{f}/\beta$, we see that $\Gal\bigl(E(\beta)/E\bigr)$ is generated by the transpositions $\tau \: \beta \mapsto -\beta$ and $\sigma \: \beta \mapsto \beta'$. Moreover, in $\Gal\big(E(\beta)/K\big)$ we have $\rho \: \beta \mapsto \beta, \sqrt{f} \mapsto -\sqrt{f}$. In this way, $\pi \coloneqq \sigma \rho \: \beta \mapsto \beta', \sqrt{f} \mapsto -\sqrt{f} $ is an element of order $4$ whose square and cube are; respectively, $\tau$ and $ \rho \sigma$. That is, $\Gal\big(E(\beta)/K\big)$ is generated by two elements $\sigma$ and $\pi$ satisfying relations: $\sigma^2=1$, $\pi^4=1$, and $\sigma \pi = \pi^3 \sigma$. In other words, $\Gal\big(E(\beta)/K\big)$ is isomorphic the dihedral group---the symmetries of the square. Thus,
\[
\Gal\big(E(\beta)/K\big) = \{1,\sigma, \rho, \tau, \pi, \sigma \rho, \pi \sigma, \sigma \tau \}.
\]

Let $S$ be the integral closure of $R$ in $E(\beta)$. Next, we claim the following.
\begin{claim}
$S=R\big[\sqrt{z},\beta, \beta'\big]$
\end{claim}
\begin{proof}[Proof of claim]
By \autoref{pro.KummerCoversAreTame}, $R\big[\sqrt{f}, \sqrt{z}\big]$ is normal and so it is the integral closure of $R$ in $E(\alpha)$---the fixed field of $\langle\tau\rangle$. Thus, we just need to prove that $S$ is the integral closure of $R\big[\sqrt{f}, \sqrt{z}\big]$ in $E(\beta)$. To this end, we prove that any element $\gamma \in S$ is an $R\big[\sqrt{f}, \sqrt{z}\big]$-linear combination of $1$, $\beta$, and $\beta'$. We know that $\gamma = a + b \beta \in E(\beta)$ for some (uniquely determined) $a,b \in E(\alpha)$. Since $E(\beta)/E(\alpha)$ is a quadratic extension, the minimal polynomial of $\gamma$ is described in terms of its trace and norm as follows:
\[
T^2 + \Tr_{E(\beta)/E(\alpha)} (\gamma) T + \Norm_{E(\beta)/E(\alpha)}(\gamma).
\]
Observe that $\Tr_{E(\beta)/E(\alpha)} (\gamma) = 2a$ and $\Norm_{E(\beta)/E(\alpha)}(\gamma)=a^2-b^2\alpha$. Therefore, $\gamma \in S$ if and only if both $2a$ and $a^2-b^2\alpha$ belong to $R\big[\sqrt{f}, \sqrt{z}\big]$, which is equivalent to $a,b^2\alpha \in R\big[\sqrt{f}, \sqrt{z}\big]$. 

Now, since $b^2 \alpha$ belongs to $R\big[\sqrt{f}, \sqrt{z}\big]$ so does $b^2 f = b^2 \alpha \alpha'$. That is, since $b^2\alpha \in R\big[\sqrt{f}, \sqrt{z}\big]$, then $b^2 f$ belongs to the ideal $(\alpha')\subset R\big[\sqrt{f}, \sqrt{z}\big]$. Since $b^2f=\big(b\sqrt{f}\big)^2$, this is to say  that $b \sqrt{f} \in E(\alpha)$ is integral over $(\alpha') \subset R\big[\sqrt{f}, \sqrt{z}\big]$. Given that $R\big[\sqrt{f}, \sqrt{z}\big]$ is integrally closed in $E(\alpha)$, we conclude that $b\sqrt{f} \in \sqrt{(\alpha')} \subset R\big[\sqrt{f}, \sqrt{z}\big]$; see \cite[Chapter 2, Corollary 2.6]{KunzIntroToACAG}. The result then follows once we have shown that
\begin{equation} \label{eqn.LaBellaEcuacion}
\sqrt{(\alpha')} = \big(\alpha',\sqrt{f}\big), \text{ in } R\big[\sqrt{f}, \sqrt{z}\big].
\end{equation}
Indeed, granted \autoref{eqn.LaBellaEcuacion}, we would have that:
\[
\gamma=a+b\beta = a+\big(b\sqrt{f}\big) \frac{\beta}{\sqrt{f}} = a + \big(c\alpha' + d \sqrt{f}\big) \frac{\beta}{\sqrt{f}} = a + d\beta + c\frac{\alpha'}{\beta'} = a+d\beta+ c\beta',
\]
for some $d,c \in R\big[\sqrt{f}, \sqrt{z}\big]$---we saw before that $a\in R\big[\sqrt{f}, \sqrt{z}\big]$.

To prove \autoref{eqn.LaBellaEcuacion}, observe that the containment from right to left is clear, for $\sqrt{f}^2 = \alpha \alpha'$. For the converse containment, observe that $R\big[\sqrt{f}, \sqrt{z}\big]$ is free over $R$ with basis $1, \sqrt{z}, \sqrt{f}, \sqrt{z}\sqrt{f}$. In particular, if an element in $R\big[\sqrt{f}, \sqrt{z}\big]$ belongs to $\sqrt{(\alpha')}$ then so does the summand in the $R$-span of $1$ and $\sqrt{z}$. Thus, it is enough to prove that $r+s\sqrt{z} \in \sqrt{(\alpha')}$; with $r,s \in R$, belongs to $(\alpha')$. That is, it suffices to explain why the contraction of $\sqrt{(\alpha')} \subset R\big[\sqrt{f}, \sqrt{z}\big]$ to $R\big[\sqrt{z}\big]$ is the ideal $(\alpha')$. This, however, follows from observing that $(\alpha')\subset R\big[\sqrt{z}\big]$ is a prime ideal. Indeed, observe that $R\big[\sqrt{z}\big]$ is a regular local ring (and so an UFD) as its maximal ideal is given by $\fram \oplus R \cdot \sqrt{z} =\big(x,y,z,\sqrt{z}\big)=\big(x,y,\sqrt{z}\big)$. On the other hand, the extension of the prime ideal $(f) \subset R$ to $R\big[\sqrt{z}\big]$ splits as $(f)=(\alpha)(\alpha')$. Since there cannot be more than two prime ideals of $R\big[\sqrt{z}\big]$ lying over $(f) \subset R$, we conclude that these are $(\alpha)$ and $(\alpha')$.
\end{proof}

\begin{claim}
The extension $R \subset S$ is tamely ramified with respect to the reduced divisor $D=\Div z + \Div f$. Moreover, for both prime divisors $(z), (f) \subset R$, there are exactly two prime ideals of $S$ lying over with ramification index $2$ and inertia degree $2$.
\end{claim}
\begin{proof}[Proof of claim]
We begin by proving that $R_{zf} \subset S_{zf}$ is \'etale. Indeed, we have a tower $R\subset R\big[\sqrt{f},\sqrt{z}\big] \subset S$ where the bottom extension is flat of degree $4$. One readily verifies that the discriminant ideal of the bottom extension is $(zf)$, so it is \'etale over $R_{zf}$. It suffices to check that $R\big[\sqrt{f},\sqrt{z}\big]_{zf} \subset S_{zf}$ is \'etale. To this end, notice that by inverting $f$ we invert $\alpha$ and $\alpha'$ in $R\big[\sqrt{f},\sqrt{z}\big]$ and $\beta$ and $\beta'$ in $S$, for we have the relation $\alpha\alpha'=f=\beta^2\beta'^2$. In particular, $1/\alpha \in R\big[\sqrt{f},\sqrt{z}\big]_{zf}$ and $1/\beta \in S_{zf}$. Therefore, $R\big[\sqrt{f},\sqrt{z}\big]_{zf} \subset S_{zf}$ is free with basis $1, \beta$ (as $\beta'=\frac{\sqrt{f}}{\beta} = \frac{\sqrt{f}}{\alpha} \beta$) and its discriminant ideal is generated by
\[
\begin{vmatrix}
1 & 1 \\
\beta & - \beta
\end{vmatrix}^2 = (-2\beta)^2=4\alpha,
\]
which is a unit and consequently the extension is \'etale; as needed.

It remains to prove that $E(\beta)/K$ is tamely ramified with respect to both DVRs $R_{(z)}$ and $R_{(f)}$. Nevertheless, this follows from simple characteristic considerations. Indeed, since the extensions are Galois, we know in each case that $8=n\cdot e \cdot i$ with $n$ being the number of primes lying over, $e$ the ramification indexes, and $i$ the residual degrees; as in \autoref{pro.nefFormular}. Then, $e$ and $i$ are necessarily prime to characteristic, which was assumed odd from the beginning. Recall that $f=\beta^2 \beta'^2$. Using \autoref{pro.nefFormular}, this implies that $(\beta), (\beta') \subset S$ are (the) prime ideals of $S$ lying over $(f) \subset R$, and the ramification index is $2$ as well as the residual degree.\footnote{To see that these two ideals are different, note that otherwise would imply that $(\alpha)=(\alpha')$ in $R\big[\sqrt{z}\big]$, which is tantamount to say that $(f) \in \Spec R$ is a branch point of $R\subset R[z]$. This, however, is not the case.} Similarly, we have that $2y\sqrt{z}=(\beta-\beta')(\beta+\beta')$, so that $4y^2z=(\beta-\beta')^2(\beta+\beta')^2$. Therefore, $(\beta-\beta'), (\beta+\beta') \subset S$ are (the) two prime ideals of $S$ lying over $(z) \subset R$, with ramification index and inertia degree equal to $2$.\footnote{Notice that $(\beta-\beta') \neq (\beta+\beta')$ in $S$ as otherwise this would yield that the common ideal contains both $(\beta)$ and $(\beta')$, which is absurd as then they are all the same ideal.}
\end{proof}
\end{example}

\begin{example} \label{ex.Parabola}
We may specialize \autoref{ex.WhitneyUmbrella} by setting $y=1$. More precisely, we may consider $(R,\fram,\kay,K)$ to be a regular local ring of odd residual characteristic with regular system of parameters $\fram=(x,z)$ and set $f \coloneqq x^2-z$. Letting $L/K$ be the splitting prime of $T^4-2xT^2+f \in K[T]$, the same arguments \emph{mutatis mutandis} as in \autoref{ex.WhitneyUmbrella} show that $S \coloneqq R^L = R\big[\sqrt{z}, \sqrt{x\pm\sqrt{z}}\big]$ and moreover that $R \subset S$ is a degree-$8$ Galois tamely ramified extension over $D = \Div z + \Div f$. Further, for both (regular) prime divisors $(z), (f) \subset R$ there are exactly two prime of $S$ lying over with ramification and inertia indexes equal to $2$.
\end{example}

\begin{remark}[Failure of Abhyankar's lemma for divisors without normal crossings]
Observe that \autoref{ex.TheCusp}, \autoref{ex.WhitneyUmbrella}, and \autoref{ex.Parabola} are counterexamples for Abhyankar's lemma if no regularity condition is imposed on the divisor. Indeed, in each case, we may consider $R$ to be additionally strictly local, then it admits a tamely ramified cover (e.g. $S$) that is not Kummer (for it is not cyclic). In the cusp case, the divisor $D$ has not normal crossings for it is cut out by a singular (irreducible) equation. In the Whitney's umbrella case, the divisor has not normal crossings because $f$ is not a regular element in the ring $R^{\mathrm{sh}}_{(x,y)}$ as $f=x^2-y^2z = (x-y\sqrt{z})(x+y\sqrt{z})$ in this ring. In the case of \autoref{ex.Parabola}, we have that $z$ and $f$ are both regular elements yet $R/(z,f)$ is not regular as $(z,f)=(z,x^2)$.
\end{remark}

\subsection{Main formal theorems}
\label{sect.MainFormal}
Next, we explain why our main results on $\pi_1^{\mathrm{t},P}(X^{\circ})$ can be seen as formal consequences of some interesting properties of the Galois category $\mathsf{Rev}^{P}(X^{\circ})$. With notation as in \autoref{setupTameGroups}, to a Galois object $f\: Y^{\circ} \to X^{\circ}$ in $\mathsf{Rev}^{P}(X^{\circ})$ of degree $d_f$ we may associate three positive integers $n_f$, $e_f$, and $i_f$ which are subject to the relation $d_f=n_f\cdot e_f \cdot i_f$; see \autoref{pro.nefFormular}. With this in mind:

\begin{terminology} \label{term.MAnyTerms} Working in \autoref{setupTameGroups}, we consider the following properties on $\mathsf{Rev}^{P}(X^{\circ})$.
\begin{enumerate}
    \item \emph{$P$-irreducibility}: Every connected cover $f\: Y^{\circ} \to X^{\circ}$ in $\mathsf{Rev}^{P}(X^{\circ})$ satisfies that $Q\coloneqq \bigl(f^{-1}(P)\bigr)_{\mathrm{red}}$ is a prime divisor on $Y^{\circ}$. In other words, with notation as in \autoref{rem.ReductionLocalAlgebra}, there is exactly one prime, say $\q$, lying over $\p$ in the extension $R\subset S$. If $f$ is Galois, this means $n_f = 1$.
    \item \emph{inertial boundedness}: There exists $N \in \bN$ such that $i_f \leq N$ for all Galois objects $f\: Y^{\circ} \to X^{\circ}$ in $\mathsf{Rev}^{P}(X^{\circ})$.
    \item \emph{inertial tameness}: The inertial degree $i_f$ is prime-to-$p$ for all Galois objects $f\: Y^{\circ} \to X^{\circ}$ in $\mathsf{Rev}^{P}(X^{\circ})$.
    \item \emph{inertial decantation}: Assuming the $P$-irreducibility of $\mathsf{Rev}^{P}(X^{\circ})$, inertial decantation means that every Galois cover $f\: Y^{\circ} \to X^{\circ}$ in $\mathsf{Rev}^{P}(X^{\circ})$ dominates a quasi-\'etale Galois cover $Y'^{\circ} \to X^{\circ}$ in $\mathsf{Rev}^{P}(X^{\circ})$ whose generic degree is the generic degree of $Q\coloneqq \bigl(f^{-1}(P)\bigr)_{\mathrm{red}} \to P$. Equivalently, with notation as in \autoref{rem.ReductionLocalAlgebra}, if $\q \subset S$ is the only prime lying over $\p$, there is a factorization $(R, \fram, \kay,K;\p) \subset (S', \fran', \el',L';\q') \subset (S, \fran, \el,L;\q) $ such that the bottom extension induces an \'etale-over-$P$ (i.e. quasi-\'etale) cover in $\mathsf{Rev}^{P}(X^{\circ})$ and $[\kappa(\q'): \kappa(\p)]=[L':K]=[\kappa(\q):\kappa(\p)]$. When the latter degree is $1$, we say that $f$ is \emph{totally ramified}.
\end{enumerate}
\end{terminology}

\begin{remark} \label{rem.CategoricalBehaviour} If $\mathsf{Rev}^{P}(X^{\circ})$ is $P$-irreducible, we may think of the covers in $\mathsf{Rev}^{P}(X^{\circ})$ as local extensions $(R,\fram,\kay,K;\p)\subset (S,\fran,\el,L;\q)$; as in \autoref{rem.ReductionLocalAlgebra}, where $\q$ is the only (height-$1$) prime ideal of $S$ lying over $\p$. We follow the convention to denote the prime divisor corresponding to $\q$ by $Q$ and so on. Note that, if $f\: Y^{\circ} \to X^{\circ}$ is a connected cover in $\mathsf{Rev}^{P}(X^{\circ})$, then the category $\mathsf{Rev}^{Q}(Y^{\circ})$ is $Q$-irreducible and $\mathsf{Rev}^{Q}(Y^{\circ})$ is the Galois category given by the objects of $\mathsf{Rev}^{P}(X^{\circ})$ that lie over (or dominate) the object $(Y^{\circ}, Q)$. If $f$ is further Galois, then $\mathsf{Rev}^{Q}(Y^{\circ})$ is inertially bounded (resp. tame) if so is $\mathsf{Rev}^{P}(X^{\circ})$.
\end{remark}

\begin{lemma} \label{lem.Lemma(a')}
$P$-irreducibility implies inertial decantation.
\end{lemma}
\begin{proof}
With notation as in \autoref{rem.ReductionLocalAlgebra}, since $\q$ is the only prime lying over $\p$, its decomposition group $D\coloneqq \{\sigma \in \Gal(L/K) \mid \sigma(\q) = \q\}$ is the whole Galois group $\Gal(L/K)$. Therefore, its inertia group $I$ sits as the kernel in the following short exact sequence of groups
\[
1 \to I \to \Gal(L/K) \to \Aut\big(\kappa(\q)/\kappa(\p)\big) \to 1.
\]
By the tameness of the ramification, $\kappa(\q)/\kappa(\p)$ is a finite separable extension and so Galois by \cite[\href{http://stacks.math.columbia.edu/tag/09ED}{Lemma 09ED}]{stacks-project}. Thus, we have
\[
1 \to I \to \Gal(L/K) \to \Gal \big(\kappa(\q)/\kappa(\p)\big) \to 1.
\]
We may use the Galois correspondence to obtain a factorization
\[
(R,\fram, \kay, K, P) \subset \big(S^I, \fran^I, \el^I, L^{I}, Q^I\big) \subset (S, \fran, \el, L, Q)
\]
where both covers are Galois and the upper script $I$ denotes the invariant or fixed elements under the action of $I$. Moreover, $\Gal\big(L^{I}/K\big) = \Gal \big(\kappa(\q)/\kappa(\p)\big)$ and the bottom extension is \'etale at $\q^I$; see \cite[\href{http://stacks.math.columbia.edu/tag/09EH}{Lemma 09EH}]{stacks-project}, and so quasi-\'etale. Furthermore, $\big[\kappa\bigl(\q^I\bigr):\kappa(\p)\big] = \big[L^I:K\big]=\big[\kappa(\q):\kappa(\p)\big]$. This proves the lemma.
\end{proof}

\begin{defprop}
In the situation of \autoref{setupTameGroups}, the full subcategory $\mathsf{Rev}^{P}_{\mathrm{\acute{e}t}}(X^{\circ})$ of $\mathsf{Rev}^{P}(X^{\circ})$ consisting of those $Z^\circ \to X^\circ$ that are \'etale-over-$P$ is a Galois subcategory. We denote the corresponding fundamental group by $\pi^{P}_{\mathrm{1,\acute{e}t}}(X^{\circ})$.
\end{defprop}
\begin{proof}
We follow the proof of \cite[Theorem 2.4.2]{GrothendieckMurreTameFundamentalGroup} and only need to verify the conditions G1, G2 and G3 of \cite[Expos\'e V, 4]{GrothendieckSGA}. Clearly, $X^\circ$ itself is a final object. For the existence of fiber products, take $Y^\circ \to Z^\circ, W^\circ \to Z^\circ$ in $\mathsf{Rev}(X^\circ)$ and consider the following diagram
\[\begin{xy} \xymatrix{(Y^\circ \times_{Z^\circ} W^\circ)_{\mathrm{nor}} \ar[r] & Y^\circ \times_{Z^\circ} W^\circ \ar[r] \ar[d] & Y^\circ \ar[d] \\ & W^\circ \ar[r]& Z^\circ} \end{xy} \] where the normalization is taken with respect to the total ring of fractions of $Y^\circ \times_{Z^\circ} W^\circ$. By \cite[loc.\ cit.]{GrothendieckMurreTameFundamentalGroup}, this is the fiber product in $\mathsf{Rev}^{P}(X^{\circ})$. Note that $Y^\circ \times_{Z^\circ} W^\circ$ is \'etale over $P$ since \'etale morphisms are stable under base change. Moreover, as \'etale morphisms preserve normality (and $X^\circ$ is normal), we conclude that $Y^\circ \times_{Z^\circ} W^\circ \to X^\circ$ is normal at $P$ and thus the normalization is an isomorphism at $P$. The existence of direct sums is clear. Consider now $Y^\circ \to X^\circ$ a morphism in $\mathsf{Rev}^P(X^\circ)$, $G$ a finite subgroup of $\Aut(Y^\circ)$. Then, 
\[
\begin{xy} \xymatrix{Y^\circ \ar[r] \ar[d]_f &Y^\circ/G \ar[dl]^u\\ X^\circ} \end{xy}
\]
is a commutative diagram in $\mathsf{Rev}^P(X^\circ)$. Assume that $f$ is \'etale over $P$. If $Q$ is a point in $Y^\circ$ lying over $P$ with image $Q'$ in $Y^\circ/G$, we have inclusions of DVRs $\sO_{X^\circ, P} \subset \sO_{Y^\circ/G, Q'} \subset \sO_{Y^\circ, Q}$. Since the inclusion $\sO_{X^\circ, P} \subset \sO_{Y^\circ, Q}$ is unramified, the first extension is also unramified. Hence, $u$ is \'etale at $P$. Condition G3 follows just as in \cite[loc.\ cit.]{GrothendieckMurreTameFundamentalGroup}
\end{proof}

\begin{lemma} \label{lem.UniversalCover}
Work in \autoref{setupTameGroups}. Suppose that $\mathsf{Rev}^{P}(X^{\circ})$ is $P$-irreducible. Then, $\mathsf{Rev}^{P}(X^{\circ})$ inertially bounded if and only if $\mathsf{Rev}^{P}_{\mathrm{\acute{e}t}}(X^{\circ})$ has a universal cover, i.e. $\pi^{P}_{\mathrm{1,\acute{e}t}}(X^{\circ})$ is finite. If $\tilde{f} \:\tilde{X}^{\circ} \to X^{\circ}$ is such universal cover, then $\mathsf{Rev}^{\Tilde{P}}(\Tilde{X}^{\circ})$ is so that its Galois objects have inertia degree equal to $1$. If $\mathsf{Rev}^{P}(X^{\circ})$ is further inertially tame, then the order of $\pi^{P}_{\mathrm{1,\acute{e}t}}(X^{\circ})$ is prime-to-$p$. 
\end{lemma}
\begin{proof}
Since $\mathsf{Rev}^{P}(X^{\circ})$ is $P$-irreducible, the degree of a Galois object in $\mathsf{Rev}^{P}_{\mathrm{\acute{e}t}}(X^{\circ})$ coincides with its inertial degree. The first and third statements then follow. The second statement follows from the inertial decantation on $\mathsf{Rev}^{\Tilde{P}}(\Tilde{X}^{\circ})$; see \autoref{lem.Lemma(a')} and \autoref{rem.CategoricalBehaviour}.
\end{proof}

By purity of the branch locus, we only need to check inertial boundedness and/or tameness on the regular locus of $X$, \ie $X^\circ = X_\mathrm{reg}$. This will play a crucial role in \autoref{sec.CharZeroBusiness}. We make this precise next but it can be skipped for now. 

\begin{proposition} \label{pro.RoleOfRegularLocus}
Work in \autoref{setupTameGroups}. There is a fully faithful functor between Galois categories $\mathsf{Rev}^{P}_{\mathrm{\acute{e}t}}(X^{\circ}) \to \mathsf{F\acute{E}t}(X_{\mathrm{reg}})$, which induces a surjective homomomorphism between the corresponding fundamental groups. Moreover, this functor induces an isomorphism between fundamental groups if $Z$ cuts out the singular locus of $X$.
\end{proposition}
\begin{proof}
Recall that $\mathsf{F\acute{E}t}(X_{\mathrm{reg}})$ is equivalent to the Galois subcategory of the absolute Galois category of $K$ given by finite separable extensions $K \subset L \subset K^{\mathrm{sep}}$ such that the integral closure of $R \subset R^L$ in $L$ is \'etale over $X_{\mathrm{reg}}$; see \cite[\S 2.4]{CarvajalSchwedeTuckerEtaleFundFsignature}. On the other hand, as mentioned before in \autoref{rem.ReductionLocalAlgebra}, $\mathsf{Rev}^{P}(X^{\circ})$ corresponds the the Galois subcategory given by field extensions where $R\subset R^{L}$ is \'etale over $U$ and $L/K$ is tamely ramified with respect to $R_{\p}$, whereas $\mathsf{Rev}^{P}_{\mathrm{\acute{e}t}}(X^{\circ})$ is the one in which $R \subset R^L$ is \'etale over $U$ and $L/K$ is \'etale with respect to $R_{\p}$; see \autoref{def.TameRamification}. In particular, $\mathsf{Rev}^{P}_{\mathrm{\acute{e}t}}(X^{\circ})$ is (or can be identified with) a full Galois subcategory of $\mathsf{F\acute{E}t}(X_{\mathrm{reg}})$. Indeed, if $L/K$ is in $\mathsf{Rev}^{P}_{\mathrm{\acute{e}t}}(X^{\circ})$ then $R \subset R^{L}$ is quasi-\'etale, and so induces an \'etale cover over $X_{\mathrm{reg}}$ by Zariski--Nagata--Auslander purity of the branch locus for regular schemes \cite[\href{https://stacks.math.columbia.edu/tag/0BMB}{Lemma 0BMB}]{stacks-project}, \cf \cite{ZariskiPurity,NagataPurity,NagataPurityII,AuslanderPurity}. Moreover, if $X^{\circ} = X_{\mathrm{reg}}$ (i.e. $Z$ cuts out the singular locus), we have the same categories as in that case $U \subset X_{\mathrm{reg}}$ and $X_{\mathrm{reg}}$ contains the regular point of $P$. It is worth noticing that the normality of $X$ is essential through the previous arguments. Finally, observe that the remaining statements are formal consequences of the just proven; see \cite[Chapter 5]{MurreLecturesFundamentalGroups}.
\end{proof}

\begin{corollary}
\label{le.b'-condition}
Work in \autoref{setupTameGroups}. Suppose that $\mathsf{Rev}_{\mathrm{\acute{e}t}}^{P}(X_\mathrm{reg})$ has a universal cover (of prime-to-$p$ degree). Then, $\mathsf{Rev}_{\mathrm{\acute{e}t}}^{P}(X^\circ)$ has a universal cover (of prime-to-$p$ degree).
\end{corollary}
\begin{proof}
\autoref{pro.RoleOfRegularLocus} can be summarized as follows: $\pi^{P}_{1,\mathrm{\acute{e}t}}(X_{\mathrm{reg}}) \cong \pi_1^{\mathrm{\acute{e}t}}(X_{\mathrm{reg}}) \twoheadrightarrow \pi^{P}_{\mathrm{1,\acute{e}t}}(X^{\circ})$. Thus, finiteness/tameness on the left-hand side group implies finiteness/tameness on the right-hand one.
\end{proof}

\begin{definition}
Work in \autoref{setupTameGroups}. Define $N^P(X^{\circ}) \subset \bN$ as the set of prime-to-$p$ positive integers $n \in \bN$ for which there is a divisor $D$ on $X$ such that $P-n \cdot D \in \Cl X$ has prime-to-$p$ torsion and $D|_U$ is Cartier. Likewise, define $M^P(X^{\circ}) \subset \bN$ as the set of prime-to-$p$ positive integers $n \in \bN$ for which there is a divisor $D$ on $X$ such that $P=n \cdot D \in \Cl X$ and $D|_U$ is Cartier. Note that $1 \in M^P(X^{\circ}) \subset N^P(X^{\circ})$.
\end{definition}

With the above in place, we have the following theorem.

\begin{theorem} \label{thm.MainFormal}
Work in \autoref{setupTameGroups}. Suppose that $\mathsf{Rev}^{P}(X^{\circ})$ is $P$-irreducible and has bounded inertia. Then, there exists a short exact sequence of topological groups
\begin{equation} \label{eqn.CokernelSequence}
\hat{\mathbb{Z}}^{(p)}\to \pi_1^{\mathrm{t},P}(X^{\circ}) \to \pi^{P}_{\mathrm{1,\acute{e}t}}(X^{\circ}) \to 1
\end{equation}
where $\pi^{P}_{\mathrm{1,\acute{e}t}}(X^{\circ})$ is  finite and $\hat{\mathbb{Z}}^{(p)}$ is the prime-to-$p$ part of the profinite completion of $\Z$ (if $p=0$ we shall agree upon $\hat{\mathbb{Z}}^{(p)}\coloneqq \hat{\Z}$). If $\mathsf{Rev}^{P}(X^{\circ})$ also has tame inertia, then the order of $\pi^{P}_{\mathrm{1,\acute{e}t}}(X^{\circ})$ is prime-to-$p$ and the following two statements hold:
 \begin{itemize}
    \item If $P \in \Cl X$ has prime-to-$p$ torsion, 
    \autoref{eqn.CokernelSequence} yields a short exact sequence
    \begin{equation} \label{eqn.TorsionCase}
     0 \to \hat{\Z}^{(p)} \to \pi_1^{\textnormal{t},P}(X^{\circ}) \to \pi^{P}_{\mathrm{1,\acute{e}t}}(X^{\circ}) \to 1
    \end{equation}
    which splits (in the category of topological groups) if and only if $M^P(X^{\circ})$ equals the set of prime-to-$p$ positive integers. Further,  the sequence is split if and only if $P = 0 \in \Cl X$. 
    \item If $P \in \Cl X$ is nontorsion, \autoref{eqn.CokernelSequence} yields a short exact sequence
    \begin{equation} \label{eqn.NonTorsionCase}
   0 \to \varprojlim_{ n \in N^P(X^{\circ})
   } \bZ / n\bZ \to \pi_1^{\mathrm{t},P}(X^{\circ}) \to \pi^{P}_{\mathrm{1,\acute{e}t}}(X^{\circ}) \to 1
    \end{equation}
    which is split (in the category of topological groups) if and only if $N^P(X^{\circ}) \supset M^P(X^{\circ})$ is an equality and there is a compatible system $\{\frac{1}{n} P \in \Cl X\}_{ n \in M^P(X^{\circ})}$ of factors of $P$ meaning that $m \cdot (\frac{1}{mn} P) = \frac{1}{n} P$ in $\Cl X$ and $\frac{1}{1} P = P$ (e.g.\ if $\Cl X$ has no prime-to-$p$ torsion). In particular, if $N^P(X^{\circ})$ is finite (e.g.\ $\Cl X$ modulo prime-to-$p$ torsion is finitely generated), there is a short exact sequence
    \begin{equation} \label{eqn.NontorsionAndFiniteGeneration}
    0 \to \bZ / n\bZ \to \pi_1^{\mathrm{t},P}(X^{\circ}) \to \pi^{P}_{\mathrm{1,\acute{e}t}}(X^{\circ}) \to 1
    \end{equation}
    and so $\pi_1^{\mathrm{t},P}(X^{\circ})$ is finite of order prime-to-$p$. Likewise, \autoref{eqn.NontorsionAndFiniteGeneration} is split if and only if $P \in \Cl X$ is $n$-divisible with $\frac{1}{n} P \in \Pic U$.
\end{itemize}
\end{theorem}

\begin{remark} \label{rem.DirectabilityofNP}
The limit in the kernel of \autoref{eqn.NonTorsionCase} makes sense because, if $m\cdot M= P = n \cdot N$ in $\Cl X$ modulo prime-to-$p$ torsion (with $p \nmid m,n$ and $M|_U,N|_U \in \Pic U$), then $P= l \cdot (a \cdot N + b \cdot M)$ in $\Cl X$ modulo prime-to-$p$ torsion, where: $(m) \cap (n) = (l)$, $am + bn = k$ and $(m,n)=(k)$. In other words, if $m,n$ belong to $N^P(X^{\circ})$ then so does their least common multiple $l$.
\end{remark}

The following two lemmas are well-known to experts but are included for lack of a reference.

\begin{lemma}
\label{lem.StrictHenselizationFiniteMorphism}
Let $\phi\: (R, \mathfrak{m}) \to S$ be a finite extension of normal domains. Denote by $S^{\mathrm{sh}}$ the strict henselization of $S$ with respect to a prime $\mathfrak{n}$ lying over $\mathfrak{m}$. Then, the canonical morphism $\Spec S^{\mathrm{sh}} \to \Spec S \otimes_R R^{\mathrm{sh}}$ is a connected component, in particular a clopen (\ie closed and open) immersion. Furthermore, assume that:
\begin{enumerate}
    \item $(R,\fram)$ is a DVR,
    \item $\phi$ is a generically \'etale extension, and
    \item $\phi_{\fran} \: R \to S_{\fran}$ has trivial residue field extension for all maximal ideals $\fran$ lying over $\fram$.
\end{enumerate}
Then, all the connected components of $\Spec S \otimes_R R^{\mathrm{sh}}$ arise in this way and then are in bijective correspondence with the prime ideals of $S$ lying over $\fram$. 
\end{lemma}
\begin{proof}
By \cite[\href{https://stacks.math.columbia.edu/tag/05WR}{Lemma 05WR}]{stacks-project}, $S^{\mathrm{sh}}$ is obtained as the localization of a prime ideal of $S \otimes_R R^{\mathrm{sh}}$ lying above $\mathfrak{n}$ and $\mathfrak{m}^{\mathrm{sh}}$. Since $S$ is normal and $\Spec S \otimes R^{\mathrm{sh}} \to \Spec S$ is a colimit of \'etale morphisms, $S \otimes R^{\mathrm{sh}}$ is normal by \cite[\href{https://stacks.math.columbia.edu/tag/033C}{Lemma 033C}]{stacks-project} and \cite[\href{https://stacks.math.columbia.edu/tag/037D}{Lemma 037D}]{stacks-project}. Hence, it is a product of normal domains. Moreover, $S \otimes_R R^{\mathrm{sh}}$ is a finite algebra over the henselian local ring $R^{\mathrm{sh}}$ and thus by \cite[\href{https://stacks.math.columbia.edu/tag/04GG}{Lemma 04GG (10)}]{stacks-project} we have 
\[
S \otimes_R R^{\mathrm{sh}} = \prod_{i=1}^m \big( S \otimes_R R^{\mathrm{sh}} \big)_{\fram_i}
\]
where $\fram_1,\ldots, \fram_m$ are the maximal ideals of $S \otimes_R R^{\mathrm{sh}}$ lying over $\fram^{\mathrm{sh}}$. We conclude that $S \otimes_R R^{\mathrm{sh}}$ is a finite product of normal local domains. Since any prime of $S \otimes_R R^{\mathrm{sh}}$ lying above $\mathfrak{n}$ and $\mathfrak{m}^{\mathrm{sh}}$ is necessarily maximal,  $\Spec S^{\mathrm{sh}} \to \Spec S \otimes_R R^{\mathrm{sh}}$ is a clopen immersion.

Finally, we discuss the statement regarding the case $(R,\fram)$ is a DVR. Set $(u) = \fram$. In this case, $S$ is a semi-local Dedekind domain and in particular a PID; let $\fran_1,\ldots,\fran_n$ be the maximal ideals of $S$ lying over $\fram$. Let $K$ be the function field of $R$, so that $\Spec K \to \Spec R$ defines the open immersion given by the principal open $D(u)=\Spec R_u$. Observe that $\big(R^{\mathrm{sh}}, \theta(u)\big)$ is a (strictly henselian) DVR as well, where $\theta$ is the canonical homomorphism $R \to R^{\mathrm{sh}}$. With this being said, we see that pullback of the cartesian square
\[
\xymatrix{
\Spec S \ar[d] & \Spec S \otimes_R R^{\mathrm{sh}} \ar[d]  \ar[l] \\
\Spec R  & \Spec R^{\mathrm{sh}} \ar[l]
}
\]
to the Zariski open $\Spec K \to \Spec R $ is given by the cartesian square
\[
\xymatrix{
\Spec L  \ar[d] & \Spec L \otimes_K K\big(R^{\mathrm{sh}}\big) \ar[d] \ar[l]\\
\Spec K  & \Spec K\big(R^{\mathrm{sh}}\big) \ar[l]
}
\]
where $L$ denotes the fraction field of $S$. In particular, the generic rank of the finite $R^{\mathrm{sh}}$-algebra $S \otimes_R R^{\mathrm{sh}}$ is equal to $[L:K]$---the generic rank of $S$ over $R$. In particular, we have that
\[
[L:K] = \sum_{i=1}^m \Big[ K\Big(  \big( S \otimes_R R^{\mathrm{sh}} \big)_{\fram_i}\Big) : K\big(R^{\mathrm{sh}}\big) \Big] = \sum_{i=1}^n \big[K\big(S_{\fran_i}^{\mathrm{sh}}\big):K\big(R^{\mathrm{sh}}\big)\big] + \Sigma,
\]
where $\Sigma$ is the remaining summands, i.e. the sum corresponding to the (\emph{a priori} possible) connected components that are not isomorphic to strict henselizations of $S$ at some of its maximal ideals. Our goal is to prove that $\Sigma = 0$ (i.e. it is an empty summation). To this end, observe that, by combining assumptions (a) and (c) with \cite[\href{https://stacks.math.columbia.edu/tag/09E8}{Remark 09E8}]{stacks-project}, we have $[L:K] = \sum_{i=1}^n e_i $ where $e_i$ is the ramification index of the extension of DVRs $\phi_{\fran_i} \: R \to S_{\fran_i}$. Hence, it suffices to prove $
\big[K\big(S_{\fran_i}^{\mathrm{sh}}\big):K\big(R^{\mathrm{sh}}\big)\big] = e_i$. Observe that the ramification index of $R^{\mathrm{sh}} \to S^{\mathrm{sh}}_{\fran_i}$ is exactly $e_i$ and its residue field extension is trivial (it is tacitly assumed here that the residue field of both is the same separable closure of $R/\fram = S_{\fran_i}/\fran_iS_{\fran_i}$). The result then follows from \cite[\href{https://stacks.math.columbia.edu/tag/09E8}{Remark 09E8}]{stacks-project}.
\end{proof}

\begin{example} \label{ex.UsingLemmaConnectedComponents}
We may use \autoref{lem.StrictHenselizationFiniteMorphism} to argue the part in the proof of \autoref{cla.TameCoverCusp} where we explain why there cannot be only one prime of $S$ lying over $(\Delta)$. Indeed, if there were only one such a prime $\q \subset S$, we saw that the degree-$6$ extension of DVRs $R_{(\Delta)} \to S_{\q}$ has ramification index $6$ and Galois group isomorphic to $S_3$. However, when we apply \autoref{lem.StrictHenselizationFiniteMorphism} and its proof we obtain that $R_{(\Delta)}^{\mathrm{sh}} \to S_{\q}^{\mathrm{sh}}$ is a degree-$6$ extension with Galois group $S_3$. Nevertheless, this contradicts \autoref{theo.serretameiskummer} as it states that the Galois group must be cyclic.

Either directly or indirectly, we know that there must be three primes $\q_1=(t_2-t_3)$, $\q_2=(t_1-t_3)$, and $\q_3=(t_1-t_3)$ of $S$ lying over $(\Delta) \subset R$, all of them with ramification index $2$ and inertia degree $1$. As predicted by \autoref{lem.StrictHenselizationFiniteMorphism}, we can see directly that
\[
S \otimes_R R_{(\Delta)}^{\mathrm{sh}} \cong S_{\q_1}^{\mathrm{sh}} \times S_{\q_2}^{\mathrm{sh}} \times S_{\q_3}^{\mathrm{sh}},
\]
where each extension $R_{(\Delta)}^{\mathrm{sh}} \subset S_{\q_i}^{\mathrm{sh}}$ is a degree-$2$ Kummer extension of strictly local DVRs. Indeed, denoting $\varpi_i \coloneqq 3t_i^2+x$, we had from \autoref{ex.TheCusp} that $R[\delta,t_i]_{\varpi_i} = S_{\varpi_{i}}$, and moreover $\Spec S = \bigcup_{i=1}^3 \Spec S_{\varpi_i}$ where $\q_i \in \Spec S_{\varpi_j}$ if and only if $i=j$ (this follows from \autoref{eqn.DifferencesOft_i's} and the argument in the succeeding paragraph). This is nothing but an open covering of $\Spec S$ by standard \'etale morphisms over $\Spec R[\delta]$ \cf \cite[I, Theorem 3.14]{MilneEtaleCohomology}. In fact, the morphisms $\Spec_{\varpi_i} \to \Spec R[\delta]$ are \'etale neighborhoods of $(\delta) \subset \Spec R[\delta]$. In particular, the canonical homomorphism $S_{\varpi_i} \to S_{\q_i}$ is an isomorphism when twisted by $R[\delta]_{(\delta)}^{\mathrm{sh}}$---the strict henselization of $R[\delta]$ at $(\delta)$---which is then canonically isomorphic to each of $S_{\q_i}^{\mathrm{sh}}$. Finally, one verifies directly that the canonical homomorphism $R[\delta]\otimes_R R^{\mathrm{sh}}_{(\Delta)} \to R[\delta]_{(\delta)}^{\mathrm{sh}}$ is an isomorphism.

Finally, we point out that hypothesis (c) in \autoref{lem.StrictHenselizationFiniteMorphism} is (trivially) crucial for the proposition to hold. Indeed, suppose that $R \to S$ is a finite \'etale extension of DVRs (i.e. $n,e=1$ in \autoref{pro.nefFormular}). Then, the generic and inertia degrees coincide; denote them by $d$. However, $S \otimes_R R^{\mathrm{sh}}$ is product of $d$ copies of $R^{\mathrm{sh}}$. Roughly speaking, we get $d$ connected components of $\Spec S \otimes_R R^{\mathrm{sh}}$ out of just one prime lying over the maximal ideal of $R$ (both are a degree-$2$ Kummer cover $R^{\mathrm{sh}}_{(\delta)}$ with respect to $\Div \Delta$). This concludes the example. 
\end{example}

\begin{lemma}\label{lem.Showingprincipal}
Let $f\: Y \to X$ be a degree-$d$ finite cover of normal integral schemes. Suppose that $f_* \sO_Y$ is locally free on some big open $U \subset X$ (i.e., $X \smallsetminus U$ has codimension $\geq 2$). Then, the kernel of $f^* \: \Pic X \to \Pic Y$ is $d$-torsion. In particular, $f^*$ maps nontorsion elements into nontorsion elements.
\end{lemma}
\begin{proof}
Let $\sL$ be an invertible sheaf on $X$ such that $f^* \sL \cong \sO_Y$. Then, $f_* f^* \sL \cong f_* \sO_Y$. Nonetheless, $f_*f^* \sL \cong f_*\bigl(\sO_Y \otimes f^* \sL \bigr) \cong \sL \otimes f_* \sO_Y$ by the projection formula. Hence, we have an isomorphism $\sL \otimes f_*\sO_Y \cong f_* \sO_Y$. Note that the rank of $f_* \sO_Y$ is $d$. By letting $V=f^{-1}(U)$ and taking determinants we have $\det f_*\sO_V \cong \det \bigl(\sL_U \otimes f_* \sO_V\bigr) = \sL_U^d \otimes \det f_* \sO_V$. Therefore, $\sL_U^d \cong \sO_U$. Since $X$ is normal and $\codim X \smallsetminus U \geq 2$, we conclude that $\sL^d \cong \sO_X$. 
\end{proof}

The following observation plays a crucial role in our main theorem. It can be thought of as a singular Abhyankar's lemma for prime divisors.

\begin{lemma} \label{lem.THELEMMAFORABH}
Work in \autoref{setupTameGroups} and suppose $\mathsf{Rev}^{P}(X^{\circ})$ to be $P$-irreducible. Then, every totally ramified cyclic Galois object in $\mathsf{Rev}^{P}(X^{\circ})$ of prime-to-$p$ degree $n$ is  (up to isomorphism) a Kummer-type cyclic cover (as in \autoref{ex.KummerTypeCyclicCovers}) of degree $n \in M^P(X^{\circ})$. In particular, if $P=0 \in \Cl X$; say $\p=(f)$, and $\Cl X$ has no prime-to-$p$ torsion, then every totally ramified Galois cover in $\mathsf{Rev}^{P}(X^{\circ})$ is isomorphic to a Kummer cover of the form $\Spec \sO_{X^{\circ}}[T]/(T^n-f) \to X^{\circ}$ (with $n$ prime to the characteristic).
\end{lemma}
\begin{proof}
Let $(R,\fram,\kay, K, P) \subset (S,\fran,\kay,L,Q)$ in $\mathsf{Rev}^{P}(X^{\circ})$ be a totally ramified cyclic Galois cover of degree $n$ with $p\nmid n$.  Its pullback to $U$ induces a connected $\bZ/n\bZ$-torsor $V \to U$ (where $V=Y^{\circ} \smallsetminus Q$ and $U=X^{\circ} \smallsetminus P$). Therefore, by Kummer theory \cite[III, \S4]{MilneEtaleCohomology} and using that $R$ contains a primitive $n$-th root of unity, there exists an $n$-torsion Cartier divisor $D$ on $U$ such that $f_U \: V \to U$ is the cyclic cover defined by the spectrum of the finite $\sO_U$-algebra $\sA \coloneqq \bigoplus_{i=0}^{n-1}{\sO_U(-iD)}$ defined via a global section $\sO_U \xrightarrow{\cong} \sO_U(nD)$, say $\kappa \in K^{\times}$ such that $\Div_U \kappa + nD = 0$ (and so $\cdot \kappa \:\sO_U(-nD) \to \sO_U$). We write $L=\sA_K=K \otimes \sA = K(\kappa^{1/n})$. Let $\bar{D}$ be the closure of $D$ in $X$. Thus, $\bar{D}$ is a Weil divisor on $X$ such that $\bar{D}|_{U} = D$ and $\Div_X \kappa + n \bar{D} = e \cdot P$ where $e = \val_P \kappa$. We may assume that $e \geq 0$ (by replacing both $\kappa$ and $D$ by their respective inverses if necessary). That is, we may assume $\kappa \in R_{\p}$. Note that $ \Div_V \kappa^{1/n} + f_U^* D = 0$ on $V$, which is obtained by dividing by $n$ the pullback of $\Div_U \kappa + nD = 0$. Further, $n \cdot \val_Q \kappa^{1/n} =  \val_Q \kappa = n \cdot \val_P \kappa$ (using the hypothesis that the extension is totally ramified) and so $\val_Q \kappa^{1/n} = e \geq 0$. This lets us conclude the following. 
\begin{claim}
$(n,e)=1$
\end{claim}
\begin{proof}[Proof of claim]
Consider the subextension $R_{\p} \subset R_{\p}[\kappa^{1/n}] \subset S_{\p} = S_{\q}$. Note that $R_{\p} \subset R_\p[\kappa^{1/n}]$ is a free local extension of rank $n$, where the maximal ideal of $R_\p[\kappa^{1/n}]$ is $(t,\kappa^{1/n})$ with $t$ a uniformizer of $R_{\p}$. Also, note that $R_\p[\kappa^{1/n}] \subset S_{\q}$ is birational. In particular, $S_{\q}$ is the normalization of $R_{\p}[\kappa^{1/n}]$. However, this can only happen if $(n,e)=1$. Indeed, we may base change $R_{\p} \subset R_{\p}[\kappa^{1/n}] \subset S_{\q}$ by $R_{\p}^{\mathrm{sh}}$ to obtain $R_{\p}^{\mathrm{sh}} \subset R_{\p}^{\mathrm{sh}}[\kappa^{1/n}] \subset S_{\q}^{\mathrm{sh}}$ using \autoref{lem.StrictHenselizationFiniteMorphism} (c). Nonetheless, 
\[
R_{\p}^{\mathrm{sh}}[\kappa^{1/n}] \to \big(R[\kappa^{1/n}]\big)_{(t,\kappa^{1/n})}^{\mathrm{sh}}
\] is an isomorphism as $R_{\p}^{\mathrm{sh}}[\kappa^{1/n}]$ is a strictly local algebra over $R[\kappa^{1/n}]$ (use \cite[I, Corollary 4.3]{MilneEtaleCohomology}) with the same residue field as $R_{\p}^{\mathrm{sh}}$. Therefore, $R_{\p}^{\mathrm{sh}}[\kappa^{1/n}] \subset S_{\q}^{\mathrm{sh}}$ is a normalization by \cite[\href{https://stacks.math.columbia.edu/tag/0CBM}{Tag 0CBM}]{stacks-project}. However, $R_{\p}^{\mathrm{sh}}[\kappa^{1/n}] \cong R_{\p}^{\mathrm{sh}}[T]/(T^n - \kappa ) \cong R_{\p}^{\mathrm{sh}}[T]/(T^n - t^e)$ using that $\kappa = u \cdot t^e$ for some unit $u \in R_{\p}$ (the latter isomorphism is of course $T \leftrightarrow T/u^{1/n}$). Hence, $(n,e)=1$ for $R_{\p}^{\mathrm{sh}}[\kappa^{1/n}]$ is a domain. 
\end{proof}
Thus, there are $a,b \in \bZ$ such that $1=an+be$ and so:
\[
P = (an +b e) \cdot P = n\big(a\cdot P+b \cdot\bar{D}\big) + \Div_X \kappa^b.
\] 
Further, $\big(a\cdot P+b \cdot\bar{D}\big)\bigm|_U = b \cdot D \in \Pic U$ and so $n \in M^P(X^{\circ})$. Now, the above establishes that $K(\kappa^{b/n})/K$ defines (after taking integral closure) an object in $\mathsf{Rev}^{P}(X^{\circ})$ that is a cyclic cover of Kummer-type. However, $L=K(\kappa^{1/n})= K(\kappa^{b/n})$ as $(b,n)=1$ (for $1=an+be$). Then, $S/R$ in $\mathsf{Rev}^{P}(X^{\circ})$ is a cyclic cover of Kummer-type and $n \in M^P(X^{\circ})$; as required. 

For the last statement when $\p = (f)$, see \autoref{ex.KummerTypeCyclicCovers}.
\end{proof}

\begin{proof}[Proof of \autoref{thm.MainFormal}]
We will subdivide the proof in two parts. First, we prove the statements of \autoref{thm.MainFormal} except for those establishing when the short exact sequences split. Once this is done, we proceed to show the statements of \autoref{thm.MainFormal} characterizing when the given short exact sequences split.  

We start with the first part now. By formal properties of Galois categories, we obtain from \autoref{lem.UniversalCover} a short exact sequence of topological groups:
\[
1\to \pi_1^{\mathrm{t}, \tilde{P}}\bigl(\tilde{X}^{\circ}\bigr) \to \pi_1^{\mathrm{t},P}(X^\circ) \to \Gal\bigl(\tilde{X}^{\circ}/X^{\circ}\bigr) \to 1
\]
where 
$G \coloneqq \pi^{P}_{\mathrm{1,\acute{e}t}}(X^{\circ}) = \Gal\bigl(\tilde{X}^{\circ}/X^{\circ}\bigr)$. Let $d$ be its order. 
\begin{claim} \label{cla.TorsionImpliestrivialreduction}
By replacing $X$ by $\tilde{X}$, we may assume that: $G$ is trivial, Galois objects have inertial degree equal to $1$, if $P$ is prime-to-$p$ torsion then it is trivial, and further $N^P(X^{\circ})=M^P(X^{\circ})$. 
\end{claim}
\begin{proof}[Proof of claim]
Consider the induced homomorphism $\tilde{f}^{*} \: \Cl X \to \Cl \tilde{X}$. Then, $\tilde{f}^*\:P \mapsto \tilde{P}$. Since $\tilde{f}$ is quasi-\'etale, its restriction to $X^{\circ}_{\mathrm{reg}}$ is a Galois \'etale cover by the purity of the branch locus. By \autoref{lem.Showingprincipal}, $\ker \tilde{f}^{*}$ is $d$-torsion and $\tilde{P}$ is nontorsion if so is $P$ (as $\Cl X$ is the same as $\Pic U$ for any regular big open $U \subset X$).

Let $\phi \: \Cl \tilde{X} \to \Cl \Tilde{U}$ be the restriction homomorphism and set $\Lambda \coloneqq \phi^{-1}\big( \Pic \Tilde{U}\big)$ (where $\tilde{U} = \tilde{X}^{\circ} \smallsetminus \tilde{P}$ and so on). We observe that $\Lambda$ has no prime-to-$p$ torsion. Indeed, let $D$ be a divisor on $\tilde X$ such that $D \in \Lambda$ and $D \in \Cl \Tilde{X}$ has prime-to-$p$ torsion, say of index $p \nmid n > 1$. Then, the corresponding Veronese-type cyclic cover $\tilde{R} \to \bigoplus_{i=0}^{n-1}\Tilde{R}(-iD)$ induces a quasi-\'etale degree-$n$ Galois object in $\mathsf{Rev}^{\tilde{P}}(\tilde{X}^{\circ})$, which violates the universality of $\tilde{f}$.

Now, if $P \in \Cl X$ is torsion, then so is $\tilde{P} \in \Cl \Tilde{X}$ and its order divides the one of $P$. Thus, if $P \in \Cl X$ is prime-to-$p$ torsion then $\Tilde{P} \in \Cl \tilde{X}$ is trivial as $\Tilde{P} \in \Lambda$.

Finally, we must explain why
\[
N^P(X^{\circ}) = \big\{n \bigm| p\nmid n, \, \Tilde{P} = n \cdot N \text{ with } N \in \Lambda \big\} = N^{\tilde{P}}(\tilde{X}^{\circ})
\]
The second equality follows from $\Lambda$ having no prime-to-$p$ torsion as $\Tilde{P}- n \cdot N \in \Lambda$ if $N \in \Lambda$. The first equality is obtained as follows. The inclusion ``$\subset$'' is clear. Indeed, if $P = n \cdot M + T$ in $\Cl X$ where $T \in \Cl X$ has prime-to-$p$ torsion and $M|_U \in \Pic U$. Then, $\tilde{P}=n \cdot \tilde{f}^* M + \tilde{f}^* T$ where $\tilde{f}^* M \in \Lambda$ and $\tilde{f}^* T = 0$ as it is a prime-to-$p$ torsion element of $\Lambda$. Conversely, suppose that $\tilde{P}= n \cdot N$ in $\Cl \tilde{X}$ for some $N \in \Lambda$. Let us pullback everything to $W \coloneqq X^{\circ}_{\mathrm{reg}}$ whose inverse image under $\tilde{f}$ we denote by $\tilde{W}$, which is a regular big open of $\Tilde{X}$. Then, by using the Hochschild--Serre spectral sequence \cite[III, Theorem 2.20]{MilneEtaleCohomology}, the image of $\Tilde{f}^* \: \Pic W \to \Pic \Tilde{W}$ lies inside $(\Pic \Tilde{W})^G$. As in \autoref{lem.Showingprincipal}, we may consider the norm homomorphism $\Norm_{\tilde{f}} \: \Pic \tilde{W} \to \Pic W$, which is obtained by applying $H^1$ to the norm morphism of multiplicative groups $f_*\sO_{\tilde{W}}^{\times} \to \sO_W^{\times}$ (as subsheaves of $ \tilde{K} \supset K$ respectively). See \cite[\href{https://stacks.math.columbia.edu/tag/0BCX}{Tag 0BCX}]{stacks-project} for details. The key property is that the composition 
\[
\Pic W \xrightarrow{\tilde{f}^*} \Pic \tilde{W} \xrightarrow{\Norm_{\tilde{f}}} \Pic W
\] 
is multiplication-by-$d$. However, since $\tilde{f}$ is Galois, $\Norm_{\tilde{f}}\big((\Pic \tilde{W})^G\big) \subset d \cdot \Pic W$ by the same principle and the same for any open of $X$ in place of $W$. Treating the equality $\Tilde{P}= n \cdot N$ in $\Pic \tilde{W}$, on the left hand side we have an element of $(\Pic \Tilde{W})^G$ as $\tilde{P} = \tilde{f}^* P$. Since, $\Lambda$ has no prime-to-$p$ torsion and $n$ is prime-to-$p$, this implies that $N \in (\Pic \tilde{W})^G$. Therefore, by taking norms, we get $d\cdot P = d \cdot n \cdot M$ in $\Pic W$ where $d \cdot M= \Norm_{\Tilde{f}}(N)$, which implies that $M|_{U} \in \Pic U$ as $N|_{\Tilde{U}} \in \Pic \tilde{U}$. Thus, $P = n \cdot M + T$ where $T$ is $d$-torsion and $M|_U \in \Pic U$.
\end{proof}

With the above reductions in place, we let $X' \coloneqq \Spec \sO_{X,P}^{\textnormal{sh}}$ be the \'etale germ of $X$ at (the generic point of) $P$. Note that $\sO_{X,P}^{\textnormal{sh}}$ is none other than the strict henselization of $\sO_{X, P} = R_{\p}$ at its maximal ideal. We argue next that the canonical morphism $X' \to X$ induces a surjection of fundamental groups
\[
\eta \: \pi_1^{\textnormal{t},P'}(X'^{\circ}) \to \pi_1^{\textnormal{t},P}(X^{\circ}) 
\]
where $P'$ is the divisor on $X'$ corresponding to its codimension-$1$ closed point and $X'^{\circ}$ is the inverse image of $X^{\circ}$ along $X' \to X$. 
\begin{claim} \label{cla.TheFundamentalClaim}
The pullback functor $\mathsf{Rev}^{P}(X^{\circ}) \to \mathsf{Rev}^{P'}(X'^{\circ})$ induces a surjective homomorphism of topological groups $\eta \: \pi_1^{\textnormal{t},P'}(X'^{\circ}) \to \pi_1^{\textnormal{t},P}(X^{\circ})$. 
\end{claim}
\begin{proof}[Proof of claim]
Note that the pullback functor is well-defined by \autoref{lem.Lemma228GM71}. By the abstract nonsense regarding Galois categories, the first statement amounts to proving the compatibility between the fiber or fundamental functors; see \cite[Chapter 5]{MurreLecturesFundamentalGroups}. Recall that, implicitly, we always take our base point to be some fixed separable closure $K^{\mathrm{sep}}$ of $K$.  We are going to choose the base point of $\mathsf{Rev}^{P'}(X'^{\circ})$ compatibly, i.e. so that we have a commutative diagram:
\[
\xymatrix{
R_{\p}^{\mathrm{sh}} \ar[r] & K\big(R_{\p}^{\mathrm{sh}}\big) \ar[r] &  K\big(R_{\p}^{\mathrm{sh}}\big)^{\mathrm{sep}}\\
R \ar[u]^-{\theta} \ar[r] & K \ar[u]_-{\theta_K} \ar[r] & K^{\mathrm{sep}} \ar[u]_-{\theta_K^{\mathrm{sep}}} 
}
\]
Equivalently, we choose $K^{\mathrm{sep}}$ to be the subfield of $K\big(R_{\p}^{\mathrm{sh}}\big)^{\mathrm{sep}}$ of elements that are algebraic separable over $K$. To simplify notation, we denote the rings on the top of the diagram from left to right; respectively, by $R'$, $K'$, and $K'^{\mathrm{sep}}$. Recall that the fiber functor $\sF \: \mathsf{Rev}^{P}(X^{\circ}) \to \mathsf{FSet}$ is given by
\[
\sF(S/R)=\Hom_{R\textnormal{-alg}}\big(S,K^{\mathrm{sep}}) = \Hom_{K\textnormal{-alg}}\big(L,K^{\mathrm{sep}}\big)
\]
for all $S/R$ connected in $\mathsf{Rev}^{P}(X^{\circ})$. Of course, the same definition applies to the fiber functor $\sF' \: \mathsf{Rev}^{P'}(X'^{\circ}) \to \mathsf{FSet}$ with $K'^{\mathrm{sep}}$ in place of $K^{\mathrm{sep}}$ and so on. We need to verify the commutativity of the following diagram of functors
\[
\xymatrix{
\mathsf{Rev}^{P}(X^{\circ}) \ar[rr] \ar[rd]_-{\sF} & & \mathsf{Rev}^{P'}(X'^{\circ})\ar[ld]^-{\sF'}\\
& \mathsf{Set}&
}
\]
where the horizontal arrow is the pullback functor; \cf \cite[\S 5.1, Example]{MurreLecturesFundamentalGroups}. To this end, we perform the following computation with $S/R$ a connected object of $\mathsf{Rev}^{P}(X^{\circ})$:
\begin{align*}
    \sF'\big(S\otimes_R R'\big) = \Hom_{R'\textnormal{-alg}}\big(S\otimes_R R',K'^{\mathrm{sep}}\big) = \Hom_{R\textnormal{-alg}}\big(S, K'^{\mathrm{sep}}\big) &= \Hom_{K\textnormal{-alg}}\big(L, K'^{\mathrm{sep}}\big)\\
    &= \Hom_{K\textnormal{-alg}}\big(L, K^{\mathrm{sep}}\big) \\
    & =\sF(S).
\end{align*}
where the penultimate equality follows from the compatibility in our choices of base points. Indeed, since $L/K$ is a finite separable extension, any $K$-embedding of $L$ into $K'^{\mathrm{sep}}$ is going to be contained in $K^{\mathrm{sep}}$---the subfield of $K'^{\mathrm{sep}}$ of separable elements over $K$.

Finally, we explain why $\eta$ is surjective. According to the abstract nonsense \cite[\S 5.2.1]{MurreLecturesFundamentalGroups}, $\eta$ is surjective if and only if the pullback of connected objects is connected. Hence, the surjectivity of $\eta$ is a simple consequence of the equality $S\otimes_R R'= S_{\q}^{\mathrm{sh}}$ provided by \autoref{lem.StrictHenselizationFiniteMorphism}---once we know there is only one prime lying over with trivial inertia degree.
\end{proof}

As a direct application of \autoref{theo.serretameiskummer}; \cf \cite[I, \S 5, Remark 5.1 (e)]{MilneEtaleCohomology}, we have that:
\[
\pi_1^{\textnormal{t},P'}(X'^{\circ}) = \varprojlim_{p \, \nmid \, n }{\bm{\mu}_n(K')} \xleftarrow{\cong} \varprojlim_{p \, \nmid \, n }{\Z/n\Z} \eqqcolon \hat{\Z}^{(p)}, 
\]
where it is worth noting that the isomorphism is not canonical as it depends on choices of compatible primitive roots of unity of $K'$ in $K'^{\mathrm{sep}}$. We have constructed a (non-canonical) surjective homomorphism of topological groups
\[
\eta'\: \hat{\bZ}^{(p)} \xrightarrow{\cong} \pi_1^{\textnormal{t},P'}(X'^{\circ}) \twoheadrightarrow \pi_1^{\textnormal{t},P}(X^{\circ}).
\]
which explains \autoref{eqn.CokernelSequence}. To describe $\ker \eta'$ (resp. $\Image \eta'$), consider the following assertion.
\begin{claim} \label{cla.Factorization}
There is a factorization of continuous homomorphisms
\[ 
\xymatrix{
\hat{\bZ}^{(p)} \ar@{->>}[d]_-{\mathrm{can}} \ar@{->>}[r]^{\eta'} & \pi_1^{\textnormal{t},P}(X^{\circ}) \\
\displaystyle \varprojlim_{n \in N^P(X^{\circ})} \bZ/n\bZ \ar@{->>}[ru]_-{\eta''} &
}
\]
where $N^P(X^{\circ}) = \big\{n \in \bN \bigm| p \nmid n, \, P = n \cdot D \text{ with } D|_U \in \Pic U\big\}$.
\end{claim}
\begin{proof}[Proof of claim] We may assume that $P \in \Cl X$ is nontrivial and so nontorsion. Since $\eta'$ is surjective, a Galois object $(R,\fram,\kay, K, P) \subset (S,\fran,\kay,L,Q)$ in $\mathsf{Rev}^{P}(X^{\circ})$ is cyclic, i.e. $\Gal(L/K)  \cong \Z/n\Z$. Then, \autoref{cla.Factorization} amounts to $n \in N^P(X^{\circ})$ whenever such cover exists and so the claim follows from \autoref{lem.THELEMMAFORABH}.
\end{proof}

\begin{claim} \label{clai.InjectivityEta''}
The homomorphism $\eta''$ in \autoref{cla.Factorization} is injective.
\end{claim}
\begin{proof}[Proof of claim]
By \cite[\S 5.2.4]{MurreLecturesFundamentalGroups}, $\eta''$ is injective if and only if: for all $n \in N^P(X^{\circ})$ there exists a cover in $\mathsf{Rev}^{P}(X^{\circ})$ whose pullback to $X'$ (has a connected component that) is a Kummer cover $\sO_{X,P}^{\textnormal{sh}} \subset \sO_{X,P}^{\textnormal{sh}}\big[t^{1/n}\big]$. To this end, for $n \in N^P(X^{\circ})$, let us set $\Div_X \kappa + n \cdot D = P$ so that $D|_U \in \Pic U$. We then invoke \autoref{ex.KummerTypeCyclicCovers}, \autoref{pro.KummerCoversAreTame}. That is, we embed $K(\kappa^{1/n})$ in $K^{\mathrm{sep}}$ and note that $K(\kappa^{1/n})/K \in \mathsf{Rev}^{P}(X^{\circ})$ has the required property. 
\end{proof}

With the above in place, we explain next why the short exact sequence of topological groups
\[    
0 \to \varprojlim_{n \in N^P(X^{\circ})} \bZ/n\bZ  \to \pi_1^{\textnormal{t},P}(X^{\circ}) \to G \to 1
\]
splits if and only if the containment $N^P(X^{\circ}) \supset  M^P(X^{\circ})$ is an equality and there are divisors $\frac{1}{n} P$ on $X$ for all $n \in M^P(X^{\circ})$ such that: $\frac{1}{1}P=P$ and $m(\frac{1}{mn}P)=\frac{1}{n}P$ in $\Cl X$ for all $m,n \in \bN$ so that $mn \in M^P(X^{\circ})$. Note that $M^P(X^{\circ})$ is an (inversely) directed subset $N^P(X^{\circ})$ (with respect to divisibility); see \autoref{rem.DirectabilityofNP}.

Suppose that $N^P(X^{\circ}) =  M^P(X^{\circ})$ and the existence of a compatible system $\{\frac{1}{n}P\}_{n \in M^P(X^{\circ})}$ of quotients of $P$. Recall that $\tilde{f} \: \tilde{X} \to  X$ denotes the universal cover of $\mathsf{Rev}_{\mathrm{\acute{e}t}}^{P}(X^\circ)$. The equality $N^P(X^{\circ}) =  M^P(X^{\circ})$ means that for every $n \in N^P(X^{\circ})$ there exists $\kappa_n \in K^{\times}$ such that $\Div_{\tilde X} \kappa_n + n \cdot \tilde{f}^*D_n =\tilde{P}$ where $D_n$ is a divisor on $X$ such that $D_n|_U$ is Cartier. Although the choice of $\kappa_n$ is not unique, the cyclic field extension $\tilde{K} \subset \tilde{K}(\kappa_n^{1/n}) \subset K^{\mathrm{sep}}$ is.
That is, $\tilde{K}_n \coloneqq \tilde{K}(\kappa_n^{1/n})$ is independent of the choice of $\kappa_n$ (which is not necessarily true for $K(\kappa_n^{1/n})$). In fact, these are precisely the Galois objects of $\mathsf{Rev}^{\tilde{P}}(\tilde{X}^{\circ})$ with field of fractions inside $K^{\mathrm{sep}}$. Consider the field $\tilde{K}_{\infty} \coloneqq \bigcup_{n\in N^P(X^{\circ})} \tilde{K}_n \subset K^{\mathrm{sep}}$ (so $\pi_1^{\textnormal{t},P}(X^{\circ})= \Gal(\tilde{K}_{\infty}/K)$). It is worth observing that $K(\kappa_n^{1/n}) \cdot \tilde{K} = \tilde{K}_{n}$ and $K(\kappa_n^{1/n}) \cap \tilde{K} = K$ (as the normalization of $R$ in $K(\kappa_n^{1/n})$ is totally ramified whereas in $\tilde{K}$ it is quasi-\'etale). By Galois theory \cite[VI, \S 1, Theorem 1.12]{LangAlgebra}, this implies that $\tilde{K}_{n}/K(\kappa_n^{1/n})$ is Galois and the homomorphism
\[
\Gal\big(\tilde{K}_{n}/K(\kappa_n^{1/n})\big) \to \Gal(\tilde{K}/K) = G, \quad \sigma \mapsto \sigma|_{\tilde{K}}
\]
is an isomorphism. In other words, there is an action of $G$ on $\tilde{K}_n$ by $K$-automorphisms so that $\tilde{K}_n^G = K(\kappa_n^{1/n})$. By the same token, since $K(\kappa_n^{1/n}/K)$ is Galois, the exact sequence 
\[
1 \to \Gal(\tilde{K}_n/\tilde{K}) \to \Gal(\tilde{K}_n/K) \to G \to 1
\] then
splits (as a direct product) for every $n$ (\cf \cite[VI, \S 1, Theorem 1.14]{LangAlgebra}). Roughly speaking, with the equality $N^P(X^{\circ}) =  M^P(X^{\circ})$, we can split the quotient $\Gal(\tilde{K}_{\infty}/K) \twoheadrightarrow G$ at each finite level quotient $\Gal(\tilde{K}_n/K) \twoheadrightarrow  G$. To do this globally, we need these splittings to be compatible under divisibility in $M^P(X^{\circ})$. This is precisely what the additional hypothesis regarding the existence of $\{\frac{1}{n}P\}_{n \in M^P(X^{\circ})}$ accomplishes. Indeed, by setting $D_n = \frac{1}{n}P$ and $K_n \coloneqq K(\kappa_n^{1/n})$, the field extensions $\{K_n/K\}_{n \in N^P(X^{\circ})}$ yield a projective system in $\mathsf{Rev}^{P}(X^{\circ})$ and we may then set the field $K_{\infty} \coloneqq \bigcup_{n \in N^P(X^{\circ})}K_n \subset K^{\mathrm{sep}}$. Then, $K_{\infty}/K$ is Galois, $K_{\infty} \cdot \tilde{K} = \tilde{K}_{\infty}$, and $K_{\infty} \cap \tilde{K} = K$ (as $K(\kappa_n^{1/n}) \cap \tilde{K} = K$). As before, this implies that $\tilde{K}_{\infty}/K_{\infty}$ is Galois and the homomorphism
\[
\Gal(\tilde{K}_{\infty}/K_{\infty}) \to \Gal(\tilde{K}/K) = G, \quad \sigma \mapsto \sigma|_{\tilde{K}}
\]
is an isomorphism and the exact sequence
\[
1 \to \Gal(\tilde{K}_{\infty}/\tilde{K}) \to \Gal(\tilde{K}_{\infty}/K) \to G \to 1
\]
splits; as required. This shows the direction ``$\Longleftarrow$''. 

Conversely, suppose that there is a projection $\pi_1^{\mathrm{t},P}(X^{\circ}) \twoheadrightarrow \varprojlim_{n \in N^P(X^{\circ})} \bZ/n\bZ$ splitting the inclusion $\varprojlim_{n \in N^P(X^{\circ})} \bZ/n\bZ \hookrightarrow \pi_1^{\mathrm{t},P}(X^{\circ})$. In particular, for each $n \in N^P(X^{\circ})$, there is a cyclic Galois cover $f_n \: Y_n^{\circ} \to X_n^{\circ}$ whose ``Galois-theoretic'' pullback to $\mathsf{Rev}^{\Tilde{P}}(\Tilde{X}^{\circ})$ is a cyclic Galois cover. Thus, $f_n$ must be totally ramified (as can been seen by using the inertial decantation property). Therefore, each $f_n$ is a cyclic cover of Kummer-type by \autoref{lem.THELEMMAFORABH}. In particular, $N^P(X^{\circ}) \subset M^P(X^{\circ})$. It remains to explain why we have a compatible system of elements $\{\frac{1}{n} P\} $ in $\Cl X$. To this end, assume that $f_n$ is given by $\Div \kappa_n + n D_n = P$ (this for each $n$). We have an inclusion $K(\kappa_n^{1/n}) \subset K(\kappa_{nm}^{1/nm})$. Since the latter extension is cyclic, $K(\kappa^{1/n}_{nm}) = K(\kappa_n^{1/n})$. Hence, by Kummer theory \cite[see p. 126]{MilneEtaleCohomology}, $\kappa_n/\kappa_{nm} \in (K^{\times})^n$; say $\kappa_n/\kappa_{nm}=\varkappa^n$. From this we deduce that $n \Div \varkappa +n (D_n-m D_{nm}) = 0$ and so $\Div \varkappa + (D_n-m D_{nm}) = 0$, as equalities in $\DIV X$. Hence, $D_n = m D_{nm}$ in $\Cl X$. In particular, we may take $\frac{1}{n} P \coloneqq D_n$ for all $n$. 

Finally, assuming $P \in \Cl X$ is prime-to-$p$ torsion, we still need to show that $M^P(X^\circ) = \mathbb{Z}_{>0}^{(p)}$ if and only if $P =0$ in $\Cl X$. The ``if'' direction is clear. Conversely, assume to the contrary that $P$ is non-trivial but $M^P(X^\circ) = \mathbb{Z}_{>0}^{(p)}$. Let $\ell$ be the order/index of $P$ in $\Cl X$. Since $\ell^a \in M^P(X^\circ)$ for any positive integer $a$, we find a divisor $D_a$ which is Cartier on $U$ such that $P =\ell^a D_a \in \Cl X$. Let $o_a$ be the order of $D_a$. In particular, $o_a \mid \ell^{a+1}$ but $o_a \nmid \ell^a$. Let $\ell =\ell_1^{s_1} \cdots \ell_r^{s_r}$ be the prime factorization of $\ell$ ($\ell_i \neq p$). Then, there is some index $i$ (depending on $a$) such that $\ell_i^{s_i(a+1)} \mid o_a$. Since $a$ is arbitrary, $\ell_i \geq 2$, and $s_i \geq 1$, we conclude that $o_a$ is arbitrarily large. On the other hand, we may consider the quasi-\'etale Veronese-type cyclic cover defined via $\Div_{X} \kappa + o_a \cdot D_a  =0$. These then yield objects of $\mathsf{Rev}_{\mathrm{\acute{e}t}}^{P}(X^\circ)$ of arbitrarily large degree, which contradicts the already proven finiteness of $\pi^{P}_{\mathrm{1,\acute{e}t}}(X^{\circ})$.

This demonstrates \autoref{thm.MainFormal}.
\end{proof}

\begin{remark}
The homomorphism $\eta$ in \autoref{cla.TheFundamentalClaim} can be defined more succinctly as follows. Recall that $\pi_1^{\textnormal{t},P}(X^{\circ})$ is the limit $\varprojlim \Gal(L/K)$ traversing all the finite Galois extensions $K \subset L \subset K^{\mathrm{sep}}$ so that the integral closure $S/R$ of $R$ in $L$ is tamely ramified with respect to $P$, and \emph{verbatim} for $\pi_1^{\textnormal{t},P'}(X'^{\circ})$, where we have fixed $K^{\mathrm{sep}} \subset K'^{\mathrm{sep}}$. For any such a $L/K$, we must define compatible homomorphisms $\pi_1^{\textnormal{t},P'}(X'^{\circ}) \to \Gal(L/K)$. Since $L/K$ is Galois, $\Gal(L/K)$ acts transitively and faithfully on $\sF(S) = \Hom_{K\textnormal{-alg}}(L,K^{\mathrm{sep}})$. Nonetheless, as noticed in \autoref{cla.TheFundamentalClaim}, this set is the same set as
\begin{align*}
\sF'(S\otimes_R R')=\Hom_{R'\textnormal{-alg}}\big(S\otimes_R R', K'^{\mathrm{sep}}\big) &= \coprod_i  \Hom_{R'\textnormal{-alg}}\big(S_i,K'^{\mathrm{sep}}\big) \\
&= \coprod_i  \Hom_{K'\textnormal{-alg}}\big(K(S_i),K'^{\mathrm{sep}}\big)
\end{align*}
where $S \otimes_R R' = \prod_i S_i$ is the decomposition of $S \otimes_R R'$ as a finite product of normal, local, and finite $R'$-algebras; see the proof of \autoref{lem.StrictHenselizationFiniteMorphism}. Of course, the given inclusion $K\subset L\subset K^{\mathrm{sep}}$ is an element of this set, say $\xi$. Therefore, $\xi$ is contained in one and only one of the displayed disjoint sets; let $i_0$ denote the corresponding index. Letting $L_{i_0}$ be the Galois closure of $K(S_{i_0})$ in $K'^{\mathrm{sep}}$, we have that $\Gal(L_{i_0}/K')$ surjects onto $\Aut_{K'}\big(K(S_{i_0})\big)$. On the other hand, we define the homomorphism of groups $\varphi \: \Aut_{K'}\big(K(S_{i_0})\big) \to \Gal(L/K)$ by declaring $\varphi(h)$ to be the only element of $\Gal(L/K)$ that when acts on $\xi$ yields $\xi \circ h$. In this way, we have
\[
\pi_1^{\textnormal{t},P'}(X'^{\circ}) \xrightarrow{\mathrm{can}} \Gal(L_{i_0}/K') \twoheadrightarrow \Aut_{K'}\big(K(S_{i_0})\big) \xrightarrow{\varphi} \Gal(L/K).
\]
The limit over these defines $\eta$. Observe that $\eta$ is surjective if and only if these homomorphisms are all surjective, which is equivalent to the surjectivity of $\varphi$ for all $L/K$. However, it is not difficult to see that $\varphi$ is surjective if and only if $S \otimes_R R'$ is connected.

We illustrate with an example the failure of $\eta$ being surjective if there were more than one prime lying over. In \autoref{ex.UsingLemmaConnectedComponents}, we had a canonical isomorphism of $R'$-algebras
\[
S \otimes_R R' = S \otimes_R R_{(\Delta)}^{\mathrm{sh}} \xrightarrow{\cong} S_{\q_1}^{\mathrm{sh}} \times S_{\q_2}^{\mathrm{sh}} \times S_{\q_3}^{\mathrm{sh}}.
\]
Note that a $K$-embedding of $L$ into $K'^{\mathrm{sep}}$ is the same as a choice of a square root of $\Delta$; which in our case it was $\delta$, and the choice of a $t_i$. For instance, when we chose $t_1$ to be our ``$t$'' in \autoref{ex.TheCusp}, we were choosing the $R'$-embedding
\[
S \otimes_R R' \to S_{\q_1}^{\mathrm{sh}} \times S_{\q_2}^{\mathrm{sh}} \times S_{\q_3}^{\mathrm{sh}} \to S_{\q_1}^{\mathrm{sh}} \to K'^{\mathrm{sep}},
\]
for this is the one in which $L$ is realized as the field of fractions of $R[\delta,t_1]_{\varpi_1} \to S_{\varpi_1}$. This specific embedding was our $\xi$ all along. Now, $S_{\q_1}^{\mathrm{sh}}$ is a degree-$2$ Kummer cover over $R'$, so its Galois group is cyclic of order $2$ with generator $\tau \: \delta \mapsto -\delta$. On the other hand, under the canonical bijection $\sF(S) = \sF'(S \otimes_R R')$, we see that $\xi \circ \tau$ correspond to the $K$-embedding $L \xrightarrow{(2\, 3)} L \to K^{\mathrm{sep}}$ where $(2\, 3 ) \in \Gal(L/K) \cong S_3$ is the transposition switching $t_2$ and $t_3$ (leaving $t_1$ intact). In other words, we have the following commutative diagram of groups:
\[
\xymatrix{
\pi_1^{\mathrm{t},P}(X^{\circ}) \ar@{->>}[r] \ar[d]^-{\eta} & \Gal(L/K) \ar[r]^-{\cong} &  S_3\\
\pi_1^{\mathrm{t},P'}(X'^{\circ}) \ar@{->>}[r] & \Gal\big(K\big(S_{\q_1}^{\mathrm{sh}}\big)\big/K'\big) \ar[u]_-{\varphi} \ar[r]^-{\cong} & \Z/2\Z \ar[u]_-{1\mapsto (2 \, 3)} 
}
\]
so that $\eta$ cannot be surjective. This finishes our remarks.
\end{remark}

\section{Tame fundamental groups: Positive characteristic} \label{sec.ApplicationTameFundGPSPositiveCharacteristic}
We proceed to our study of tame Galois categories in positive characteristic.

\subsection{Cohomologically tame Galois category of an $F$-pure singularity}
\label{subsubsection.Tamefungroupfpuresing}
We start by making a simple observation about the cohomologically tame Galois category of an $F$-pure singularity. This is an application of \cite[Theorem 4.8]{CarvajalStablerFsignaturefinitemorphisms} following \cite{CarvajalSchwedeTuckerEtaleFundFsignature}.

\begin{theorem} \label{thm.CohomologicallyTameFundgropuFpure} Working in \autoref{setupTameGroups}, suppose that $X$ is $F$-pure. There exists a cover $\tilde{X}^{\circ} \to X^{\circ}$ in $\textnormal{\sf{FEt}}^{\mathrm{t},X}(X^{\circ})$ such that, for any cover $V \to \tilde{X}^{\circ}$ in $ \textnormal{\sf{FEt}}^{\mathrm{t},\tilde{X}}\bigl(\tilde{X}^{\circ}\bigr)$, its integral closure $\Spec S \to \Spec \tilde{R}$ satisfies that its restriction $V\bigl(\sp(S)\bigr) \to V\bigl(\sp(\tilde{R})\bigr)$ is trivial.
\end{theorem}
\begin{proof}
Note that \cite[Theorem 4.8]{CarvajalStablerFsignaturefinitemorphisms} implies that for all connected cover $Y^{\circ} \to X^{\circ}$ in $\textnormal{\sf{FEt}}^{\mathrm{t},X}(X^{\circ})$ with integral closure $R \subset S$ we must have $1 \geq r(S) = \bigl[ \kappa\bigl(\sp(S)\bigr) : \kappa \bigl(\sp(R) \bigr) \bigr] \cdot r(R)$. In particular, we have that the generic degree of $V\bigl(\sp(S)\bigr) \to V\bigl(\sp(R)\bigr)$ is no more than $1/r(R)$. Here, we use that $R$ is $F$-pure to say $1/r(R)<\infty$. By formal properties of Galois categories (just as in \cite{CarvajalSchwedeTuckerEtaleFundFsignature}), there exists a universal cover with the required property after we notice that if the generic degree of $V\bigl(\sp(S)\bigr) \to V\bigl(\sp(R)\bigr)$ is trivial then the map itself is trivial for both $R/\sp(R)$ and $S/\sp(S)$ are strongly $F$-regular and so normal.
\end{proof}

\begin{remark} \label{rem.HeightSplittingPrime}
In \autoref{thm.CohomologicallyTameFundgropuFpure}, notice that if $\sp(R) \neq 0$ then $\sp(R)\supset \uptau(R)$. Hence, $\sp(R)$ corresponds to a singular point of $X$, for $\uptau(R)$ cuts out the non-strongly-$F$-regular locus of $X$. In particular, since $X$ is normal, $\height \sp(R) \geq 2$. In other words, either $V\bigl(\sp(R)\bigr)$ has codimension at least $2$ or $X$ is strongly $F$-regular. Hence, if $X$ is not strongly $F$-regular, $V\bigl(\sp(R)\bigr) \subset X$ has codimension $\geq 2$. Thus, \autoref{thm.CohomologicallyTameFundgropuFpure} is only interesting in higher dimensions if $X$ is a non-$F$-regular $F$-pure singularity. In a sense, this justifies next section.
\end{remark}

\subsection{Tame fundamental group of a purely $F$-regular local pair}
\label{sect.TameFunpurelyFregularpair}
In this section, we provide a study of the Galois category $\mathsf{Rev}^{P}(X^{\circ})$ for a purely $F$-regular pair $(X,P)$ that will lead to a verification of the hypothesis in \autoref{thm.MainFormal}. To this end, the following three fundamental observations \autoref{prop.EverythingTransposes}, \autoref{prop.CohTameness}, and \autoref{prop.OnlyOnePrime} about covers $(R,\fram, \kay, K) \subset (S, \fran, \el, L)$ in $\mathsf{Rev}^{P}(X^{\circ})$; as in \autoref{rem.ReductionLocalAlgebra}, are in order. 

\subsubsection{Three fundamental properties} \label{sec.ThreeFundamentalProperties} Consider the following setup.
\begin{setup} \label{setup.SetupForFundamentalPropositions}
Let $(R,\fram,\kay, K) \subset (S,\fran, \el, L)$ be  a local finite extension of normal local domains with corresponding morphism of schemes $f \: Y \to X$. Let $X^{\circ} \subset X$ be a big open (i.e., $X^{\circ}$ contains every codimension $1$ point of $X$). Assume $f \: f^{-1}(X^{\circ}) \to X^{\circ}$ to be tamely ramified with respect to a reduced divisor $D=P_1+\cdots +P_k$ with prime components $P_i=V(\p_i)$.\footnote{It does not matter whether we think of the divisors involved as divisors on $X^{\circ}$ or on $X$.}
\end{setup}

We invite the reader to look at \cite[\S 3]{CarvajalStablerFsignaturefinitemorphisms} for further details regarding transposability.

\begin{proposition}[Transposability] \label{prop.EverythingTransposes} Work in \autoref{setup.SetupForFundamentalPropositions}. Then,
$R$ is a $\Tr_{S/R}$-transposable Cartier $\sC_R^D$-module, where $\Tr_{S/R} \colon S \to R$ is the (generically induced) nonzero trace map. Moreover, $f^*\sC_R^D \subset \sC^{E}_S$ where $E$ is the reduced and effective divisor on $Y$ supported on the prime divisors whose generic point lies over the generic point of some of the $P_i$. 
\end{proposition}
\begin{proof}
Since $X$ and $Y$ are normal, we must prove that $f^* D - \Ram$ is effective; see \cite[\S 3]{CarvajalStablerFsignaturefinitemorphisms} and \cite[Theorem 5.7]{SchwedeTuckerTestIdealFiniteMaps}. Let $\q_{i,1},\ldots,\q_{i,n_i}$ be the prime ideals of $S$ lying over $\p_i$. Then,
\[
f^*P_i= e_{i,1} \cdot Q_{i,1} + \cdots + e_{i,n_i}\cdot Q_{i,n_i}
\]
where $Q_{i,j}$ is the divisor on $Y$ corresponding to $\q_{i,j}$, and $e_{i,j}$ is the ramification index of $f$ along $\q_{i,j}$.\footnote{That is, $e_{i,j}$ is the order of the uniformizer of $R_{\p_i}$ in the DVR $S_{\q_{i,j}}$; see \cite[\S 2.2]{SchwedeTuckerTestIdealFiniteMaps}.} Since $\p_1,\ldots, \p_k \in X$ are the only codimension-$1$ branch points, the ramification divisor $\Ram$ is supported on the primes divisor $Q_{i,j}$. Moreover, since the extension is tamely ramified (over $X^{\circ}$) with respect to $D$, we have that
\[
\Ram =\sum_{i=1}^k \sum_{j=1}^{n_i} (e_{i,j} - 1) \cdot Q_{i,j}, 
\]
using the same computation as in \cite[IV, Proposition 2.2]{Hartshorne}; see \cite[Remark 2.9]{Carvajalphdthesis}, \cf \cite[Remark 4.6]{SchwedeTuckerTestIdealFiniteMaps}. In this way, it clearly follows that
\begin{equation} \label{eqn.ConceptualRetrospective}
D^* \coloneqq f^*D-\Ram = \sum_{i=1}^k f^*P_i - \Ram =  \sum_{i=1}^k \sum_{j=1}^{n_i} Q_{i,j} \eqqcolon E \geq 0,
\end{equation}
as required. The last statement follows by recalling $f^* \sC_R^D \subset \sC_S^{D^*}$ and $D^*=E$.
\end{proof}

\begin{remark}
The importance of \autoref{prop.EverythingTransposes} is that we may apply \cite[Theorem 5.12 and Remark 5.13]{CarvajalStablerFsignaturefinitemorphisms} to the pair $(R,D)$ along the map $f$. In particular, we have the equality
\[
\Tr_{S/R}\big(f_* \uptau_{S(-E)} \big(S,f^*\sC_R^D\big)\big) = \uptau_{R(-D)}(R,D)
\]
Recall that $R(-D)=\bigcap_i R(-P_i) = \bigcap_i \p_i$ and similarly for $S(-E)$. Using this together with \autoref{rem.PureDRegularPairAdjointIdeal}, we may obtain direct proofs of \autoref{prop.CohTameness} and \autoref{prop.OnlyOnePrime} below. Indeed, if $(R,D)$ is purely $F$-regular along $R(-D)$, then $\uptau_{R(-D)}(R,D) = R$  and so $\Tr_{S/R}$ is surjective. Since $\Tr_{S/R}(\fran) \subset \fram$; see \cite[Lemma 2.10]{CarvajalSchwedeTuckerEtaleFundFsignature}, \cite[Lemma 9]{SpeyerFrobeniusSplit}, we have that in that case $S=\uptau_{S(-E)}(S,f^*\sC_R^P) \subset \uptau_{S(-E)}(S,E)$. In other words, the pair $(S,E)$ is purely $F$-regular along $E$. In particular, the reduced scheme supporting $E$ must have strongly $F$-regular singularities and so must be normal. Therefore, the irreducible components $Q_1, \ldots , Q_k$ cannot intersect pairwise. Nevertheless, $S$ being local, these components intersect at the closed point. Consequently, $E$ must have exactly one irreducible component $E=Q$. Nonetheless, we will provide below proofs for these statement using splitting primes. These proofs are more elementary than \cite[Theorem 5.12]{CarvajalStablerFsignaturefinitemorphisms} and the authors believe this approach is valuable in its own right. 
\end{remark}

\begin{proposition}[{Cohomological tameness, \cf \cite{KerzSchmidtOnDifferentNotionsOfTameness}}] \label{prop.CohTameness}
 Work in \autoref{setup.SetupForFundamentalPropositions}. Suppose that $k=1$ and $(X,D=P)$ is purely $F$-regular. Then, the extension $R \subset S$ is cohomologically tame, i.e. $\Tr_{S/R}\: S \to R$ is surjective. Thus, $ \el/\kay$ is separable and $[\el:\kay]$ divides $[L:K]$. Furthermore, if $\el/\kay$ is trivial, then $p$ does not divide $[L:K]$.
\end{proposition}
\begin{proof}
 Let $E = Q_1 + \cdots +Q_n$ as in \autoref{prop.EverythingTransposes} and recall that $E=f^*D-\Ram \geq 0$. Notice that $S$ is $f^*\sC_R^P$-compatible, so $\Tr_{S/R}(S)$ is a nonzero $\sC_R^P$-compatible ideal as $\varphi \circ F^e_*\Tr_{S/R} = \Tr_{S/R} \circ \varphi^{\top}$ for all $\varphi \in \sC_{e,R}^P$. If $\Tr_{S/R}(S) \subsetneq R$, then $\Tr_{S/R}(S)$ must be contained in the ideal $R(-P)$---the splitting prime of $\sC_R^P$ by hypothesis. In other words,
\[
\Tr_{S/R} \in \Hom_R\big(S, R(-P)\big) = \Hom_R\big(S \otimes_R R(P), R\big) = \Hom_R\big(S(f^*P), R\big),
\]
which implies $S(f^*P) \subset S(\Ram)$ and so $\Ram -f^*P \geq 0$. Thus, $E=0$, which is absurd. 

For the statements regarding $\kay \subset \el$, use \cite[Proposition 3.17]{Carvajalphdthesis}, \cf \cite[Lemma 2.15]{CarvajalSchwedeTuckerEtaleFundFsignature}, with $\Delta = \Delta_{\varphi}$ and where $\varphi$ is taken to be any surjective map in $\sC_{e,R}^{P} = \bigl(\sC_{e,R}^{P} \bigr)^{\top}$ 
for $e \gg 0$. Note such a map $\varphi$ exists because $\mathfrak{p} \neq R$, \ie $(R,P)$ is $F$-pure. Notice $\sC_{e,R}^{P} = \bigl(\sC_{e,R}^{P} \bigr)^{\top}$ was demonstrated in \autoref{prop.EverythingTransposes}.

For the final statement, since $\Tr_{S/R}(\fran) \subset \fram$, there is an induced $\kay$-linear map $\overline{\Tr}\: \el \to \kay$, which is nonzero if $\Tr \: S \to R$ is surjective. However, if $\el/\kay$ is trivial then $\overline{\Tr}$ corresponds to the $\kay$-linear map $\kay \to \kay$ given by multiplication-by-$[L:K]$. In other words, $[L:K]$ is not zero as an element of $\kay$ and so it is prime to $p$ as a natural number.
\end{proof}

\begin{theorem}[Only one prime lying over] \label{prop.OnlyOnePrime} 
 Work in \autoref{setup.SetupForFundamentalPropositions}. Suppose that $k=1$ and $(X,D=P)$ is purely $F$-regular. The splitting prime $\q\coloneqq \sp\big(S,f^* \sC_R^{P}\big)$ is the one and only one prime of $S$ lying over $\p \coloneqq R(-P)$. Moreover, $(Y,Q)$ is purely $F$-regular where $Q = V(\p)$. 
\end{theorem}
\begin{proof}
First, $\q$ is well-defined by \cite[Theorem 4.8]{CarvajalStablerFsignaturefinitemorphisms}, \autoref{prop.EverythingTransposes}, and \autoref{prop.CohTameness}.\footnote{The other two conditions are always satisfied. The $S$-linear map $S \to \omega_{S/R}$; $1 \mapsto \Tr_{S/R}$, is generically an isomorphism as $L/K$ is separable. The condition $\Tr_{S/R}(\fran) \subset \fram$ holds as in \cite[Lemma 2.10]{CarvajalSchwedeTuckerEtaleFundFsignature}.} Next, we prove $\q$ is unique in lying over $\p$. We may pass to a cover of $S$ in proving this and may therefore assume that $f$ is generically Galois by \cite[Lemma 2.2.6]{GrothendieckMurreTameFundamentalGroup}. Let $\q'$ be a prime of $S$ lying over $\p$, \ie $\q' \cap R = \p$. 
It suffices to prove $\q' \subset \q$; see \cite[Corollary 5.9]{AtiyahMacdonald}. For this, we use that $\q$ is a splitting prime ideal and the corresponding definition; see \cite[\S 2.3.2]{CarvajalStablerFsignaturefinitemorphisms} for the definition. It suffices to show $\varphi^{\top}(F^e_* \q') \subset \fran$ for all $\varphi\in \sC_{e,R}^{D}$ and $e>0$ as the right $S$-span of $\big\{\varphi^{\top} \bigm| \varphi \in \sC_{e,R}^{D} \big\}$ is $f^* \sC_{e,R}^{D}$; see \cite[Remark 2.15 (c)]{CarvajalStablerFsignaturefinitemorphisms}.

\begin{claim}
$\Tr_{S/R}(\q') \subset \p$.
\end{claim}
\begin{proof}[Proof of claim]
This has been shown for $\q$ in the proof of \cite[Theorem 4.8]{CarvajalStablerFsignaturefinitemorphisms}; see \cite[Equation 4.8.2]{CarvajalStablerFsignaturefinitemorphisms}. We use the symmetry imposed by the Galois condition to induce this property to the other (possible) primes lying over $\p$. Concretely, we have that $\textnormal{Gal}(L/K)$ acts transitively on the set of primes lying over $\p$ \cite[\href{http://stacks.math.columbia.edu/tag/09EA}{Lemma 09EA}, or \href{https://stacks.math.columbia.edu/tag/0BRI}{Lemma 0BRI}]{stacks-project}---although faithfulness might be lost due to ramification. Hence, if a prime is mapped into $\p$ by $\Tr_{S/R}$ then so are its Galois conjugates, for $\Tr_{S/R}(x) = \sum_{\sigma \in \textnormal{Gal}(L/K)}{\sigma(x)}$ for all $x \in S$.
\end{proof}
For all $\varphi \in \sC_{e,R}^D$, it follows that
\[
\Tr_{S/R}\big(\varphi^{\top}(F^e_* \q')\big) = \varphi\big( F^e_*\Tr_{S/R} ( \q') \big) \subset \varphi (F^e_* \p ) \subset \p,
\]
where the last containment follows from $\p$ being the splitting prime of $\sC_R^D$. In other words, $\varphi^{\top}(F^e_* \q') \subset \Tr_{S/R}^{-1}(\p) \subsetneq S$ (as $\Tr_{S/R}$ is surjective by \autoref{prop.CohTameness}). Since $\varphi^{\top}(F^e_* \q')$ is an $S$-module, it must be contained in $\mathfrak{n}$, which was to be shown. Finally, the pair $(Y,Q)$ is purely $F$-regular by \cite[Theorem 6.12, Remark 5.15]{CarvajalStablerFsignaturefinitemorphisms} and \autoref{rem.PureDRegularPairAdjointIdeal}.
\end{proof}

\begin{example}
Let $R$ and $f$ be as in \autoref{ex.Smoothambient}. A direct consequence of \autoref{prop.OnlyOnePrime} establishes that if $L$ is a finite separable extension over $K$---the fraction field of $R$---then there is one and only one DVR of $L$ lying over $R_{(f)}$ if: $R \subset R^{L}$ is tamely ramified with respect to $\Div f$ and $R/f$ is strongly $F$-regular. This does not hold without assuming $R/f$ is strongly $F$-regular (i.e. $(R,\Div f)$ is purely $F$-regular). Indeed, consider the cusp \autoref{ex.TheCusp}. In this case, the singularities of $R/f$ are not even $F$-pure. One may still wonder if $F$-purity of $R/f$ may suffice. To see this is not the case, we may get back to the Whitney's umbrella \autoref{ex.WhitneyUmbrella}. Indeed; with notation as in \autoref{ex.WhitneyUmbrella}, we specialize to $R=\kay \llbracket x,y,z \rrbracket$ with $\kay$ an algebraically closed field of odd characteristic. In \cite[\S4.3.3]{BlickleSchwedeTuckerFSigPairs1}, it is shown that $R/f$ is $F$-pure yet not strongly $F$-regular. In fact, it is proven that $\sp(R,\Div f) = (x,y) \supsetneq (f)$.\footnote{In particular, $(R/f,(x,y)/f)$ is a purely $F$-regular pair where $R/f$ is not normal---this is the counterexample \autoref{rem.SFRAmbient} is referring to.}  Considering $R' \coloneqq R_{(x,y)}$, $(R',f)$ is a counterexample where $R'/f$ is a (non-normal) $F$-pure ring. The authors are unaware of a counterexample $(R,f)$ where $R/f$ is normal and $F$-pure.
\end{example}

\begin{example} \label{ex.MultipleComponents}
There are interesting instances of multiple components pairs $(X,P_1 + \cdots + P_k)$ where $(X,P_i)$ are all purely $F$-regular. For example, we may consider $X = \Spec R$ with $R$ as in either \autoref{ex.P1crossP1} or \autoref{ex.P2crossP2}. With $R$ as in \autoref{ex.P1crossP1}, we may let $R(-Q)=\q= (x,w)$. By symmetry on the variables, $(X,Q)$ is a purely $F$-regular pair as well and moreover $\Div x = P+Q$. We may also consider $\p'=(y,z)$, $\q'=(y,w)$, $P'=V(\p')$, and $Q'=V(\q')$, which all define purely $F$-regular pairs on $X$ as well. In fact,  $\Div xy =  P+Q+P'+Q' = \Div zw$. Thus, $(X,\Div x)$ or $(X,\Div xy)$ are example where the aforementioned setup holds. Similarly, we may let $X=\Spec R$ with $R$ as in \autoref{ex.P2crossP2}. Then, if we consider $R(-Q)=\q\coloneqq (x,y,z)$, by symmetry, $(X,Q)$ is a purely $F$-regular pair as well and  $P+Q= \Div ux = \Div vy = \Div wz$.\footnote{Indeed, one verifies that the ideal of $R$ generated by $u$ is the quotient of the ideal $(u, \Delta_1, \Delta_2, \Delta_1) = (u,vx,wx,\Delta_1)=(u,v,w) \cap (u,x,\Delta_1)$ of $A$, where the latter is a prime decomposition.} Thus, $(X,\Div ux)$ is another example. In any of these examples $(X,\Div f)$, we wonder what the structure of $\mathsf{Rev}^{\Div f}(X)$ is.
\end{example}

\subsubsection{Main Theorem}

With \autoref{sec.ThreeFundamentalProperties} in place, we are ready to establish our main result. First, we make the following observation.

\begin{remark} \label{rem.NatureCategory}
An interesting, conceptual consequence of \autoref{prop.EverythingTransposes}, \autoref{prop.CohTameness}, and \autoref{prop.OnlyOnePrime} is that we may think of the objects in the Galois category $\mathsf{Rev}^{P}(X^{\circ})$ as quintuples $(S, \fran, \el, L, Q)$ where $Q$ is a prime divisor minimal center of $F$-purity, which corresponds to the only height-$1$ prime divisor lying over $\p$; namely, the splitting prime of both $f^* \sC_R^{P} \subset \sC_S^{Q}$, say $\mathfrak{q}$. In particular, $(S,Q)$ is a purely $F$-regular pair. Applying \cite[Theorem 4.8]{CarvajalStablerFsignaturefinitemorphisms}  (note that its assumptions are verified by \autoref{prop.CohTameness}), we have
\[
1 \geq r(S,Q)=[\kappa(\q):\kappa(\p)] \cdot r(R,P)>0.
\]
Hence, $ [\kappa(\q):\kappa(\p)] \leq 1/r(R,P)$. In retrospective, we also see that $Q$ happens to be the divisor $P^* = f^* P - \Ram$ in \autoref{eqn.ConceptualRetrospective}.
\end{remark}

\begin{theorem} \label{thm.TameApplication}
Work in \autoref{setupTameGroups} and suppose that $(X,P)$ is a purely $F$-regular pair. Then, $\mathsf{Rev}^{P}(X^{\circ})$ is $P$-irreducible and has both inertial boundedness and tameness. In particular, there exists an exact sequence of topological groups
\begin{equation*}
 \hat{\Z}^{(p)} \to \pi_1^{\textnormal{t},P}(X^{\circ}) \to \pi^{P}_{\mathrm{1,\acute{e}t}}(X^{\circ}) \to 1,
\end{equation*}
where $\pi^{P}_{\mathrm{1,\acute{e}t}}(X^{\circ})$ is a finite group of order at most $\mathrm{min}\bigl\{1\big/r(R,P),1/s(R)\bigr\}$ and prime-to-$p$. Furthermore, if $P$ is a prime-to-$p$ torsion element of $\Cl X$, the homomorphism $\hat{\Z}^{(p)} \to \pi_1^{\textnormal{t},P}(X^{\circ})$ is injective and so we have a short exact sequence
\begin{equation*} 
 0 \to \hat{\Z}^{(p)} \to \pi_1^{\textnormal{t},P}(X^{\circ}) \to \pi^{P}_{\mathrm{1,\acute{e}t}}(X^{\circ}) \to 1,
\end{equation*}
which splits if and only if $P$ is the trivial element of $\Cl X$. If $P$ is nontorsion, we have a short exact sequence
\[
0 \to \varprojlim_{n \in N^P(X^{\circ})} \bZ/n\bZ \to \pi_1^{\textnormal{t},P}(X^{\circ}) \to \pi^{P}_{\mathrm{1,\acute{e}t}}(X^{\circ}) \to 1
\]
which splits if and only if $N^P(X^{\circ})= M^P(X^{\circ})$ and there is a compatible system $\{\frac{1}{n} P \in \Cl X\}_{ n \in M^P(X^{\circ})}$ of factors of $P$.
\end{theorem}
\begin{proof}
This is an application of \autoref{thm.MainFormal} and \autoref{sec.ThreeFundamentalProperties}, \cf \autoref{rem.NatureCategory}. Indeed, $P$-irreducibility holds by \autoref{prop.OnlyOnePrime}. Inertial boundedness was explained in \autoref{rem.NatureCategory} whereas inertial tameness follows from \autoref{prop.CohTameness}. For the statements regarding the order of $G$, recall that it is realized as the Galois group of a universal \'etale-over-$P$ cover $\tilde{X}^{\circ} \to X^{\circ}$. In particular, its generic degree equals $[\kappa(\tilde{\p}): \kappa(\p)]$, which is bounded by both $1\big/r(R,P)$ and $1/s(R)$ (for the latter bound simply use \cite[Theorem 3.11]{CarvajalSchwedeTuckerEtaleFundFsignature}).
\end{proof}

\begin{remark}
It is an important folklore conjecture that the divisor class group of a strongly $F$-regular singularity is finitely generated. We were taught about this question by Karl~Schwede but believe that it was originally raised by Melvin~Hochster. For more, see \cite{PolstraATheoremAboutMCMM}. It is known that the torsion subgroup is finite in this case; see \cite{PolstraATheoremAboutMCMM,MArtinNumberOfTorsionDivisorsinSFRrings}. Finite generation of the class group is known to be true in dimension $\leq 3$; see \cite{CRStaeblerKollarFundamentalGroups}. Whenever finite generation of the class group is known, we may improve upon \autoref{thm.TameApplication} to say that $\pi_1^{\textnormal{t},P}(X^{\circ}) \in \Ext(G,\bZ/n\bZ)$ for some $n \gg 0$ (in particular finite) and it is a trivial extension if $P$ is $n$-divisible in $\Cl X$ with $\frac{1}{n} \cdot P$ Cartier on $U$.
\end{remark}

\begin{remark}
In \autoref{thm.TameApplication}, if $X^{\circ} = X_{\mathrm{reg}}$, we may replace $\pi^{P}_{\mathrm{1,\acute{e}t}}(X^{\circ})$ with $G=\pi_1^{\mathrm{\acute{e}t}}(X_{\mathrm{reg}})$ by \autoref{pro.RoleOfRegularLocus}. By \cite[Corollary 1.2]{TaylorInversionOfAdjuntionFSignature}, $\mathrm{min}\bigl\{1\big/r\big(R,P\big),1/s(R)\bigr\} = 1/s(R)$ if $P$ is prime-to-$p$ torsion in $\Cl X$.
\end{remark}

\begin{corollary} \label{cor.Purity}
Let $f:Y \to X$ be a quasi-\'etale cover. If there is a divisor $\Delta$ on $X$ such that $(X,\Delta)$ is purely $F$-regular and $r(\sO_{X,x},\Delta)>1/2$ for all $x\in X$, then $f$ is \'etale.
\end{corollary}
\begin{proof}
The proof is \emph{mutatis mutandis} the same as in \cite[Corollary 3.3]{CarvajalSchwedeTuckerEtaleFundFsignature}. 
\end{proof}

\begin{remark}
In light of \cite[Corollary 1.2]{TaylorInversionOfAdjuntionFSignature}, it is unclear to the authors whether there are cases where \autoref{cor.Purity} improves upon \cite[Corollary 3.3]{CarvajalSchwedeTuckerEtaleFundFsignature}. One potential candidate for such examples would be determinantal singularities. In \cite[Example 4.12]{Carvajalphdthesis}, the first named author proved; based on \cite{CutkoskyPurity}, that determinantal singularities satisfy purity of the branch locus. With notation as in \autoref{que.DeterminantalRings}, it is known that the $F$-signature of $C_{1,2}$ is $11/24=1/2-1/24$; see \cite{SinghFSignatureOfAffineSemigroup}. On the the other hand, we have estimated that  $r(C_{1,2}, P)\geq 1/6$ in \autoref{ex.P2crossP2}. Nonetheless, our methods were not sufficient to prove (nor disprove) that $r(C_{1,2}, P) > 1/2$.
\end{remark}
\begin{corollary} \label{cor.AbhyankarLemmaPureFRegPairs}
In the setup of \autoref{thm.TameApplication}, the conclusion of \autoref{lem.THELEMMAFORABH} holds.
\end{corollary}

\begin{example}[Determinantal singularities]
Let $R$ be a determinantal singularity with $P$ a prime divisor generating $\Cl R$; see \autoref{que.DeterminantalRings}. Then, $\pi_1^{\textnormal{\'et}}(X^{\circ})$ is trivial for all $Z$ by \cite[Example 4.12]{Carvajalphdthesis}. Therefore, if $(R,P)$ is a purely $F$-regular pair; see \autoref{que.DeterminantalRings}, then $\pi_1^{\mathrm{t},P}(X^{\circ})$ is trivial too as $P$ is not divisible in $\Cl X$.
\end{example}

\begin{question}
Let $(X, \Div f)$ be any of the examples in \autoref{ex.MultipleComponents}. Does Abhyankar's lemma hold for $(X,\Div f)$?
\end{question}

\begin{example}[Graded hypersurfaces] \label{ex.gradedHypersurfaceGroups}
With notation as in \autoref{exa.weightedhypersurface}, suppose that $A$ is strongly $F$-regular. If $n$ is prime-to-$p$, then $\pi_1^{\mathrm{t},P}(X_{\mathrm{reg}})$ is a non-trivial element of $\Ext\big(\bZ/n\bZ, \hat{\bZ}^{(p)}\big)$. Indeed, the corresponding degree-$n$ cyclic cover is its universal \'etale-over-$P$ cover. If $n$ is a $p$-power---$R$ might be referred to as a Zariski hypersurface---its \'etale-over-$P$ universal cover is trivial; see \cite[Proposition 7.2.2]{MurreLecturesFundamentalGroups}. Therefore, all we can say is that there is a surjection $\hat{\bZ}^{(p)} \twoheadrightarrow \pi_1^{\mathrm{t},P} (X^{\circ})$. Determining the kernel of this surjection may require obtaining an analog of \cite[Proposition 7.2.2]{MurreLecturesFundamentalGroups} for the category $\mathsf{Rev}^P(X^{\circ})$. 
\end{example}

\section{Tame fundamental groups: Characteristic zero}
\label{sec.CharZeroBusiness}
The goal of this section is to prove the following by reduction to positive characteristics.
\begin{theorem}
\label{thm.TameApplicationCharZero}
Let $(X, \Delta)$ be a log canonical pair, $\dim X \geq 2$, $x \in X$ be a closed point, and $Z \subset X$ be a closed subscheme of codimension $\geq 2$. Denote by $P$ the minimal LC center through $x$ which we assume to be a divisor. Write $X^{\circ} = \Spec \mathcal{O}^{\mathrm{sh}}_{X,\bar{x}} \smallsetminus Z$ and denote by $\Delta$ and $P$ the pullback of $\Delta$ and $P$ to $X^{\circ}$.
Then, $\mathsf{Rev}^{P}(X^{\circ})$ is $P$-irreducible and has inertial boundedness. In particular, there exists an exact sequence of topological groups
\[ 
\hat{\mathbb{Z}} \to \pi_1^{\mathrm{t}, P}(X^{\circ}) \to \pi^{P}_{\mathrm{1,\acute{e}t}}(X^{\circ}) \to 1, 
\]  where $\pi^{P}_{\mathrm{1,\acute{e}t}}(X^{\circ})$ is finite. Furthermore, if $P$ is a torsion element of $\Cl X$, the homomorphism $\hat{\Z} \to \pi_1^{\textnormal{t},P}(X^{\circ})$ is injective and so we have a short exact sequence
\begin{equation*} 
 0 \to \hat{\Z} \to \pi_1^{\textnormal{t},P}(X^{\circ}) \to \pi^{P}_{\mathrm{1,\acute{e}t}}(X^{\circ}) \to 1,
\end{equation*}
which splits if and only if $P$ is the trivial element of $\Cl X$. If $P \in \Cl X$ is nontorsion, we have a short exact sequence
    \begin{equation*}
   0 \to \bZ/n\bZ \to \pi_1^{\textnormal{t},P}(X^{\circ}) \to \pi^{P}_{\mathrm{1,\acute{e}t}}(X^{\circ}) \to 1
    \end{equation*}
    which splits if and only if there is a divisor $D$ such that $P = nD$ in $\Cl X$ and $D\vert_U$ is Cartier.
\end{theorem}
\begin{proof}
In order to prove \autoref{thm.TameApplicationCharZero}, recall that we may work in \autoref{sec.LastSectionSetup2}. We want to use \autoref{thm.MainFormal}, and thus need to verify that $P$-irreducibility and intertial boundedness (\autoref{term.MAnyTerms}) hold for the PLT pair $(R=\sO_{X,\bar{x}}^{\mathrm{sh}},\Delta)$. This will be proven below in \autoref{pro.(a)HoldPLT} and \autoref{pro.(b)HoldPLT} respectively.

We still need to explain why, if $P$ is nontorsion, the set $N^P(X^\circ)$ is finite so that \autoref{eqn.NontorsionAndFiniteGeneration} holds. This however is a direct consequence of \cite[Theorem 6.1]{BingenerFlennerClassGroup}.
\end{proof}

\begin{proposition} \label{pro.(a)HoldPLT}
Work in \autoref{sec.LastSectionSetup2}. Then $\mathsf{Rev}^P(X^\circ)$ satisfies $P$-irreducibility.
\end{proposition}

\begin{proposition} \label{pro.(b)HoldPLT}
Work in \autoref{sec.LastSectionSetup2}. Then $\mathsf{Rev}^P(X^\circ)$ satisfies inertial boundedness.
\end{proposition}

We shall see that inertial boundedness follows from minor modifications of the arguments in \cite{BhattGabberOlssonFinitenessFun}; see \autoref{sec.Condb} below. Thus, we prove intertial boundedness by spreading out. While a prove of $P$-irreducibility is also possible via spreading out, there is a direct proof in characteristic zero which we give below. We are thankful to Karl Schwede for pointing this out to us.

\begin{corollary}
\label{cor.abyankarcharzero}
In the setup of \autoref{thm.TameApplicationCharZero}, the conclusion of \autoref{lem.THELEMMAFORABH} holds.
\end{corollary}

\subsection{$P$-irreducibility in characteristic zero} \label{sec.ProofProp5.2}

We need some preparatory lemmata. Recall that an \emph{\'etale neighborhood} of a geometric point $\bar{x} \to \Spec R$ is a factorization through an \'etale morphism $\Spec R' \to \Spec R$.

\begin{lemma}
\label{le.ShMorphismLiftPullback}
Let $R$ be normal domain and $R^{\mathrm{sh}}$ be its strict henselization at a closed point $x\in \Spec R$. Let $f\: \Spec S \to \Spec R^{\mathrm{sh}}$ be a finite dominant morphism. Then, there exists a connected \'etale neighborhood $\Spec R'$ of $\bar{x}$ and a cartesian square
\begin{equation}\label{eqn.Diagram1}\begin{xy} \xymatrix{ \Spec S \ar[r]^f \ar[d]^g & \Spec R^{\mathrm{sh}} \ar[d]^h \\ \Spec S' \ar[r]^{f'} & \Spec R'} \end{xy} \end{equation} with $f'$ finite.
Furthermore, if $\mathfrak{p} \subset R$ is a height-$1$ prime such that $\mathfrak{p}R^{\mathrm{sh}}$ is prime, then $h(\p R^{\mathrm{sh}})$ is a height-$1$ prime of $R$ and $h^{-1}(h(\p R^{\mathrm{sh}})) = \p R^{\mathrm{sh}}$. Finally, if $R$ is local then $R'$ is normal and $S'$ is normal if and only if $S$ is normal.
\end{lemma}
\begin{proof}
Fix generators $(a_1, \ldots, a_m)$ of $\mathfrak{p}$ and write $S = R^{\mathrm{sh}}[b_1, \ldots, b_e]$. As $R^{\mathrm{sh}} \to S$ is finite, there are monic polynomials $f_i \in R^{\mathrm{sh}}[T]$ with $f_i(b_i) = 0$. We denote the coefficients of these $f_i$ by $c_{ij}$. Since $R^{\mathrm{sh}}$ is obtained as a filtered colimit of connected \'etale neighborhoods $R \to R'$ of $\bar{x}$, there is some $R \to R'$ in the colimit system such that $R'$ contains all the $a_i$ and $c_{ij}$. By construction, $R \to R'$ is \'etale and setting $S' = R'[b_1, \ldots, b_e]$ one readily checks that the above diagram is a pullback square. In particular, $f'$ is finite by construction. Since the fibers of $R \to R'$ are of dimension zero, $h(\p R^{\mathrm{sh}}) = \p R'$ is of height $1$. Clearly, $h^{-1}(h(\p R^{\mathrm{sh}})) = \p R^{\mathrm{sh}}$.

For the final assertion, note that the weakly \'etale homomorphism $R' \to R^{\mathrm{sh}}$ is faithfully flat since $\mathfrak{m}R^{\mathrm{sh}}$ is the maximal ideal of $R^{\mathrm{sh}}$. Since \autoref{eqn.Diagram1} is a pullback square, $S' \to S$ is a weakly \'etale faithfully flat homomorphism too. Thus $S'$ is normal if and only if $S$ is and similarly for $R'$ and $R^{\mathrm{sh}}$ by \cite[\href{https://stacks.math.columbia.edu/tag/033G}{Lemma 033G}]{stacks-project} and \cite[\href{https://stacks.math.columbia.edu/tag/0950}{Tag 0950}]{stacks-project}.
\end{proof}

\begin{remark}
\label{rem.ExtensionPrimetoSHprimefornormal}
If $R \to R^{\mathrm{sh}}$ is the strict henselization with respect to some maximal ideal $\mathfrak{m}$ then, given any ideal $\mathfrak{a} \subset \mathfrak{m}$ such that $R/\mathfrak{a}$ is normal, the extension $\mathfrak{a}R^{\mathrm{sh}}$ is prime. Indeed, we may localize $R$ at $\mathfrak{m}$ and thus assume that $R$ is a local ring. Then, the assertion follows from $R^{\mathrm{sh}} \otimes R/I = (R/I)^{\mathrm{sh}}$ for any ideal $I \subset R$ (\cite[\href{https://stacks.math.columbia.edu/tag/05WS}{Lemma 05WS}]{stacks-project}) and the fact that $S^{\mathrm{sh}}$ is a normal domain if and only if $S$ is a normal domain (\cite[\href{https://stacks.math.columbia.edu/tag/033G}{Lemma 033G}]{stacks-project}).
\end{remark}

\begin{lemma}
\label{lem.TameEtaleReflectedFromSHtoStalk}
Let $g: \Spec T \to \Spec R$ be a surjective \'etale morphism or a surjective pro-\'etale morphism. Let $f: \Spec S \to \Spec R$ be a morphism. Consider the base change diagram 
\[ 
\xymatrix{\Spec S \otimes_R T \ar[r]^-{f'} \ar[d]^-{g'} &\Spec T \ar[d]^-{g} \\ \Spec S \ar[r]^-f & \Spec R} 
\]
The morphism $f$ is tame with respect to $D$, if and only if $f'$ is tame with respect to $g^{-1}(D)$.
\end{lemma}
\begin{proof}
The ``only if'' implication follows from \cite[Lemma 2.2.7]{GrothendieckMurreTameFundamentalGroup}. For the converse, by \cite[\href{https://stacks.math.columbia.edu/tag/033C}{Lemma 033C}]{stacks-project} and \cite[\href{https://stacks.math.columbia.edu/tag/033G}{Lemma 033G}]{stacks-project}, $R$ is normal if and only if $T$ is normal. Likewise, $S$ is normal if and only if $T \otimes_R S$ is normal. Thus, it makes sense to talk about tame morphisms. The remaining assertion is a consequence of \cite[Proposition 2.2.9]{GrothendieckMurreTameFundamentalGroup}.
\end{proof}

\begin{proposition}
\label{pro.CharZeroTameCoverOnePrime}
Let $(X, \Delta)$ be an affine PLT pair where $\lfloor \Delta \rfloor = P$ is a prime divisor. If $g\: Y \to X$ is a tamely ramified cover over $P$, then $\bigl(g^{-1}(P)\bigr)_{\mathrm{red}}$ is a normal divisor
\end{proposition}
\begin{proof}
Write $\Delta = P + \Delta_1$.
By \cite[Corollary 2.43, (2.41.4)]{KollarSingulaitieofMMP} the pair $(Y, \Delta')$ is PLT, where $\Delta' = \bigl(g^{-1}(P)\bigr)_{\mathrm{red}} + g^\ast \Delta_1$, and $K_{X'} + \Delta' \sim_\mathbb{Q} g^\ast(K_X + \Delta)$. Note that $\lfloor g^\ast \Delta_1 \rfloor = 0$. Indeed, since $(X, \Delta)$ is PLT, $\Delta_1$ and $P$ have no components in common. Since $g$ is \'etale over $X \setminus P$ the assertion follows. In this way, we see that $\bigl(g^{-1}(P)\bigr)_{\mathrm{red}}$ is a minimal LC center for some closed point $y \in Y$. Hence, by \cite[Theorem 7.2]{FujinoGongyoCanonicalBundleFormulasAndSubadjunction} $\bigl(g^{-1}(P)\bigr)_{\mathrm{red}}$ is normal.
\end{proof}

\begin{proof}[Proof of \autoref{pro.(a)HoldPLT}]
We use the notation of \autoref{sec.LastSectionSetup2} and write $R^{\mathrm{sh}}$ for $\mathcal{O}_{X,x}^{\mathrm{sh}}$. Let $V \to \Spec R^\mathrm{sh} \smallsetminus Z$ be a cover in $\mathsf{Rev}^P(\Spec R^\mathrm{sh} \smallsetminus Z)$.
Denote the integral closure of $R^{\mathrm{sh}}$ inside $\sO_V(V)$ by $S$. Using \autoref{le.ShMorphismLiftPullback}, we obtain a cartesian square
\[
\begin{xy}
\xymatrix{ \Spec S \ar[r]^f \ar[d]^g & \Spec R^{\mathrm{sh}} \ar[d]^h \\ \Spec S' \ar[r]^{f'} & \Spec R'} 
\end{xy}
\]
with $f'$ finite and $R'$ a connected \'etale neighborhood of $\bar{x} \to W \subseteq X$, where $W$ is some Zariski open neighborhood of $x \in X$. As usual, we write $\p$ for the prime corresponding to our fixed prime divisor $P$ on $X$. As $f$ is finite, $S$ is also strictly henselian (\cite[\href{https://stacks.math.columbia.edu/tag/04GH}{Tag 04GH}]{stacks-project}) and $S$ is the strict henselization of $S'$ with respect to some ideal $\n$ lying over $x$.

Since $R/\p$ is normal as a minimal LC center (\autoref{theo.centerspreadout}), we deduce from \autoref{rem.ExtensionPrimetoSHprimefornormal} that $\p R^{\mathrm{sh}}$ is prime. Write $\p' = h(\p R^{\mathrm{sh}})$. Using \autoref{le.ShMorphismLiftPullback} again, we have that $h^{-1}(\p') = \p R^{\mathrm{sh}}$. Note that $f'$ is tamely ramified with respect to $P'$ by \autoref{lem.TameEtaleReflectedFromSHtoStalk}. Since $\Spec R'$ is an \'etale neighborhood of $\bar{x} \to X$, say $\varphi: \Spec R' \to X$, we conclude that $(\Spec R', \varphi^\ast(\Delta))$ is PLT with $\lfloor \Delta \rfloor = P'$. Thus we can apply \autoref{pro.CharZeroTameCoverOnePrime} to conclude that $Q' = \bigl(f'^{-1}(P')\bigr)_{\mathrm{red}}$ is normal. We denote the corresponding ideal by $\q'$ and note that $\q' \subseteq \fran$. Using \autoref{rem.ExtensionPrimetoSHprimefornormal} we see that $\q \coloneqq \q'S$ is prime. In other words, there is only one prime in $S'$ lying over $\p'$ and contained in $\fran$. Assume now that $\mathfrak{a} \in f^{-1}(\p)$. Then $h(f(\mathfrak{a})) = \p'$ and hence $f'(g(\mathfrak{a})) = \p'$. In particular, $g(\mathfrak{a}) \in f'^{-1}(\p')$. But clearly, $g(\mathfrak{a}) \subseteq \fran$. Hence, $\mathfrak{a} = \q$ as desired.
\end{proof}

\subsection{Intertial boundedness via spread-out}
\label{sec.Condb}

In the situation of \autoref{sec.LastSectionSetup2}, write $Y = \Spec \sO_{X,x}^\mathrm{sh}$ and $Y_\mathrm{reg}$ for its regular locus. By \autoref{le.b'-condition}, it suffices to show that intertial boundedness holds for $Y_\mathrm{reg}$. To this end, we use the result of \cite[Theorem 1.1]{BhattGabberOlssonFinitenessFun}, where finiteness of $\pi_1(Y \smallsetminus \{x\})$ is proved via reduction mod $p$. The proof of \cite[Theorem 1.1]{BhattGabberOlssonFinitenessFun} is a combination of Theorem 4.1, Proposition 5.1 and Proposition 6.4 in ibid. We can directly apply the latter two in our situation. The argument of Theorem 4.1 needs to be modified slightly. We record this below for completeness. 

We will use the following notation for spreading out: If $R$ is a $\kay$-algebra, $A \subset \kay$ a finitely generated $\mathbb{Z}$-algebra, then we will write $R_A$ for any fixed finite type $A$ algebra whose base changed generic fiber $R_A \otimes_A \Frac(A) \otimes_{\Frac(A)} \kay$ recovers $R$. If $s \in \Spec A$ is a point, then we will write $R_s$ for the corresponding fiber of $R_A$. We will use similar notation for schemes.

\begin{theorem}
\label{theo.BGOanalog}
Let $A$ be a finitely generated $\mathbb{Z}$-algebra equipped with an embedding $A \to \mathbb{C}$. Fix an affine finite type scheme $X_A$ over $\Spec A$, a closed point $x_a \in X_A$, and a closed subset $x_a \in Z_A \subset X_A$ of codimension $\geq 2$. Let us denote by $X$, $Z$, and $x$ the base changes to $\Spec \mathbb{C}$. Let us furthermore assume that $X$ is normal. Then, there is a dense open $V \subset \Spec A$ such that for every morphism $\Spec \kay \to V$ with $\kay$ an algebraically closed field of characteristic $p$ there is a canonical isomorphism \[\pi_1(\Spec \mathcal{O}_{X, x}^{\mathrm{sh}} \smallsetminus Z)^{(p)} \cong \pi_1(\Spec \sO_{X_\kay, x_\kay}^{\mathrm{sh}} \smallsetminus Z_{\kay})^{(p)},\] where by abuse of notation we write $Z$ for $\alpha^{-1}(Z)$ where $\alpha: \Spec \sO_{X, x}^\mathrm{sh} \to \Spec \sO_{X}$ is the canonical morphism and similarly for $Z_\kay$.
\end{theorem}
\begin{proof}
Using resolution of singularities, we may choose a truncated proper hypercover $f: Y_{\bullet} \to X$ indexed by $\bullet \in \Delta^{\mathrm{op}}_{\leq 2}$ with $Y_i$ smooth and $D_\bullet \coloneqq f^{-1}(Z)_{\mathrm{red}} \subset Y_\bullet$ giving an SNC divisor at each level. Moreover, by first blowing up $x$ and then $Z$, so that both are Cartier divisors, we have that $E_\bullet \coloneqq f^{-1}(x)_{\mathrm{red}} \subset Y_\bullet$ also yields an SNC divisor at each level.  Denoting by $I_\bullet$ the finite index set of components of $D_\bullet$, each $D_{\bullet, i}$ is smooth over $\mathbb{C}$ and proper over $Z$. Denoting by $J_\bullet$ the subset of $I_\bullet$ that yields the components of $E_\bullet$, we also obtain that the $E_{\bullet, j}$ are smooth and proper varieties over $\mathbb{C}$. We write $U_\bullet \coloneqq Y_\bullet \smallsetminus D_\bullet$. Let $\sY_{\bullet, \ell} \to Y_{\bullet}$ be the $\ell$-th root stack associated to the divisors in $D_\bullet$ and let $\sE_{\bullet, \ell} \to E_\bullet$ be its pullback to $E$. Now, we base change everything along $\alpha: \Spec \mathcal{O}_{X,x}^{\mathrm{sh}} \to \Spec \sO_{X}$ adding a superscript $\mathrm{sh}$ for base changes, \eg $U^{\mathrm{sh}}_\bullet \coloneqq U_\bullet \times_X \Spec \sO_{X,x}^{\mathrm{sh}}$.
The appropriate pullback maps induce equivalences 
\begin{align*}
\FEt\big(\Spec \sO_{X,x}^{\mathrm{sh}} \smallsetminus Z\big) \to \lim_{\bullet \in \Delta_{\leq 2}} \FEt \big(U_{\bullet}^{\mathrm{sh}}\big) \leftarrow \lim_{} \colim_{\ell} \FEt\big(\sY_{\bullet, \ell}^\mathrm{sh}\big) &\to \lim_{\bullet \in \Delta_{\leq 2}} \colim_{\ell} \FEt(\sE_{\bullet, \ell}) \\&\cong \colim_\ell \lim_{\bullet \in \Delta_{\leq 2}} \FEt(\sE_{\bullet, \ell}).
\end{align*}
From left to right, these equivalences are given by \cite[Lemma 2.1]{BhattGabberOlssonFinitenessFun}, \cite[Lemma 2.8 (2)]{BhattGabberOlssonFinitenessFun}, \cite[Lemma 2.2]{BhattGabberOlssonFinitenessFun}, and the isomorphism is due to the fact that filtered colimits commute with finite limits. Having made these minor changes, the rest of the argument now proceeds exactly as \cite[Theorem 4.1]{BhattGabberOlssonFinitenessFun}.
\end{proof}

\begin{proof}[Proof of \autoref{pro.(b)HoldPLT}]
Let us write $Y = \Spec \sO_{X,x}^\mathrm{sh}$. By \autoref{le.b'-condition} (and its proof), it suffices to show that $\pi_1(Y_\mathrm{reg})$ is finite. Using \autoref{le.plttokltperturbation}, we may perturb $\Delta$ to $\Delta'$ so that $(X, \Delta')$ is KLT. The non-regular locus of $Y$ is cut out by a radical ideal $I$ and likewise the closed subset $Z$ is also given by some radical ideal $J$. Passing to a connected \'etale neighborhood $f\: \Spec R' \to \Spec R$ of $\bar{x}$; where $\Spec R$ is some Zariski neighborhood of $x$, we may assume that $I, J \subset R'$. Note that $(\Spec R', f^\ast \Delta')$ is also KLT; see \cite[2.14 (2)]{KollarSingulaitieofMMP}.

Spreading out over some finitely generated $\mathbb{Z}$-algebra $A$ and passing to closed fibers, we obtain pairs $(\Spec R'_s, f^\ast \Delta'_s)$ that are $F$-regular for all $s$ in a dense open of $\Spec A$ (by \cite[Corollary 3.4]{TakagiInterpretationOfMultiplierIdeals}). By the Nullstellensatz applied to the Jacobson ring $A$, $\kappa(s)$ is finite and its the algebraic closure $\kappa(\bar{s})$ is a separable. Hence, $(\Spec R'_{\bar{s}}, f^\ast \Delta'_{\bar{s}})$ is also $F$-regular. Write $Y_{\bar{s}}$ for the spectrum of a strict henselization of $R'_{\bar{s}}$ at $x_{\bar{s}}$ and $W_{\bar{s}}$ for the closed subset defined by $I_{\bar{s}}$. Applying \cite[Theorem 5.1]{CarvajalSchwedeTuckerEtaleFundFsignature} we get $\pi_1(Y_{\bar{s}} \smallsetminus W_{\bar{s}}) \leq 1/s(Y_{\bar{s}})$. Apply \cite[Proposition 6.4, Proposition 5.1]{BhattGabberOlssonFinitenessFun} and \autoref{theo.BGOanalog} to conclude that $\pi_1(X_\mathrm{reg})$ is finite.
\end{proof}

\appendix
\section{{Splitting primes under strict henselizations}}

The goal of this appendix is to show that taking the splitting prime commutes with strict henselization. That is, if $\sp(\sC) = \sp(R,\sC) \subset R$ is the splitting prime for some Cartier algebra $\sC$ acting on $(R, \m, \kay)$ and $f: \Spec R^{\mathrm{sh}} \to \Spec R$ is the strict henselization with respect to $\m$, then $\sp(\sC) {R}^{\mathrm{sh}} = \sp(f^\ast \sC)$. To make sense of this we first need to explain the notion $f^\ast \sC$.

\begin{lemma}
\label{pro.ColimCartierAlgebras}
Let $R$ be a noetherian ring. Consider a colimit over a directed system of ring homomorphisms $\{ f_{ij}\colon S_i\to S_j\}$ of $R$-algebras, a Cartier $R$-algebra $\sC$ and a $\sC$-module $M$. Assume that all of the morphisms $f_{ij}$ and one structural map $R \to S_i$ are either finite, \'etale, or smooth. Then, $\sD = \colim g_{ij}^\ast \sC$ exists and $\sM = \colim g_{ij}^! M$ is naturally a $\sD$-module, where we denote by $g_{ij}\colon \Spec S_j \to \Spec S_j$ the map corresponding to $f_{ij}$.
\end{lemma}
\begin{proof}
If $g\colon \Spec S \to \Spec R$ then by definition $g^\ast \sC = \sC \otimes_R S$ so that we obtain a corresponding directed system of Cartier algebras (cf.\ \cite[Proposition 5.3]{BlickleStablerFunctorialTestModules}). We thus need to verify that the directed system of modules from which we construct $\sM$ are Cartier modules. This is true by \cite[Theorem 5.5]{BlickleStablerFunctorialTestModules}.
\end{proof}

\begin{lemma}
\label{lem.EtaleMorphismSplittingPrime}
Let $f\: \Spec S \to \Spec R$ be a surjective (essentially of finite type) \'etale morphism of $F$-finite local rings. Then $\sp(f^\ast \sC) = \sp(\sC) S$.
\end{lemma}
\begin{proof}
By \cite[Theorem 6.5]{BlickleStablerFunctorialTestModules},\footnote{Since we are only dealing with Cartier modules that do not have non-minimal associated primes, we may use test element theory in the sense of \cite[Theorem 3.11]{BlickleTestIdealsViaAlgebras}---thus we may weaken the assumption that the base is essentially of finite type over an $F$-finite field to $F$-finite.} $R$ is $F$-regular if and only if $S$ is so. Similarly, by \cite[Proposition 5.13, Lemma 6.1]{BlickleStablerFunctorialTestModules}, $R$ is $F$-pure if and only if $S$ is so. Therefore, we may assume that both splitting primes are non-trivial. Consider the following diagram
\[ 
\begin{xy}\xymatrix{\Spec S \ar[r]^f & \Spec R \\ \Spec S/\sp(\sC)S \ar[r]^{f'} \ar[u] & \Spec R/\sp(\sC) \ar[u]} \end{xy}
\]
By \autoref{pro.OurMethodtoProvePureFRegularity}, $R/\sp(\sC)$ is $F$-regular. Since $f'$ is \'etale, $S/\sp(\sC)S$ is also $F$-regular (note that since $f$ is surjective the fiber is non-empty). Since the minimal primes of $\sp(\sC)$ are $f^\ast \sC$-submodules (cf.\ \cite[Corollary 4.8]{SchwedeCentersOfFPurity}), any minimal prime of $\sp(\sC)S$ is a maximal proper $f^\ast \sC$-submodule. However, $\sp(f^\ast \sC)$ is unique and (since $S$ is $F$-pure) the maximal proper $f^\ast \sC$-submodule. Thus, there is only one prime lying over $\sp(\sC)$. Since $R/\sp(\sC)$ is reduced and $f'$ is \'etale, $\sqrt{\sp(\sC)S} = \sp(\sC)S$ is prime and so coincides with $\sp(f^\ast \sC)$.
\end{proof}

\begin{proposition}
\label{pro.splittingprimepreservedsh}
Let $(R, \mathfrak{m}, \kay, K)$ be a normal local domain and $\sC$ be a Cartier $R$-algebra. Denote by $R^{\mathrm{sh}}$ the strict henselization of $(R, \mathfrak{m})$ and by $\sD$ the $R^{\mathrm{sh}}$-Cartier algebra obtained as a colimit over the corresponding system of \'etale algebras. Then, $(R, \sC)$ is $F$-pure if and only if $(R^{\mathrm{sh}}, \sD)$ is so. Moreover, if $\sp(\sC)$ is the splitting prime of $(R, \sC)$ then $\sp(\sC) R^{\mathrm{sh}} = \sp(\sD)$, where $\sp(\sD)$ is the splitting prime of $(R^{\mathrm{sh}}, \sD)$. Conversely, $\sp(\sD) \cap R = \sp(\sC)$.
\end{proposition}

\begin{proof}
Strict henselizations are obtained by a filtered colimit. By \cite[\href{https://stacks.math.columbia.edu/tag/0032}{Lemma 0032}]{stacks-project}, we may even obtain it by a \emph{small} filtered colimit. Moreover, having constructed $R^{\mathrm{sh}}$ via the usual direct limit of triples, we may \emph{a posteriori} also obtain it as a filtered colimit of a system of \'etale maps by viewing everything as embedded in $R^{\mathrm{sh}}$. Using \autoref{pro.ColimCartierAlgebras}, we obtain $\sD$.

If $(R, \sC)$ is $F$-pure, then also is $(R^{\mathrm{sh}}, \sD)$ as well as $(S, \varphi^\ast \sC)$ for any (essentially) \'etale morphism $\varphi\colon \Spec S \to \Spec R$. Indeed, this follows from \cite[Proposition 5.13]{BlickleStablerFunctorialTestModules} in the latter case and the former case follows from the latter. Conversely, if $(R^{\mathrm{sh}}, \sD)$ is not $F$-pure, say $x \notin \sD_+ R^{\mathrm{sh}}$, then we find a surjective \'etale morphism $\varphi: \Spec S \to \Spec R$ for which $x \in S$. Thus $S$ is not $F$-pure but then by faithful flatness $R$ is also not $F$-pure (using \cite[Lemma 6.1]{BlickleStablerFunctorialTestModules}). In particular, if $(R, \sC)$ or $(R^{\mathrm{sh}}, \sD)$ is not $F$-pure, then the statement about splitting primes is trivially true.

Assume that both $(R, \sC)$ and $(R^{\mathrm{sh}}, \sD)$ are $F$-pure.
Let $\varphi\colon \Spec S \to \Spec R$ be an \'etale morphism occurring in the colimit and $\mathfrak{n} \subset S$ a prime above $\mathfrak{m}$. As in the proof of \cite[\href{https://stacks.math.columbia.edu/tag/04GN}{Lemma 04GN}]{stacks-project}, we may assume that $\mathfrak{m}S = \mathfrak{n}$. Since $R^{\mathrm{sh}}$ is local with maximal ideal $\mathfrak{m}R^{\mathrm{sh}}$, the map $S \to R^{\mathrm{sh}}$ factors through the localization $S_\mathfrak{n}$. Thus, $\varphi'\colon \Spec S_{\mathfrak{n}}\to \Spec R$ is an essentially \'etale surjective homomorphism. We apply \autoref{lem.EtaleMorphismSplittingPrime} to conclude that the splitting prime $\sp(\varphi'^!\sC)$ of $S_{\mathfrak{n}}$ is $\sp(\sC) S_{\mathfrak{n}}$. Next, note that any homogeneous element of $\varphi'^! \sC$ is of the form $\kappa \otimes s^{q}$ with $\kappa \in \sC_e$, which acts on $x = r \otimes t \in R\otimes_R S_{\mathfrak{n}}$ as $\kappa \otimes s^{q}(r \otimes t) = \kappa(r) \otimes st$; see \cite[Theorem 5.5]{BlickleStablerFunctorialTestModules}. Use now the well-known isomorphism $F^e_\ast R \otimes_R S_{\mathfrak{n}} \to F^e_\ast S_{\mathfrak{n}}$, $r \otimes s \mapsto rs^{q}$.

We now prove $\sp(\sC) R^{\mathrm{sh}}\subset \sp(\sD)$. If $x \in \sp(\sC)R^{\mathrm{sh}}$ and $\kappa \in \sD_e$, there is an essentially \'etale morphism  $\varphi'\: \Spec S_{\mathfrak{n}}\to \Spec R$ as above so that $x \in \sp(\sC) S_{\mathfrak{n}} = \sp(\varphi'^! \sC)$ and $\kappa 
\in \varphi'^! \sC$. Then, since $x \in \sp(\varphi'^! \sC)$, we have $\kappa(x) \in \mathfrak{n} = \mathfrak{m}S \subset \mathfrak{m} R^{\mathrm{sh}}$ and so $x \in \sp(\sD)$. Conversely, let $x \in \sp(\sD)$. Then, we find $\varphi'\: \Spec S \to \Spec R$ as above such that $x \in S$. Since $\kappa(x) \in \mathfrak{m}R^{\mathrm{sh}}$ for all $\kappa \in \sD$, we also have $\kappa(x) \in \mathfrak{m}R^{\mathrm{sh}} \cap S = \mathfrak{n}$ for all $\kappa \in \varphi'^!(\sC)$. Hence, $x \in \sp(\varphi'^! \sC) = \sp(\sC) S \subset \sp(\sC)R^{\mathrm{sh}}$ as desired.

We now show $\sp(\sD) \cap R = \sp(\sC)$. If $x \in \sp(\sD) \cap R$, for all $\kappa \in \sD$ we have $\kappa(x) \in \mathfrak{m}R^{\mathrm{sh}}$ and so $\kappa(x) \in \mathfrak{m}$ for all $\kappa \in \sC$. If $x \in \sp(\sC)$ then $x \in \sp(\sC) R^{\mathrm{sh}} = \sp(\sD)$ by the above.
\end{proof}

\bibliographystyle{skalpha}
\bibliography{MainBib}

\def\cfudot#1{\ifmmode\setbox7\hbox{$\accent"5E#1$}\else
  \setbox7\hbox{\accent"5E#1}\penalty 10000\relax\fi\raise 1\ht7
  \hbox{\raise.1ex\hbox to 1\wd7{\hss.\hss}}\penalty 10000 \hskip-1\wd7\penalty
  10000\box7}
\providecommand{\bysame}{\leavevmode\hbox to3em{\hrulefill}\thinspace}
\providecommand{\MR}{\relax\ifhmode\unskip\space\fi MR}
\providecommand{\MRhref}[2]{%
  \href{http://www.ams.org/mathscinet-getitem?mr=#1}{#2}
}
\providecommand{\href}[2]{#2}
\begin{thebibliography}{CEPT96}

\bibitem[Amb99]{AmbroThesis}
{\sc F.~Ambro}: \emph{The adjunction conjecture and its applications}, ProQuest
  LLC, Ann Arbor, MI, 1999, Thesis (Ph.D.)--The Johns Hopkins University.
  {\sf\scriptsize 2698988}

\bibitem[AM69]{AtiyahMacdonald}
{\sc M.~F. Atiyah and I.~G. Macdonald}: \emph{Introduction to commutative
  algebra}, Addison-Wesley Publishing Co., Reading, Mass.-London-Don Mills,
  Ont., 1969. {\sf\scriptsize MR0242802 (39 \#4129)}

\bibitem[Aus62]{AuslanderPurity}
{\sc M.~Auslander}: \emph{On the purity of the branch locus}, Amer. J. Math.
  \textbf{84} (1962), 116--125. {\sf\scriptsize 0137733}

\bibitem[BGO17]{BhattGabberOlssonFinitenessFun}
{\sc B.~{Bhatt}, O.~{Gabber}, and M.~{Olsson}}: \emph{{Finiteness of \'etale
  fundamental groups by reduction modulo $p$}}, ArXiv e-prints (2017).

\bibitem[BF84]{BingenerFlennerClassGroup}
{\sc J.~Bingener and H.~Flenner}: \emph{Variation of the divisor class group},
  J. Reine Angew. Math. \textbf{351} (1984), 20--41. {\sf\scriptsize 749675}

\bibitem[Bli13]{BlickleTestIdealsViaAlgebras}
{\sc M.~Blickle}: \emph{Test ideals via algebras of {$p^{-e}$}-linear maps}, J.
  Algebraic Geom. \textbf{22} (2013), no.~1, 49--83. {\sf\scriptsize 2993047}

\bibitem[BST12]{BlickleSchwedeTuckerFSigPairs1}
{\sc M.~Blickle, K.~Schwede, and K.~Tucker}: \emph{{$F$}-signature of pairs and
  the asymptotic behavior of {F}robenius splittings}, Adv. Math. \textbf{231}
  (2012), no.~6, 3232--3258. {\sf\scriptsize 2980498}

\bibitem[BS19]{BlickleStablerFunctorialTestModules}
{\sc M.~Blickle and A.~St\"{a}bler}: \emph{Functorial test modules}, J. Pure
  Appl. Algebra \textbf{223} (2019), no.~4, 1766--1800. {\sf\scriptsize
  3906525}

\bibitem[BV88]{BrunsVetterDeterminantalRings}
{\sc W.~Bruns and U.~Vetter}: \emph{Determinantal rings}, Lecture Notes in
  Mathematics, vol. 1327, Springer-Verlag, Berlin, 1988. {\sf\scriptsize
  953963}

\bibitem[Cad13]{CadoretGaloisCategories}
{\sc A.~Cadoret}: \emph{Galois categories}, Arithmetic and geometry around
  {G}alois theory, Progr. Math., vol. 304, Birkh\"{a}user/Springer, Basel,
  2013, pp.~171--246. {\sf\scriptsize 3408165}

\bibitem[CRST18]{CarvajalSchwedeTuckerEtaleFundFsignature}
{\sc J.~Carvajal-Rojas, K.~Schwede, and K.~Tucker}: \emph{Fundamental groups of
  {$F$}-regular singularities via {$F$}-signature}, Ann. Sci. \'{E}c. Norm.
  Sup\'{e}r. (4) \textbf{51} (2018), no.~4, 993--1016. {\sf\scriptsize 3861567}

\bibitem[CRST21]{CarvajalSchwedeTuckerBertiniFSignature}
{\sc J.~Carvajal-Rojas, K.~Schwede, and K.~Tucker}: \emph{Bertini theorems for
  {$F$}-signature and {H}ilbert--{K}unz multiplicity}, Math. Z. \textbf{299}
  (2021), no.~1-2, 1131--1153. {\sf\scriptsize 4311632}

\bibitem[CRS20]{CarvajalSmolkinUSTPandDFR}
{\sc J.~Carvajal-Rojas and D.~Smolkin}: \emph{The uniform symbolic topology
  property for diagonally {$F$}-regular algebras}, J. Algebra \textbf{548}
  (2020), 25--52. {\sf\scriptsize 4044703}

\bibitem[CS19]{CarvajalStablerFsignaturefinitemorphisms}
{\sc J.~{Carvajal-Rojas} and A.~{St{\"a}bler}}: \emph{{On the behavior of
  {$F$}-signatures, splitting primes, and test modules under finite covers}},
  arXiv e-prints (2019), arXiv:1904.10382.

\bibitem[CSK20]{CRStaeblerKollarFundamentalGroups}
{\sc J.~{Carvajal-Rojas}, A.~{St{\"a}bler}, and J.~{Koll{\'a}r}}: \emph{{On the
  local {\'e}tale fundamental group of KLT threefold singularities}}, April
  2020.

\bibitem[CR18]{Carvajalphdthesis}
{\sc J.~A. Carvajal-Rojas}: \emph{Arithmetic {A}spects of {S}trong
  {F}-{R}egularity}, Ph.D. thesis, 2018, p.~114, Thesis (Ph.D.)--The University
  of Utah. {\sf\scriptsize 4082619}

\bibitem[CR22]{CarvajalFiniteTorsors}
{\sc J.~A. Carvajal-Rojas}: \emph{Finite torsors over strongly {$F$}-regular
  singularities}, \'{E}pijournal G\'{e}om. Alg\'{e}brique \textbf{6} (2022),
  Aert. 1--30. {\sf\scriptsize 4391081}

\bibitem[CEPT96]{ChinburgErezPappasTaylorTameActions}
{\sc T.~Chinburg, B.~Erez, G.~Pappas, and M.~J. Taylor}: \emph{Tame actions of
  group schemes: integrals and slices}, Duke Math. J. \textbf{82} (1996),
  no.~2, 269--308. {\sf\scriptsize 1387229}

\bibitem[Cut95]{CutkoskyPurity}
{\sc S.~D. Cutkosky}: \emph{Purity of the branch locus and {L}efschetz
  theorems}, Compositio Math. \textbf{96} (1995), no.~2, 173--195.
  {\sf\scriptsize 1326711}

\bibitem[Fed83]{FedderFPureRat}
{\sc R.~Fedder}: \emph{{$F$}-purity and rational singularity}, Trans. Amer.
  Math. Soc. \textbf{278} (1983), no.~2, 461--480. {\sf\scriptsize MR701505
  (84h:13031)}

\bibitem[FG12]{FujinoGongyoCanonicalBundleFormulasAndSubadjunction}
{\sc O.~Fujino and Y.~Gongyo}: \emph{On canonical bundle formulas and
  subadjunctions}, Michigan Math. J. \textbf{61} (2012), no.~2, 255--264.
  {\sf\scriptsize 2944479}

\bibitem[Gro63]{GrothendieckSGA}
{\sc A.~Grothendieck}: \emph{Rev\^etements \'etales et groupe fondamental.
  {F}asc. {II}: {E}xpos\'es 6, 8 \`a 11}, S\'eminaire de G\'eom\'etrie
  Alg\'ebrique, vol. 1960/61, Institut des Hautes \'Etudes Scientifiques,
  Paris, 1963. {\sf\scriptsize MR0217088 (36 \#179b)}

\bibitem[GM71]{GrothendieckMurreTameFundamentalGroup}
{\sc A.~Grothendieck and J.~P. Murre}: \emph{The tame fundamental group of a
  formal neighbourhood of a divisor with normal crossings on a scheme}, Lecture
  Notes in Mathematics, Vol. 208, Springer-Verlag, Berlin, 1971.
  {\sf\scriptsize MR0316453 (47 \#5000)}

\bibitem[HW02]{HaraWatanabeFRegFPure}
{\sc N.~Hara and K.-I. Watanabe}: \emph{F-regular and {F}-pure rings vs. log
  terminal and log canonical singularities}, J. Algebraic Geom. \textbf{11}
  (2002), no.~2, 363--392. {\sf\scriptsize MR1874118 (2002k:13009)}

\bibitem[Har77]{Hartshorne}
{\sc R.~Hartshorne}: \emph{Algebraic geometry}, Springer-Verlag, New York,
  1977, Graduate Texts in Mathematics, No. 52. {\sf\scriptsize MR0463157 (57
  \#3116)}

\bibitem[JS22]{JeffriesSmirnovTransformationRule}
{\sc J.~Jeffries and I.~Smirnov}: \emph{A {T}ransformation {R}ule for {N}atural
  {M}ultiplicities}, Int. Math. Res. Not. IMRN (2022), no.~2, 999--1015.
  {\sf\scriptsize 4368877}

\bibitem[KSSZ14]{KatzmanSchwedeSinghZhangRingsFrobOperators}
{\sc M.~Katzman, K.~Schwede, A.~K. Singh, and W.~Zhang}: \emph{Rings of
  {F}robenius operators}, Math. Proc. Cambridge Philos. Soc. \textbf{157}
  (2014), no.~1, 151--167. {\sf\scriptsize 3211813}

\bibitem[KS10]{KerzSchmidtOnDifferentNotionsOfTameness}
{\sc M.~Kerz and A.~Schmidt}: \emph{On different notions of tameness in
  arithmetic geometry}, Math. Ann. \textbf{346} (2010), no.~3, 641--668.
  {\sf\scriptsize MR2578565}

\bibitem[Kol13]{KollarSingulaitieofMMP}
{\sc J.~Koll\'{a}r}: \emph{Singularities of the minimal model program},
  Cambridge Tracts in Mathematics, vol. 200, Cambridge University Press,
  Cambridge, 2013, With a collaboration of S\'{a}ndor Kov\'{a}cs.
  {\sf\scriptsize 3057950}

\bibitem[KM98]{KollarMori}
{\sc J.~Koll{\'a}r and S.~Mori}: \emph{Birational geometry of algebraic
  varieties}, Cambridge Tracts in Mathematics, vol. 134, Cambridge University
  Press, Cambridge, 1998, With the collaboration of C. H. Clemens and A. Corti,
  Translated from the 1998 Japanese original. {\sf\scriptsize MR1658959
  (2000b:14018)}

\bibitem[Kun13]{KunzIntroToACAG}
{\sc E.~Kunz}: \emph{Introduction to commutative algebra and algebraic
  geometry}, Modern Birkh\"{a}user Classics, Birkh\"{a}user/Springer, New York,
  2013, Translated from the 1980 German original [MR0562105] by Michael
  Ackerman, With a preface by David Mumford, Reprint of the 1985 edition
  [MR0789602]. {\sf\scriptsize 2977456}

\bibitem[Lan02]{LangAlgebra}
{\sc S.~Lang}: \emph{Algebra}, third ed., Graduate Texts in Mathematics, vol.
  211, Springer-Verlag, New York, 2002. {\sf\scriptsize 1878556}

\bibitem[Mar22]{MArtinNumberOfTorsionDivisorsinSFRrings}
{\sc I.~Martin}: \emph{The number of torsion divisors in a strongly
  {$F$}-regular ring is bounded by the reciprocal of {$F$}-signature},
  Communications in Algebra \textbf{50} (2022), no.~4, 1595--1605.

\bibitem[Mat80]{MatsumuraCommutativeAlgebra}
{\sc H.~Matsumura}: \emph{Commutative algebra}, second ed., Mathematics Lecture
  Note Series, vol.~56, Benjamin/Cummings Publishing Co., Inc., Reading, Mass.,
  1980. {\sf\scriptsize MR575344 (82i:13003)}

\bibitem[Mil80]{MilneEtaleCohomology}
{\sc J.~S. Milne}: \emph{\'{E}tale cohomology}, Princeton Mathematical Series,
  vol.~33, Princeton University Press, Princeton, N.J., 1980. {\sf\scriptsize
  559531}

\bibitem[Mur67]{MurreLecturesFundamentalGroups}
{\sc J.~P. Murre}: \emph{Lectures on an introduction to {G}rothendieck's theory
  of the fundamental group}, Tata Institute of Fundamental Research, Bombay,
  1967, Notes by S. Anantharaman, Tata Institute of Fundamental Research
  Lectures on Mathematics, No 40. {\sf\scriptsize 0302650}

\bibitem[Nag58]{NagataPurity}
{\sc M.~Nagata}: \emph{Remarks on a paper of {Z}ariski on the purity of branch
  loci}, Proc. Nat. Acad. Sci. U.S.A. \textbf{44} (1958), 796--799.
  {\sf\scriptsize 0095847}

\bibitem[Nag59]{NagataPurityII}
{\sc M.~Nagata}: \emph{On the purity of branch loci in regular local rings},
  Illinois J. Math. \textbf{3} (1959), 328--333. {\sf\scriptsize 0106930}

\bibitem[Pol22]{PolstraATheoremAboutMCMM}
{\sc T.~Polstra}: \emph{A {T}heorem {A}bout {M}aximal {C}ohen--{M}acaulay
  {M}odules}, Int. Math. Res. Not. IMRN (2022), no.~3, 2086--2094.
  {\sf\scriptsize 4373232}

\bibitem[Rom06]{RomanFieldTheory}
{\sc S.~Roman}: \emph{Field theory}, second ed., Graduate Texts in Mathematics,
  vol. 158, Springer, New York, 2006. {\sf\scriptsize 2178351}

\bibitem[Sch09]{SchwedeFAdjunction}
{\sc K.~Schwede}: \emph{{$F$}-adjunction}, Algebra Number Theory \textbf{3}
  (2009), no.~8, 907--950. {\sf\scriptsize 2587408 (2011b:14006)}

\bibitem[Sch10]{SchwedeCentersOfFPurity}
{\sc K.~Schwede}: \emph{Centers of {$F$}-purity}, Math. Z. \textbf{265} (2010),
  no.~3, 687--714. {\sf\scriptsize 2644316 (2011e:13011)}

\bibitem[ST14]{SchwedeTuckerTestIdealFiniteMaps}
{\sc K.~Schwede and K.~Tucker}: \emph{On the behavior of test ideals under
  finite morphisms}, J. Algebraic Geom. \textbf{23} (2014), no.~3, 399--443.
  {\sf\scriptsize 3205587}

\bibitem[Ser79]{SerreLocalFields}
{\sc J.-P. Serre}: \emph{Local fields}, Graduate Texts in Mathematics, vol.~67,
  Springer-Verlag, New York, 1979, Translated from the French by Marvin Jay
  Greenberg. {\sf\scriptsize MR554237 (82e:12016)}

\bibitem[Sin05]{SinghFSignatureOfAffineSemigroup}
{\sc A.~K. Singh}: \emph{The {$F$}-signature of an affine semigroup ring}, J.
  Pure Appl. Algebra \textbf{196} (2005), no.~2-3, 313--321. {\sf\scriptsize
  MR2110527 (2005m:13010)}

\bibitem[SS07]{SinghSpiroffDivisorClassGroups}
{\sc A.~K. Singh and S.~Spiroff}: \emph{Divisor class groups of graded
  hypersurfaces}, Algebra, geometry and their interactions, Contemp. Math.,
  vol. 448, Amer. Math. Soc., Providence, RI, 2007, pp.~237--243.
  {\sf\scriptsize 2389245}

\bibitem[Smo19]{Smolkinphdthesis}
{\sc D.~Smolkin}: \emph{Subadditivity of test ideals and diagonal
  ${F}$-regularity}, Ph.D. thesis, University of Utah, ProQuest LLC, Ann Arbor,
  MI, 2019, Thesis (Ph.D.).

\bibitem[Smo20]{SmolkinSubadditivity}
{\sc D.~Smolkin}: \emph{A new subadditivity formula for test ideals}, J. Pure
  Appl. Algebra \textbf{224} (2020), no.~3, 1132--1172. {\sf\scriptsize
  4009572}

\bibitem[Spe20]{SpeyerFrobeniusSplit}
{\sc D.~E. Speyer}: \emph{Frobenius split subvarieties pull back in almost all
  characteristics}, J. Commut. Algebra \textbf{12} (2020), no.~4, 573--579.
  {\sf\scriptsize 4194942}

\bibitem[St{\"a}17]{staeblerunitftestmodules}
{\sc A.~St{\"a}bler}: \emph{Test module filtrations for unit {$F$}-modules}, J.
  Algebra \textbf{477} (2017), 435--471. {\sf\scriptsize 3614158}

\bibitem[Tak04]{TakagiInterpretationOfMultiplierIdeals}
{\sc S.~Takagi}: \emph{An interpretation of multiplier ideals via tight
  closure}, J. Algebraic Geom. \textbf{13} (2004), no.~2, 393--415.
  {\sf\scriptsize MR2047704 (2005c:13002)}

\bibitem[Tak08]{TakagiPLTAdjoint}
{\sc S.~Takagi}: \emph{A characteristic {$p$} analogue of plt singularities and
  adjoint ideals}, Math. Z. \textbf{259} (2008), no.~2, 321--341.
  {\sf\scriptsize MR2390084 (2009b:13004)}

\bibitem[Tak10]{TakagiHigherDimensionalAdjoint}
{\sc S.~Takagi}: \emph{Adjoint ideals along closed subvarieties of higher
  codimension}, J. Reine Angew. Math. \textbf{641} (2010), 145--162.
  {\sf\scriptsize 2643928 (2011f:14032)}

\bibitem[{Tay}19]{TaylorInversionOfAdjuntionFSignature}
{\sc G.~{Taylor}}: \emph{{Inversion of adjunction for $F$-signature}}, arXiv
  e-prints (2019), arXiv:1909.10436.

\bibitem[{The}20]{stacks-project}
{\sc {The Stacks Project Authors}}: \emph{\textit{Stacks Project}}, 2020.

\bibitem[TW92]{TomariWatanabeNormalZrGradedRings}
{\sc M.~Tomari and K.~Watanabe}: \emph{Normal {$Z_r$}-graded rings and normal
  cyclic covers}, Manuscripta Math. \textbf{76} (1992), no.~3-4, 325--340.
  {\sf\scriptsize 1185023 (93j:13002)}

\bibitem[Xu14]{XuFinitenessOfFundGroups}
{\sc C.~Xu}: \emph{Finiteness of algebraic fundamental groups}, Compos. Math.
  \textbf{150} (2014), no.~3, 409--414. {\sf\scriptsize 3187625}

\bibitem[Yao06]{YaoObservationsAboutTheFSignature}
{\sc Y.~Yao}: \emph{Observations on the {$F$}-signature of local rings of
  characteristic {$p$}}, J. Algebra \textbf{299} (2006), no.~1, 198--218.
  {\sf\scriptsize MR2225772 (2007k:13007)}

\bibitem[Zar58]{ZariskiPurity}
{\sc O.~Zariski}: \emph{On the purity of the branch locus of algebraic
  functions}, Proc. Nat. Acad. Sci. U.S.A. \textbf{44} (1958), 791--796.
  {\sf\scriptsize 0095846}

\end{thebibliography}

\end{document}